\documentclass[12pt]{article}
\usepackage[margin=1in]{geometry} 
\usepackage{amsmath,amsthm,amssymb}
\usepackage{indentfirst}
\usepackage{cleveref}
\usepackage{dsfont}
	\newcommand{\1}{\mathds{1}}
\usepackage{enumitem}

\usepackage{tikz-cd}

\newtheorem{thm}{Theorem}[section]
\newtheorem{lemma}[thm]{Lemma}
\newtheorem{prop}[thm]{Proposition}
\newtheorem{corollary}[thm]{Corollary}
\newtheorem{defn}[thm]{Definition}
\newtheorem{remark}[thm]{Remark}

\numberwithin{equation}{section}

\newcommand{\N}{\mathbb{N}}
\newcommand{\Z}{\mathbb{Z}}
\newcommand{\R}{\mathbb{R}}
\renewcommand{\P}{\mathbb P}
\newcommand{\C}{\mathbb C}
\newcommand{\EX}{\mathbb E}

\newcommand{\DD}{\mathbb D}

\newcommand{\<}{\leq}
\renewcommand{\>}{\geq}

\def\sE{\mathcal E}
\def\sF{\mathcal F}
\def\sL{\mathcal L}
\def\sN{\mathcal N}
\def\sX{\mathcal X}

\newcommand{\2}{\alpha}
\newcommand{\3}{\beta}
\renewcommand{\r}{\gamma}
\newcommand{\e}{\varepsilon}
\newcommand{\8}{\infty}
\newcommand{\6}{\partial}
\renewcommand{\-}{\setminus}
\newcommand{\w}{\omega}
\renewcommand{\d}{\delta}
\renewcommand{\.}{\cdot}
\renewcommand{\k}{\kappa}
\renewcommand{\t}{\tau}

\newcommand{\inn}{\subset}

\renewcommand{\~}{\widetilde}
\renewcommand{\limsup}{\varlimsup}
\renewcommand{\liminf}{\varliminf}

\newcommand{\jd}{\varrho} 
\newcommand{\leb}{m_0}
\newcommand{\Rd}{\R^d}

\renewcommand{\v}{\check}

\begin{document}

\title{Discrete approximation to time-changed Brownian motions}
\author{Zhen-Qing Chen \ and \ Yang Yu}
\date{}

\maketitle

\begin{abstract}
	We develop a general discrete approximation scheme for time-changed Brownian motions on $\R^d$.
	Our approximation scheme works for any smooth measure with full quasi-support on $\Rd$ with suitable initial distributions. 
	Under some mild conditions on the smooth measure, the discrete approximation scheme works for every starting point. Our results in particular give a discrete approximation scheme for Liouville Brownian motions.
\end{abstract}
\begingroup
\renewcommand{\thefootnote}{}
\endgroup

\tableofcontents

\section{Introduction and Main Results} \label{section intro}

Time change is an important transformation for Markov processes and for Brownian motion in particular. 
It has been extensively studied in the literature; see, e.g., \cite{BG, CF12, FOT} and the references therein. 
Let $d\>1$ be a positive integer and $W$ be the standard Brownian motion in $\R^d$
whose infinitesimal generator is $\frac 12 \Delta$.
Suppose that $\rho$ is a strictly positive, locally bounded Borel measurable function on $\R^d$.
It gives rise to a positive continuous additive functional (PCAF) $A:=\{A_t; t\geq 0\}$ of $W$ defined by $A_t:=\int_0^t 
\rho (W_s)ds$,
which has the property that $A_{t+s} = A_t + A_s \circ \theta_t$ for any $s, t\geq 0$. Here
$\{\theta_s; s\geq 0\}$  are  the time-shift operators so that $W_t (\theta_s \omega) = W_{t+s} (\omega)$. 
Let 
\begin{equation}\label{e:1.1a} 
\tau_t:=\inf\{r\geq 0: A_r >t\}
\end{equation}
 be the generalized inverse of $A$. 
It is known that the time-changed process
 $X_t:=W_{\tau_t}$ is a continuous strong Markov process on $\R^d$
that is symmetric with respect to the measure $\mu(dx):=\rho (x) dx$ and  has infinitesimal generator 
\begin{equation}\label{e:1.1}
\sL = \frac1{2\rho (x)} \Delta . 
\end{equation} The measure $\mu$ is called the Revuz measure of the PCAF $A$.
They are related by
\begin{equation} \label{e:1.2}
\lim_{t\downarrow0} \frac1t \int_{\R^d} \EX_x \left[ \int_0^t f(W_s) dA_s \right] h(x) dx = \int_{\R^d} h(x) f(x) 
\mu (dx) 
\end{equation}
for any Borel measurable $f\geq 0$ on $\R^d$ and any $\gamma$-excessive function $h$ of $W$ for some $\gamma\geq 0$.
In fact, there is a one-to-one correspondence between PCAFs and a family of $\sigma$-finite measures on $\R^d$ that do not
charge polar sets, called smooth measures, 
through the Revuz correspondence \eqref{e:1.2}. 
In general, smooth measures on $\R^d$ do not need to be absolutely continuous with respect to the Lebesgue measure on $\R^d$. 

Let $\mu$ be a smooth measure   and $A$ its corresponding PCAF of $W$. 
Let 
$F:=\{x\in \R^d: \P_x (\tau_0=0)=1\}$, where $\tau_t$ is defined by \eqref{e:1.1a}.
The set $F$ is called the support of the PCAF $A$, which is also called the quasi-support of $\mu$.
In the rest of this paper, $\mu$ is a Radon measure that has full quasi-support on $\R^d$. 
In this case the PCAF $A_t$ is strictly increasing in $t$ and so its inverse 
$\tau_t:=\inf\{r\geq 0: A_r >t\}$ is continuous in $t\geq 0$. 
The time-changed Brownian motion $X_t:=W_{\tau_t}$  is a continuous strong Markov process on $\R^d$.
Heuristically, the infinitesimal generator of $X$ should be ``conformal'' to the Laplacian but it cannot be explicitly expressed
in the form of \eqref{e:1.1} if $\mu$ is not absolutely continuous with respect to the Lebesgue measure. 
Instead, it should be characterized through the Dirichlet form associated with the time-changed Brownian motion. 
See \Cref{potential theory} below for a brief introduction of these concepts, and we refer the reader 
to \cite{FOT, CF12} for more information.

 The goal of this paper is to develop a scheme for constructing and simulating the time-changed Brownian motion $X$ using only $\mu$ as input, without requiring Brownian sample paths or inverses of PCAFs.
It is well known that Brownian motion is the scaling limit of simple random walks. A natural idea is to change the holding times or waiting times of random walks according to the mass of $\mu$ at that place in a suitable way. We show that such a scheme indeed works. 
In fact, a more general scheme is developed in this paper.

For  any $r>0$, we consider a continuous time random walk process $\~W^r$ approximating the standard Brownian motion $W$, and a measure $\mu_r$ that converges to $\mu$ vaguely as $r\downarrow 0$. We do time change on $\~W^{r}$ to get a time-changed random walk process $X^r$. A natural question is when $X^r$ converges to the time-changed Brownian motion $X$ as $r\downarrow 0$.

We construct $X^{r}$ as follows. Let $\jd$ be a probability density function (with respect to the Lebesgue measure) that has zero mean and covariance $\sigma_{\jd}^2 \d_{ij}$ for some constant $\sigma_{\jd}>0$, where $\d_{ij}$ is the Kronecker delta. 
It will serve as
the  one-step  distribution of the random walk.
A typical example of $\jd$ is $\jd(x) = \frac{\1_{B(0,1)}(x)}{|B(0,1)|}$.
Here for $x_0\in \R^d$ and $r>0$, $B(x_0,r):=\{y\in\R^d:|y-x_0|<r\}$ denotes the open ball
centered at $x_0$ with radius $r$,   and $|B(x_0,r)|$ denotes its Lebesgue volume.  
Let $\phi$ be a probability density function (with respect to the Lebesgue measure) which 
will be used  as the mollifier for the measure $\mu$. Denote by $\v f(x) = f(-x)$ the reflection of any function $f$ on $\Rd$
with respect to the origin. 
For each $r>0$, set 
$\phi_r(x):=r^{-d}\phi(x/r)$,  $\jd_r(x):=r^{-d}\jd(x/r)$, 
\begin{eqnarray}
\mu_r(dx)&:=& (\v\phi_r*\mu)(dx)=\int \phi_r(y-x)\mu(dy)dx,  \nonumber \\
 m_r(dx) &:=& (\sigma_{\jd}^2r^2)^{-1}dx ,  \nonumber  \\ 
\lambda_r(x) &:=& \frac{dm_r}{d\mu_r}(x)=\left(\sigma_{\jd}^2r^2 \int \phi_r(y-x)\mu(dy)\right)^{-1},
\label{e:1.4}
\end{eqnarray} 
  and $Q_r(x,dy) := \jd_r(y-x)dy$.
We construct a  time-changed random walk process $X^{r}$ that waits for an exponentially distributed time with rate $\lambda_r$ and then moves according to the
 transition probability kernel $Q_r$. More specifically,
let $\{\xi_i\}_{i=1}^\8$ be i.i.d. $\R^d$-valued random variables with probability density function $\jd$.
Let $\{\eta_i\}_{i=0}^\8$ be i.i.d. exponentially distributed with parameter 1 that are independent of $\{\xi_i\}_{i=1}^\infty$.
 Let $\xi_0$ be an $\R^d$-valued random variable  independent of $\{\xi_i\}_{i=1}^\8$ and $\{\eta_i\}_{i=0}^\8$.
For each $r>0$,  set 
\begin{equation}\label{e:1.5}
W^r_n := \xi_0+ \sum_{i=1}^n r \xi_i \quad \hbox{for } n\in \N.
\end{equation}
We  use the convention that $\sum_{i=1}^n r \xi_i = 0$ if $n=0$.
Noting that  $r\xi_i$ has distribution $\jd_r$, 
 $\{W^r_n; n\geq 0\}$ is a random walk on $\R^d$ with initial distribution $\xi_0$ and step distribution $\jd_r$.
Define $\eta_i^r(x):=\eta_i/\lambda_r(x)$ (which is exponentially distributed with mean  $1/\lambda_r(x)$) and 
$N^r_t := \inf\{n\in\N:\sum_{i=0}^n\eta^r_i(W^r_i)>t\}$
with the convention $\inf \emptyset := \infty$.   Define $X^r_t:=W^r_{N^r_t}$ if $N^r_t <\infty$ and $X^r_t:=\partial$ if $N^r_t =\infty$, 
where $\partial$ is a cemetery point added to $\R^d$ as a one-point compactification. 
The process $X^r$ is the so-called regular step process with road map $Q_r$ and speed function $\lambda_r$ (see, e.g.,  \cite[Subsection 2.2.1]{CF12}). When $\mu(dx)$ is the Lebesgue measure on $\R^d$, 
 we denote the corresponding $X^r$ by $\~ W^r$. Note that  $\~W^r$ is simply the continuous time random walk with step distribution $\jd_r$ and jumping rate $(\sigma_{\jd}^{2} r^{2})^{-1}$ or mean holding time $\sigma_{\jd}^{2} r^{2}$. 

It is natural to speculate that the time-changed continuous time random walk $X^r$ converges 
in distribution to the time-changed Brownian motion $X$ with initial distribution of $\xi_0$
as $r\downarrow 0$. Under some mild assumptions on $\jd$, $\mu$ and $\phi$, we show that this is indeed the case.
When $\mu(dx)$ is the Lebesgue measure on $\R^d$, 
we have $\lambda_r(x)=(\sigma_{\jd}^{2} r^2)^{-1}$ for any $x\in \R^d$ and $r>0$.
In this case, it is well known that
 $X^r$ converges to the standard Brownian motion starting from $\xi_0$ as $r\downarrow0$.
  For a general  smooth measure $\mu$ on $\R^d$ with full quasi-support, 
  we will show in this paper that 
   $X^r$  converges to the corresponding time-changed Brownian motion with respect to the measure $\mu$ as $r\downarrow0$. 
   Here is a diagram that illustrates our approximation scheme.

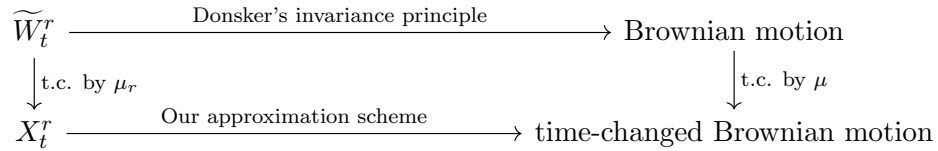
\begin{center}
    
    \begin{tikzcd}[column sep=6cm]
        \~ W_{t}^{r}
        \arrow[r, "\text{Donsker's invariance principle}"] 
        \arrow[d, "\text{t.c. by } \mu_{r}" ]
        & \text{Brownian motion} \arrow[d, "\text{t.c. by } \mu"] \\
        X_{t}^{r} \arrow[r, "\text{Our approximation scheme}" ]
        & \text{time-changed Brownian motion}
    \end{tikzcd}
     \end{center}

\medskip

Throughout this paper, we impose the following two basic assumptions on the probability densities $\jd$ and the 
measure $\mu$:

\begin{enumerate}
\item[(a)] $\jd (x)=\jd(-x)$   on $ \R^d$  with  $ \int_{\R^d}  x_{i} x_{j} \jd(x) dx = \sigma_\jd^2  \d_{ij} $
	  for  $ 1\leq i, j\leq d$.

\item[(b)] $\mu$  is a Radon measure that has full quasi-support on $\R^d$ and does not charge polar sets. 
\end{enumerate} 
\
Note that Assumption (a) implies that $\jd$ has zero mean and $\~W^{r}$ converges to a 
Brownian motion as $r\downarrow 0$ 
by Donsker's invariance principle. Assumption (b) allows us to construct a \textit{continuous} time change of Brownian motion according to the measure $\mu$.
Without the full quasi-support condition, time changes may not be continuous and the time-changed Brownian motions may have jumps.

\medskip

We list below some other assumptions on the probability densities $\jd$, $\phi$ and measure $\mu$
that will be used later in some parts of the paper.
\begin{enumerate}[label=(A.\arabic*),ref=(A.\arabic*)]
	\item \label{jd0} 
 	$\phi$ is bounded and  has compact support
	and first order weak derivative (in the sense of Schwartz distributions) 
	such that $|\nabla \phi| \< C\jd$ on $\Rd$ for some constant $C>0$.   
	
	\item \label{jd8} 
	$\jd$ is bounded above by a constant $C_{\jd}>0$.
    Moreover there exists $\d>0$ such that $\jd$ has a finite $\max\{ 2,d-2+\d\}$-th absolute moment. 

	\item \label{condA}
	For any $R > 0$, there exist constants  $C_{R}>0$ and $\3_{\mu} = \3_{\mu}(R) > d - 2$
	 so that for any $x\in B(0,R)$ and $r\in(0,1)$ we have 
 	\begin{equation}\label{measure upper bound condition}
		\mu(B(x,r))\<C_R r^{\3_{\mu}}, 
	\end{equation}
	and
	\begin{equation} \label{measure lower bound condition}
		\lim_{r\downarrow0}r^{\3_{\mu}-d+2}\log\left(\int_{B(0,2R)}\frac{dx}{\int \phi((y-x)/r)\mu(dy)}\right)=0.
	\end{equation}
 \end{enumerate}

\begin{remark} \label{R:1.1} \rm 
\begin{enumerate}  
\item [(i)]
Under \ref{jd0}, for each $r>0$, the function $x\mapsto\int\phi_r(y-x)\mu(dy)$ is finite everywhere and strictly positive Lebesgue-a.e.
Indeed, if it vanished on a bounded set $A$ of positive Lebesgue measure, the nonzero continuous function $y\mapsto\int_A\phi_r(y-x)dx$ would have zero integral against $\mu$, contradicting $\mu$ having full support  on $\R^d$. Thus $\lambda_r$ is well defined Lebesgue-a.e. on $\R^d$.
We set $\lambda_r(x)=1$ wherever $\int\phi_r(y-x)\mu(dy) =0$.
When both \ref{jd0} and \ref{jd8} hold, $\phi\in W^{1,\8}(\Rd)$, so it is continuous by the Sobolev embedding theorem. We use the continuous version of $\phi$; $\mu$ having full support then makes the integral strictly positive everywhere.

 \item[(ii)]  
 A sufficient condition for \ref{jd0} to hold is the existence of
an open ball $B_{\jd} \inn \Rd$ on which $\jd$ is bounded below by a constant $c_{\jd}>0$.
In this case, we can easily find a
 $C^1$ smooth probability density function $\phi$ supported in $B_{\jd}$ so that \ref{jd0} holds.
 A typical example of $\jd$
  that satisfies both  \ref{jd0} and \ref{jd8} is 
  $\jd(x) = \frac{\1_{B(0,1)}(x)}{|B(0,1)|}$.
In this case, $\sigma_\jd^2 = \int_{\R^d} x_1^2 \jd(x) dx= \frac{1}{d+2}$. But we do not require $\jd$ to have compact support (so random walks can have long-range jumps).

 \item[(iii)]  When $d=1$, condition \eqref{measure upper bound condition} is automatically satisfied by taking $\beta_\mu=0$.
Assuming \ref{jd0} and \ref{jd8}, a sufficient condition for \eqref{measure lower bound condition} 
of condition \ref{condA} to hold
is that for any $R > 0$, there exist $c_{R}>0$ and $\3_{\mu}'>0$ such that for any $x\in B(0,R)$ and $r\in(0,1)$ we have $\mu(B(x,r))\>c_R r^{\3_{\mu}'}$.
To see this, note that $\phi$ is continuous by \ref{jd0} and \ref{jd8}, as shown in (i). Hence $\phi$ is bounded from below by a positive constant on an open ball. Consequently, the term inside the logarithm in \eqref{measure lower bound condition} is bounded by $C r^{-\beta_{\mu}'}$ for some $C=C(R)>0$.
 A typical nontrivial example of $\mu$ that satisfies \ref{condA}  is the Liouville measure 
on $\R^2$
from Liouville quantum gravity, which we will briefly introduce after \Cref{mainthm2}.  
\qed
\end{enumerate} 
\end{remark}

To state our results precisely, let
 $(\Omega, \mathcal A,  \P_x,  x\in\Rd )$ be the probability space on which $\{\xi_i,\eta_i\}_{i=0}^\8$ are defined, with $\P_x(\xi_0=x)=1$.
 Set $\P_\nu=\int\nu(dx)\P_x$ for any finite measure $\nu$; when $\nu$ is a probability measure, this is the law of the process with initial distribution $\nu$.
Given a metric space $(E, \mathsf{d})$, let $\DD([0,\8); E)$ 
(respectively, $\DD([0,T]; E)$ with $T\in (0, \infty)$)
denote the space of right continuous paths with left limits from $[0,\8)$ 
(respectively, $[0, T]$)
to $E$ equipped with the Skorokhod $J_{1}$ topology; see, e.g. \cite[Chapter 14]{kallenberg1997foundations} for a brief introduction of this topology. See also \cite{Bil99,EK86}. Throughout this paper, unless otherwise specified, 
when we say the Skorokhod topology, we mean the Skorokhod $J_{1}$ topology. In this paper, 
we take $E=\R^d$ endowed with the ordinary Euclidean metric or  $E=\Rd_{\partial}:=\R^d\cup\{\partial\}$,
the one-point compactification of $\R^d$,   
endowed with  the metric 
\begin{equation}\label{e:metric}
 \mathsf{d} (x,y):= 
 \begin{cases}
  \frac{|x-y|}{\sqrt{1+|x|^2}\sqrt{1+|y|^2}}  \quad &\hbox{for } x, y\in \R^d, \\
	    \frac1{\sqrt{1+|x|^2}}  &\hbox{for } x\in \R^d \hbox{ and } y=\partial\ \hbox{(and symmetrically)},  \smallskip \\
    0 &\hbox{for }  x=y=\partial .
    \end{cases}
   \end{equation} 
 Note that the metric $\mathsf{d}$ is bounded by $1$ and induces the usual Euclidean topology on $\R^d$.
The element $\partial$ is usually called the cemetery point. 
For the processes considered here, $\partial$ is a trapping point; that is, once a path visits $\partial$,
it stays there afterwards. Unless otherwise specified, we extend a function defined on $D\subset \R^d$ to $D\cup \{\partial\}$ by setting its value at $\partial$
to be 0.
When paths $\{\Gamma_n; n\geq 1\} \subset \DD([0,\8); E)$ converge in the Skorokhod topology to a \textit{continuous} path $\Gamma\in\DD([0,\8); E)$, 
it is equivalent to uniform convergence on compact time intervals, that is, 
$$  
 \lim_{n\to \infty} \sup_{t\in[0,T]} \mathsf{d}  (\Gamma_n(t), \Gamma(t)) = 0 \quad \hbox{for every } T>0.
 $$
 The same result holds on  $\DD([0, T]; E)$. 
  See, e.g., \cite[Lemma 10.1]{EK86} for a proof of this fact.

\medskip

The following two theorems are the main results of this paper. Note that both theorems are about time change of the \textit{standard} Brownian motion, but the result can be easily extended to Brownian motion with a constant positive-definite diffusion matrix through a linear transformation. For an open subset $D\subset \R^d$, we use $C_c(D)$ to denote the space of continuous functions with compact support in $D$.

\begin{thm} [Approximation under absolutely continuous initial distributions]
\label{mainthm1}
Under assumption \ref{jd0}, let $f_0\in C_c(\R^d)$ be nonnegative and not identically zero, and set
$Z_r:=\int f_0\,d\mu_r$, $Z:=\int f_0\,d\mu$, $\nu_r:=Z_r^{-1}f_0\cdot\mu_r$, and $\nu:=Z^{-1}f_0\cdot\mu$. Then the law of $\{X^{r}_t;t\>0\}$ under $\P_{\nu_r}$ converges weakly in $\DD([0,\8);\R^d_{\partial})$ in the Skorokhod topology to the law of time-changed Brownian motion $X$ by $\mu$ with initial distribution $\nu$ as $r\downarrow0$.
\end{thm}

\begin{thm}[Approximation under individual starting points]
\label{mainthm2}
Under assumptions \ref{jd0}, \ref{jd8} and \ref{condA}, for any family $\{x_r;r>0\}\subset\R^d$ converging to $x_0\in\R^d$ as $r\downarrow0$, the law of $\{X^{r}_t;t\>0\}$ under $\P_{x_r}$ converges weakly in $\DD([0,\8);\R^d_{\partial})$ in the Skorokhod topology to the law of time-changed Brownian motion $X$ by $\mu$ starting from $x_0$ as $r\downarrow 0$. 
\end{thm}

\begin{remark}\rm \label{rm2}
\begin{enumerate}  
\item[(i)]  Theorems  \ref{mainthm1} and \ref{mainthm2} in particular give a direct way of constructing time-changed Brownian motion
$X$ directly from the smooth measure $\mu$ without going through the PCAF $A$ associated with it and its inverse process $\{\tau_t; t\geq 0\}$. They also provide a practical implementable scheme to simulate $X$. 

\item[(ii)]  Condition \ref{jd0} is used in \Cref{energy bounded} to establish the Mosco convergence in \Cref{mosco}. Condition \ref{jd8} is used to get the uniform estimates of
Green functions of 
scaled continuous-time
random walks in \Cref{green est}, 
which is a key step in establishing the spatial regularity estimates for the semigroups of time-changed continuous-time random walks.
The crude Green function estimates obtained in \Cref{green est} as well as the transition densities/probabilities of discrete/continuous time random walks on $\Rd$ obtained along the way, are also of independent interest.

\item[(iii)] The moment condition in \ref{jd8} is automatically satisfied
 when $d\< 3$ as $\jd$ is assumed a priori to have a finite second moment.
	Using \eqref{measure upper bound condition} of condition \ref{condA},
	 we can in fact show that $\mu$ charges no polar sets and is a smooth measure in the strict sense in 
	 \Cref{smooth measure}. See \cite{FOT, CF12} for the definition of smooth measures in the strict sense. This allows us to define time-changed Brownian motion
	starting from every   point in $\R^d$. Condition \eqref{measure lower bound condition} of   \ref{condA} is   used only  in \Cref{Holder}.

\item [(iv)] 
It is  shown  in \Cref{stopped} that for suitably chosen $R>0$, as $r\downarrow0$, the process $X^r$ stopped upon leaving $B(0,R)$ converges in distribution to $X$ stopped upon leaving $B(0,R)$ in
the Skorokhod space $\DD([0,\8);\R^d)$
  under the  uniform convergence  topology over  the  time interval $[0,\infty)$.
If all the processes involved ($X$ and $X^r$) are conservative, then the weak convergence in Theorems \ref{mainthm1} and \ref{mainthm2} holds in $\DD([0,\8);\R^d)$ under the Skorokhod topology. This in particular is the case for  $d=1,2$. Indeed, 
when $d=1,2$, the standard Brownian motion $W$ and random walks $W^r$ on $\R^d$ are recurrent and symmetric. 
It follows from \cite[Theorem 5.2.5]{CF12} that both $X$ and $X^r$ are also recurrent, hence have infinite lifetimes. 
\qed
\end{enumerate} 
\end{remark}

Our results are readily applicable to
Liouville Brownian motions constructed and defined in \cite{Berestycki2013DiffusionIP, garban2016liouville}, which 
are time-changed Brownian motions with respect to the Liouville measures.
Liouville Brownian motions are a family of canonical diffusion processes,  indexed by $\gamma \in (0, 2)$, under Liouville quantum gravity.
It is shown in   \cite{BG22, GMS21} that they are the scaling limits of   random walks on mated-CRT planar maps.

 For each $\r\in(0,2)$,  the Liouville measure $M_{\r}$
(in the random environment of a massive Gaussian free field $h$ on $\C=\R^2$)
is the weak limit 
$\lim_{\e\downarrow0}e^{\r h_{\e}(z) - \frac{\r^{2}}{2}\EX(h_{\e}^{2}(z))}dz$, 
  where $h_{\e}$ is the circle average of $h$. 
  It is shown in \cite[Theorem 3.1]{Andres:2016aa} (see also \cite[Theorem 2.2]{garban2016liouville}) 
that 
for almost every realization of the random Liouville measure $M_{\r}$, 
 for any $\e>0$ and $R>1$, there are positive constants $C_1$ and $C_2$ so that
\begin{equation*}
C_{1} r^{\2_{1}+\e} \leq M_{\r}(B(x,r)) \leq C_{2} r^{\2_{2}-\e}  \quad \hbox{for every } x\in B(0,R)
\hbox{ and }  r\in(0,1), 
\end{equation*}
where $\2_{1} = \frac{1}{2}(\r+2)^{2}$ and $\2_{2} =  \frac{1}{2}(2-\r)^{2}$.
This implies that the Liouville measure $M_{\r}$ satisfies
the condition \ref{condA} as long as $\phi$ is chosen to be bounded below away from 0 on an open ball. By \Cref{smooth measure} it charges no polar sets. Furthermore, by \cite[Theorem 2.7]{garban2016liouville} the PCAF with Revuz measure  $M_{\r}$ is strictly increasing to infinity for any starting point $x\in \Rd$. This implies that
the Liouville measure $M_{\r}$ has full quasi-support on $\R^2$  by \cite[Theorem 5.2.1 (i)]{CF12}.
Hence \Cref{mainthm2}, together with Remark \ref{rm2}(iv),
in particular gives the construction as well as a discrete    approximation scheme for the Liouville Brownian motion just using the Liouville measure $M_{\r}$ itself rather than its associated PCAF and its inverse. This is in fact one of the motivations of this paper. 

\begin{corollary}
	Assume \ref{jd0} and \ref{jd8}. Let $d=2$ and $\r \in(0,2)$. Let $\mu$ be the Liouville measure with parameter $\r$ and $X$ be the Liouville Brownian motion. Then for any family $\{x_r\}\subset \R^d$ converging to $x_0\in\R^d$ as $r\downarrow0$, the time-changed random walk $\{X^{r}_t;t\>0\}$ starting from $x_r$ converges in distribution in $\DD([0,\8);\R^d)$ in the Skorokhod topology to the Liouville Brownian motion $X$ starting from $x_0$ as $r\downarrow0$.
\end{corollary}

\medskip

 The main novelty of this paper and the main difficulties we encounter and overcome are summarized as follows.
 \begin{enumerate}   
 \item [(i)] This paper  provides a direct random-walk approximation scheme for Brownian motion time-changed by 
 any smooth measure $\mu$  with full quasi-support on $\R^d$, using only the information of the measure $\mu$ 
 via $\lambda_r (x)$ defined by \eqref{e:1.4}, which determines the jump rate of the discrete random walk at scale $r$.
 This scheme not only gives an implementable algorithm to generate time-changed Brownian motion, but also gives 
 an alternative way of constructing time-changed Brownian motion without using the PCAF associated with $\mu$. 
 
 \item [(ii)] The weak convergence is established in $\DD([0,\8); \R^d_\partial)$ equipped with the Skorokhod topology.
   Our strategy is to first establish Mosco convergence of the corresponding Dirichlet forms and next use the resulting semigroup convergence to prove finite-dimensional convergence of the stopped processes. We then deduce by tightness and localization   
 the weak convergence of the time-changed random walks in the Skorokhod topology.
  The advantage of Mosco convergence is that it provides the required semigroup convergence and identifies the limit process.
 However, proving the Mosco convergence is far from easy as 
 the $L^2$-space for the approximating time-changed  random walk $X^r$ changes as $r\downarrow0$.  
 We overcome this difficulty through some delicate energy estimates carried out in Subsection \ref{S:3.1}. 
  
   \item [(iii)] To prove tightness and weak convergence of the time-changed random walks 
   under absolutely continuous initial distributions, we use a localization argument together with
   Aldous's tightness criterion.
	We stop the time-changed random walks  upon leaving bounded regions so that they become uniformly integrable martingales. The main difficulty is to prove the weak convergence of the finite-dimensional distributions of the stopped processes. We do this by combining the convergence of the killed semigroups with the convergence of the exit distributions.
	
      \item [(iv)] To establish the weak convergence of the time-changed random walks  under individual starting points, Theorem \ref{mainthm2}, 
      we need to deal with the singularity arising from the atom in the transition distributions due to the holding time before the first jump.
      This singularity causes new challenges in establishing the H\"older regularity of harmonic functions and the transition semigroups of the time-changed random walks. 
To the best of the authors' knowledge, all the existing results (e.g., \cite{BKU10, chen2015quenched, SZ97}) 
on Green function estimates and the H\"older regularity for 
harmonic functions and transition semigroups for random walks are not applicable in our setting. 
We establish the needed Green function estimates in Proposition \ref{green est} 
and H\"older regularity estimates in Propositions \ref{holder Harmonic functions} and \ref{Holder}, respectively. 
As we do not assume that $\jd$ has compact support, universal bounds of Gaussian type need not hold. For Green function estimates, we utilize the local central limit theorem and the Burkholder-Davis-Gundy inequality to overcome this difficulty. Standard approaches to show H\"older regularity for harmonic functions and transition semigroups use transition density estimates. Since the transition densities may not exist in our general setting, we obtain similar estimates in integral form with delicate estimates to deal with the impact of possible atoms in transition distributions.
 
 \end{enumerate}

\medskip

The rest of the paper is organized as follows. In \Cref{Preliminaries}, we recall the necessary background on Dirichlet forms, time changes by smooth measures, and Mosco convergence for varying Hilbert spaces, and we identify the Dirichlet forms associated with the limiting and approximating processes. In \Cref{section2}, we prove \Cref{mainthm1}. 
This is done by first establishing Mosco convergence of the Dirichlet forms associated with the time-changed random walks, next using the resulting semigroup convergence to prove finite-dimensional convergence of the stopped processes, and then deducing by localization 
the weak convergence of the time-changed random walks in the Skorokhod topology.
  \Cref{section thm2} is devoted to the proof of \Cref{mainthm2}.
 This is carried out  by establishing the Green function estimates in \Cref{green est}, H\"older regularity of harmonic functions
 of the time-changed random walks in \Cref{holder Harmonic functions} and of their transition semigroups 
 in \Cref{Holder} and \Cref{killed Holder}.

\medskip

In this paper, we use the following notations and conventions. 
We use $:=$ as a way of definition.
For $a,b \in \R$,  $a\wedge b:=\min\{a,b\}$, $a\vee b:=\max\{a,b\}$ and $a_+:=a\vee0$.
We will fix the dimension $d\>1$. We use $\N$ and $\N^*$ to denote nonnegative integers and positive integers respectively. We denote $\mathbb{Q}_+$ for nonnegative rational numbers.
Unless otherwise specified, functions defined on Euclidean spaces or their Borel subsets are assumed to be Borel measurable.
The symbols $c,C$ with or without numerical subscripts stand for positive constants whose value may change from place to place, and they may only depend on the dimension $d$ and values related to $\jd$, $\mu$ or $\phi$, unless dependence on other variables is explicitly specified. 
By adding letter subscripts except numbers to the symbols $c,C$ we indicate their dependence on those subscripts, and their values may also change from place to place. 
We use $r>0$ as subscript/superscript on $W,X,Q$ etc. to denote the step size $r$ in the approximation scheme. 
Sometimes we will fix a sequence $\{r_n; n\geq 1\} \subset (0, 1]$ with $r_n\downarrow0$ and with a slight abuse of notation, we will replace subscript/superscript $r$ on $W,X,Q$ etc. by $n$ to denote the corresponding object when $r=r_n$. We use $O(g)$ and $o(g)$ for quantities whose ratios to $g$ remain bounded and tend to $0$, respectively, in the stated limit; $o_h(1)$ denotes a quantity tending to $0$ as $h$ approaches its stated limit. We use $L^{2}(\sX;\mu)$ to denote the $L^{2}$-space on $\sX$ with respect to the measure $\mu$; when $\sX=\Rd$ we simply denote it by $L^{2}(\mu)$. Given a Dirichlet form $(\mathcal{E}, \mathcal{F})$, we write $\mathcal{E}(u) = \mathcal{E}(u,u)$ for $u\in\mathcal{F}$.

\section{Preliminaries}\label{Preliminaries}

In this section we introduce the background and preliminary results that will be used later. 

\subsection{Some basics of potential theory}\label{potential theory}

We state basic concepts from potential theory. For detailed discussions, see \cite{FOT, CF12}. 

Let $\sX$ be a locally compact separable metric space, and let $m$ be a positive Radon measure on $\sX$. 
Denote by $\sX_\partial:= \sX \cup \{\partial\}$ the one-point compactification of $\sX$. 
Let $(\mathcal E,\mathcal F)$ be a symmetric regular Dirichlet form on $L^2(\sX; m)$.
Its extended Dirichlet  space $\sF_e$ consists of $m$-equivalence classes  of $m$-measurable functions
$f$ on $\sX$ such that $|f|<\infty$ $m$-a.e. on $\sX$ and there is an $\sE$-Cauchy sequence $\{f_n; n\geq 1\}
\subset \sF$ such that $\lim_{n\to \infty} f_n=f$ $m$-a.e. on $\sX$. 
It is known (see, e.g.,  \cite[Theorem 1.1.5]{CF12}) that for $f\in \sF_e$,
$\sE( f, f):= \lim_{n\to \infty} \sE (f_n, f_n)$ is 
well-defined independent of the approximating sequence $\{f_n; n\geq 1\}$ of $f$ in its definition,
 and that  
 \begin{equation}\label{e:2.1}
 \sF =\sF_e \cap L^2(\sX; m).
 \end{equation}

 For a closed subset $F$ of $\sX$, define 
\begin{equation*}
	\mathcal{F}_{F} := \{f\in\mathcal{F}: f=0 \  ~m\text{-a.e. on } \sX\- F \}.
\end{equation*}
Define $\mathcal E_1(f, g) := \mathcal E(f, g ) + \int_\sX f(x) g(x) m(dx)$ for $f, g\in \sF$. 
Denote by $\mathcal O$ the family of all open subsets of $\sX$. For $A\in \mathcal O$ we define 
$$\mathcal{L}_A =\{u \in \mathcal{F}: u \geq 1 ~m \text {-a.e. on } A\},$$
\begin{align*}
	\operatorname{Cap}(A) & = \begin{cases}\inf _{u \in \mathcal{L}_A} \mathcal{E}_1(u, u), & \mathcal{L}_A \neq \emptyset \\
\infty & \mathcal{L}_A=\emptyset,\end{cases}
\end{align*}
and for any set $A\inn \sX$ we let 
$$\operatorname{Cap}(A)=\inf_{B\in \mathcal O, A\inn B} \operatorname{Cap}(B).$$
We call this the capacity of $A$. 
We call a subset $\sN\subset \sX$ an $\sE$-polar set if $\sN$ has zero capacity. 
A property is said to hold $\sE$-quasi-everywhere ($\sE$-q.e. in abbreviation)
if it holds everywhere on $\sX\setminus \sN$ for some $\sE$-polar set $\sN$.

An increasing sequence $\{F_k;  k \> 1\}$  of closed sets of $\sX$ is an $\mathcal E$-nest  
if $\cup_{k} \mathcal{F}_{F_{k}}$ is $\mathcal E_1$-dense in $\mathcal{F}$.
It is known that $\sN\subset \sX$ is $\sE$-polar if and only if there is an $\sE$-nest $\{F_k; k\geq 1\}$ so that 
$\sN \inn \cap _{k\>1}(\sX\- F_k)$. A function $f$ defined on $\sX$ is said to be $\sE$-quasi-continuous on $\sX$
if there is an $\sE$-nest $\{F_k;  k \> 1\}$ so that $f$ is continuous on $F_k$ for every $k\geq 1$.
It is known that every $f\in \sF_e$ has an $\sE$-quasi-continuous $m$-version on $\sX$, which is unique $\sE$-q.e.
In the following, we always represent $f\in \sF_e$ by its $\sE$-quasi-continuous version. 

There is an $m$-symmetric Hunt process $X:=(X_t, t\geq 0; \P_x, x\in \sX\setminus \sN)$ associated with
the regular Dirichlet form $(\sE, \sF)$ on $L^2(\sX; m)$, where $\sN\subset \sX$ is $\sE$-polar. 
The Hunt process $X$ is defined on the sample space $\Omega$ and takes values in $(\sX_\partial,\mathcal B(\sX_\partial))$. It is unique in distribution up to an $\sE$-polar set $\sN$ for starting points.
Without loss of generality, we can take the sample space to be the canonical Skorokhod space
$\DD ([0, \infty); \sX_\partial)$ consisting of right-continuous $\sX_\partial $-valued functions having left limits
with $\partial$ being a trapping point. In this case, $X_t(\w) = \w (t)$ for $\w \in \Omega$. 
  Denote the lifetime of $X$ by $\zeta$; that is, $\zeta (\w ) := \inf\{t\geq 0: X_t (\w)= \partial\}$. 
 Let $\{\sF_t; t\geq 0\}$ be the minimum augmented filtration generated by $X$
and $\{\theta_t; t\geq 0\} $ be the time-shift operators on $\Omega$ so that    $\theta_t(\w) (s) =\w (t+s) $ for any $t, s\geq 0$.

A nearly Borel measurable subset $\sN\subset \sX$ is said to be properly exceptional for the Hunt process
 $X$ if $m(\sN)=0$ and 
$$
\P_x (\hbox{there is some } t\in [0, \infty) \hbox{ so that } X_t\in \sN \hbox{ or } X_{t-}\in \sN)=0
\quad \hbox{for every } x\in \sX\setminus \sN.
$$
It is known (see, e.g., \cite[Theorems 3.1.3 and 3.1.5]{CF12}) that every properly exceptional set is $\sE$-polar and
every $\sE$-polar set is contained in a Borel properly exceptional set.

A measure $\mu$ is called smooth if $\mu$ charges no $\mathcal E$-polar set and there exists a $\mathcal{E}$-nest $\{F_k, k \> 1\}$ such that $\mu(F_k)<\8$ for every $k\>1$.

A numerical function $A_t(\omega)$ of two variables $t \geq 0$, $\omega \in \Omega$ is called an additive functional of $X$ if there exists
 $\Lambda \in \mathcal{F}_{\infty}$ and a properly exceptional set $\sN \subset \sX$ with
\[
\P_x(\Lambda)=1 \text { for } x \in \sX \setminus \sN \quad \text { and } \quad \theta_t \Lambda \subset \Lambda \text { for } t>0,
\]
and the following conditions are satisfied:
\begin{itemize}
	\item For each $t \geq 0,\left.A_t\right|_{\Lambda}$ is $\left.\mathcal{F}_t\right|_{\Lambda}$-measurable, where $A_{t}|_{\Lambda}$ is the restriction of $A_t$ to $\Lambda$ and $\left.\mathcal{F}_t\right|_{\Lambda}$ is the $\sigma$-algebra of $\mathcal{F}_t$ restricted to $\Lambda$.
	\item For any $\omega \in \Lambda$, $	A_.(\omega)$ is right continuous on $[0, \infty)$ and has left limits on $(0, \zeta(\omega))$, $A_0(\omega)=0,\left|A_t(\omega)\right|<\infty$ for $t<\zeta(\omega)$, and $A_t(\omega)=A_{\zeta(\omega)}(\omega)$ for $t \geq \zeta(\omega)$.
	\item  The additivity
	\[
	A_{t+s}(\omega)=A_t(\omega)+A_s\left(\theta_t \omega\right) \quad \text { for every } t, s \geq 0,
	\]
	is satisfied.
\end{itemize}
A positive continuous additive functional (PCAF) of $X$ is a continuous additive functional $A$ of $X$ so that $A_t\geq 0$ for every $t\geq 0$. 
It is known that there is a one-to-one correspondence between smooth measures and PCAFs through the Revuz correspondence \eqref{e:1.2}; see \cite[Theorem 4.1.1]{CF12}.

\subsection{Convergences in different Hilbert spaces}

The following definitions are  mainly taken
 from  \cite{kuwae2003convergence}.  
 See also \cite{kolesnikov2005convergence} for a concise introduction on these concepts.

\begin{defn}\label{hilbert space convergence}
	We say that a sequence of Hilbert spaces $\mathcal{H}_{n}$ converges to a Hilbert space $\mathcal{H}$ if there exists a dense subspace $\mathcal{C} \inn \mathcal{H}$ and a sequence of linear operators $\Phi_n: \mathcal{C} \to \mathcal{H}_{n} $ such that $\lim_{n\to\8}\|\Phi_n u\|_{\mathcal{H}_{n}}=\|u\|_\mathcal{H}$ for every $u\in \mathcal{C}$.
\end{defn}

For the following discussion, we assume that the Hilbert spaces $\mathcal{H}_{n}$ converge to the Hilbert space $\mathcal{H}$ in the sense of \Cref{hilbert space convergence}.

\begin{defn}\label{strong convergence}
	(Strong convergence) We say that a sequence $u_n \in \mathcal{H}_{n}$ strongly converges to $u \in \mathcal{H}$ as $n\to\infty$ if there exists a sequence $ v_m \in \mathcal{C}$ that strongly converges to $u$ in $\mathcal{H}$ as $m\to\infty$ such that 
	$$\lim_{m\to\8}\limsup_{n\to\8}\|\Phi_n v_m - u_n\|_{\mathcal{H}_{n}}= 0.$$
	We call $\{v_m\}$ an approximating sequence for the strong convergence of $u_n \in \mathcal{H}_{n}$   to $u \in \mathcal{H}$. 
\end{defn}

\begin{defn}\label{weak convergence}
	(Weak convergence) We say that a sequence $u_n\in \mathcal{H}_{n}$ weakly converges to $u \in  \mathcal{H}$ as $n\to\infty$ if for every sequence $v_n\in \mathcal{H}_{n}$ that strongly converges to $v \in \mathcal{H}$ as $n\to\infty$, we have 
	$$\lim_{n\to\infty}(u_n,v_n)_{\mathcal{H}_{n}}= (u,v)_\mathcal{H}.$$
\end{defn}

For Hilbert spaces $\mathcal H_r$ indexed by $r>0$, we say that
$u_r\in\mathcal H_r$ converges strongly (respectively, weakly) to
$u\in\mathcal H$ as $r\downarrow0$ if, for every sequence $\{r_n;n\geq1\}$ of positive numbers decreasing to $0$,
$u_{r_n}\in\mathcal H_{r_n}$ converges to $u\in\mathcal H$ in the sense of
\Cref{strong convergence} (respectively, \Cref{weak convergence}).

\begin{prop}\label{weak convergence norm bounded}
	For any sequence $u_n\in \mathcal{H}_{n}$ that converges to $u\in \mathcal{H}$ in the sense of \Cref{weak convergence},
	 we have 
	$$\sup_{n}\|u_n\|_{\mathcal{H}_{n}}<\8 \qquad \text{and }\qquad \|u\|_{\mathcal H}\<\liminf_{n\to\infty}\|u_n\|_{\mathcal{H}_{n}}.$$
\end{prop}
\begin{proof}
	See \cite[Lemma 2.3]{kuwae2003convergence}.
\end{proof}

\begin{prop} Suppose that $u_n\in \mathcal{H}_{n}$ converges to $u\in \mathcal{H}$ and $v_n\in \mathcal{H}_{n}$ converges to $v\in \mathcal{H}$ as $n\to\infty$ in the sense of \Cref{strong convergence}. Then $u=v$ if and only if $\lim_{n\to\8}\|u_n-v_n\|_{\mathcal{H}_{n}}=0$.
\end{prop}
\begin{proof}
	See \cite[Lemma 7.1]{kolesnikov2005convergence}.
\end{proof}

\begin{prop} \label{make convergence sequence}
	If $u_n\in \mathcal{C}\inn \mathcal{H}$ converges to $u\in \mathcal{H}$ strongly in $\mathcal{H}$ as $n\to\infty$, then there exists a nondecreasing sequence of indices $k_n\uparrow\8$ such that $v_n:=u_{k_n}$ and $\Phi_n v_n \in \mathcal{H}_{n}$ converges strongly to $u\in \mathcal{H}$ in the sense of \Cref{strong convergence}.
\end{prop}

\begin{proof}
	By \Cref{hilbert space convergence} we have for any $m,n\>0$
	$$\lim_{k\to\8}\|\Phi_k (u_m-u_n)\|_{\mathcal{H}_k}=\|u_m-u_n\|_{\mathcal{H}}$$
	and convergence of $u_n$ implies
	$$\lim_{N\to \8}\sup_{m,n\>N}\|u_m-u_n\|_{\mathcal{H}}=0.$$
	Hence we can choose an increasing sequence of positive integers $\{N_n;n\geq1\}$ tending to infinity such that, for every $n\in\N^*$, for any $m\> N_n$ 
	$$\|u_m-u_{N_n}\|_{\mathcal{H}}\< 1/(2n)$$
	and	there exists $k_{n,m}\in \N$ such that 
	$$\sup_{k\>k_{n,m}}\|\Phi_k(u_m-u_{N_n})\|_{\mathcal{H}_k}\<\|u_m-u_{N_n}\|_{\mathcal{H}}+1/(2n)\<1/n.$$
	Then we find a strictly increasing sequence $\{L_i\}_{i=1}^\8$ such that $L_i > \max_{1\le j\le i} k_{j,N_i}$. Define $v_n=u_{N_i}$ if $L_i\< n < L_{i+1}$, and set $v_n=u_{N_1}$ for $n<L_1$. Then we have for $m\in [L_j,L_{j+1})\cap\N^*$
	$$\limsup_{n\to\8}\|\Phi_n(v_n-v_m)\|_{\mathcal{H}_{n}}\<1/j$$
	hence $$\lim_{m\to\8}\limsup_{n\to\8}\|\Phi_nv_n-\Phi_nv_m\|_{\mathcal{H}_{n}}=0.$$
	As $v_m$ converges to $u$ strongly in $\mathcal{H}$ as $m\to\infty$, we have that $\Phi_n v_n$ converges to $u$ as $n\to\infty$ in the sense of \Cref{strong convergence}.
\end{proof}

\begin{defn} [Operator strong convergence] \label{operator convergence}
	We say that a sequence of bounded operators $T_{n}$ on $\mathcal{H}_{n}$ strongly converges to a bounded operator $T$ on $\mathcal{H}$ if for every sequence $u_{n}$ strongly tending to $u\in \mathcal{H}$, the sequence $T_{n}(u_{n})$ strongly tends to $T(u)$ as $n\to\infty$.
\end{defn}

\subsection{Mosco convergence} \label{section mosco}

Mosco convergence is a convergence between Dirichlet forms which was first introduced by Mosco in \cite{MR1283033}. It is extended to Dirichlet forms in varying spaces in \cite{kuwae2003convergence}. See also \cite{MR2307059, chen2013discrete}.
Let $\mathcal{H}_{n}$ be Hilbert spaces which converge to the Hilbert space $\mathcal{H}$ in the sense of \Cref{hilbert space convergence}. 

\begin{defn}[Mosco Convergence] \label{Mosco Convergence}
	For quadratic forms $(\mathcal E^{n},\mathcal F^{n})$ on $\mathcal H_n$ and $(\mathcal E,\mathcal F)$ on $\mathcal H$, set $\mathcal E^n(v)=+\infty$ for $v\in\mathcal H_n\setminus\mathcal F^n$ and $\mathcal E(u)=+\infty$ for $u\in\mathcal H\setminus\mathcal F$.
	We say that $(\mathcal E^{n},\mathcal F^{n})$ is Mosco convergent to $(\mathcal E,\mathcal F)$ if the following conditions hold:
	\begin{itemize}
	\item For any sequence $u_n\in \mathcal{H}_{n}$ that weakly converges to $u\in \mathcal{H}$ as $n\to \infty$, 
	we have $\mathcal E(u)\< \liminf_{n\to\infty} \mathcal E^{n}(u_n)$.
	
	\item For every $u\in \mathcal{H}$, there exists a sequence $u_n \in \mathcal{H}_{n}$ that strongly converges to $u$ as $n\to \infty$
	such that $\mathcal E(u)\geq \limsup_{n\to\infty} \mathcal E^{n}(u_n)$.
\end{itemize}
\end{defn}

The Mosco Convergence is equivalent to the strong convergence (in the sense of \Cref{operator convergence}) of the semigroups associated to the corresponding Dirichlet forms. See \cite[Theorem 2.4]{kuwae2003convergence}.

We now give the Dirichlet forms of $X^{r}$ and $X$.
Because $\jd(x) = \jd(-x)$ for any $x\in\Rd$, it is easy to check that $m_r(dx) = (\sigma_{\jd}^{2} r^2)^{-1}dx$ is a symmetrizing measure of $Q_r(x,dy) = \jd_{r}(y-x) dy$, i.e., 
\begin{equation*}
	Q_r(x,dy) m_r(dx) = Q_r(y,dx) m_r(dy).
\end{equation*}
Recall that the speed function $\lambda_r$ defined by \eqref{e:1.4} is strictly positive on $\R^d$.
Since $\phi$ is bounded and has compact support and $\mu$ is Radon, $\int\phi_r(y-x)\mu(dy)$ is bounded above on compact sets, so $\lambda_r$ is bounded away from zero on each compact subset of $\R^d$.
By \cite[Theorem 2.2.2]{CF12} (with the discussion on the cases of unbounded speed function on page 55 of \cite{CF12}), 
we know the Dirichlet form of $X^r$ on $L^2(\R^d;\mu_r)$ is 
\begin{align*}
	\begin{cases}
		\mathcal F^{r}&= \mathcal F_e^r \cap L^2(\R^d;\mu_r)\\
		\mathcal E^{r}(u)&= \frac{1}{2} \int\int(u(x)-u(y))^{2} Q_r(x, d y) m_r (d x), 
	\end{cases}
\end{align*}
where $\mathcal F_e^r$ is the extended Dirichlet space of the continuous-time random walk $\~W^r$ before the time change, whose Dirichlet form on $L^2(\R^d;dx)$ is $(\mathcal E^r,L^2(\R^d;dx))$.
We have $\mathcal F_e^r\subset\{u\in L^2_{\mathrm{loc}}(\Rd):\mathcal E^r(u)<\infty\}$.
By the core property under time change (\cite[Section 5.2]{CF12}), $C_c^\infty(\R^d)$ is dense in $\mathcal F^r$ for $\mathcal E_1^r(u):=\mathcal E^r(u)+\|u\|_{L^2(\mu_r)}^2$.
By \cite[Theorem 5.2.2]{CF12},
 the Dirichlet form of the time-changed Brownian motion $X$ on $L^2(\R^d;\mu)$ is
\begin{align}\label{e:2.2}
	\begin{cases}
\mathcal F^\mu&= \mathcal{F}_e  \cap L^2(\R^d;\mu)\\
\mathcal{E}(u)&= \frac{1}{2} \int|\nabla u(x)|^{2}dx,
	\end{cases}
\end{align}
where $\sF_e$ is  the extended Dirichlet space  of the Dirichlet form on $L^2(\R^d; dx)$ associated with the Brownian motion on $\Rd$. 
Moreover, by \cite[Corollary 5.2.12]{CF12}, 
\begin{equation}\label{e:2.3}
\sF^\mu_e = \sF_e,
\end{equation}
where $\sF^\mu_e$ denotes the extended Dirichlet space of $(\sE, \sF^\mu)$ on $L^2(\R^d; \mu)$. 
When $d=1,2$, 
$\mathcal{F}_{e} = BL(\mathbb R^d)$, 
 where $BL(D):=\{u \in L^{2}_{\mathrm{loc}}(D): \nabla u \in L^2(D)\}$ is the Beppo-Levi space on the domain $D\inn \R^d$. When $d\>3$, 
$\mathcal{F}_{e}$ is isometric to $BL(\Rd)/\R$, the quotient space of $BL(\Rd)$ by the subspace of constant functions. 
See \cite[Theorems 2.2.12 and 2.2.13]{CF12} for the above facts.

\medskip

Besides the two Dirichlet forms $(\mathcal F^{r},\mathcal E^{r})$ on $L^2(\R^d;\mu_r)$ and $(\mathcal F^{\mu},\mathcal{E})$ on $L^2(\R^d;\mu)$, we introduce some other Dirichlet forms which are helpful for establishing our main results.
Let $D$ be a bounded open set of $\Rd$ with $C^{1,1}$ boundary, and let $X^{r,D}$ and $X^D$ be the processes obtained by killing $X^r$ and $X$ upon leaving $D$, respectively.
By \cite[Theorem 3.3.8]{CF12}, their Dirichlet forms on $L^2(D;\mu_r|_D)$ and $L^2(D;\mu|_D)$ are given by:
\begin{align}\label{dirichlet form r E}
\left\{\begin{aligned}
\mathcal F^{r,D}
&=\{u\in\mathcal F^r:u=0\ \mathcal E^r\text{-q.e. on }D^c\},\\
\mathcal E^{r,D}(u)
&=\frac12\int_{D\times D}(u(x)-u(y))^2Q_r(x,dy)m_r(dx)\\ 
&\quad +\int_D u(x)^2(1-Q_r(x,D))m_r(dx),
\end{aligned}\right.
\end{align}
and
\begin{align}\label{dirichlet form mu E}
\left\{\begin{aligned}
\mathcal F^{\mu,D}
&=\{u\in\mathcal F^\mu:u=0\ \mathcal E\text{-q.e. on }D^c\},\\
\mathcal E^{\mu,D}(u)
&=\frac12\int_D|\nabla u(x)|^2\,dx.
\end{aligned}\right.
\end{align}
We have $\mathcal E^{r,D}(u)=\mathcal E^r(u)$ for $u\in\mathcal F^{r,D}$ and $\mathcal E^{\mu,D}(u)=\mathcal E(u)$ for $u\in\mathcal F^{\mu,D}$.
Identifying elements of $C_c^\infty(D)$ with their zero extensions to $\Rd$, we obtain from \cite[Theorem 3.3.9(ii)]{CF12} that
\[
\mathcal F^{r,D}=\overline{C_c^\infty(D)}^{\mathcal E_1^{r,D}},
\qquad
\mathcal F^{\mu,D}=\overline{C_c^\infty(D)}^{\mathcal E_1^{\mu,D}},
\]
where the closures are taken with respect to the norms determined by
\[
\begin{aligned}
\mathcal E_1^{r,D}(u)&=\mathcal E^{r,D}(u)+\|u\|_{L^2(D;\mu_r)}^2,\\
\mathcal E_1^{\mu,D}(u)&=\mathcal E^{\mu,D}(u)+\|u\|_{L^2(D;\mu)}^2.
\end{aligned}
\]
We also define $\mathcal E_1^\mu(u)=\mathcal E(u)+(u,u)_\mu$ for $u\in\mathcal F^\mu$, where $(\cdot,\cdot)_\mu$ is the $L^2(\mu)$ inner product.

We will establish the Mosco convergence of $(\mathcal F^{r,D},{\mathcal E}^{r,D})$ to $(\mathcal F^{\mu,D},\mathcal E^{\mu,D})$ when $\mu(\6 D)=0$, and the Mosco convergence of $({\mathcal F}^{r},{\mathcal E}^{r})$ to $({\mathcal F^\mu},{\mathcal E})$ in \Cref{mosco}. By \cite[Theorem 2.4]{kuwae2003convergence}, the Mosco convergence of Dirichlet forms in the sense of \Cref{Mosco Convergence} is equivalent to the $L^{2}$ convergence of the semigroups in the sense of \Cref{operator convergence}.

\section{Approximation under absolutely continuous initial distributions}\label{section2}

In this section, we prove \Cref{mainthm1}. We first establish Mosco convergence of the Dirichlet forms, then use the resulting semigroup convergence to prove the finite-dimensional convergence of the stopped processes, and finally deduce by tightness and localization 
the weak convergence of the time-changed random walks  in the Skorokhod topology.

\subsection{Mosco convergence of the Dirichlet forms}\label{S:3.1}

We start with some auxiliary results.

\begin{lemma} \label{formCalulate}
	For any unit vector $e\in \Rd$ we have $\int \langle e,x \rangle_{\Rd}^2\jd(x) dx = \sigma_{\jd}^{2} < \8$, 
	  where $ \langle \cdot,\cdot\rangle_{\Rd}$ denotes the Euclidean inner product. Furthermore,
	we have for any $r>0$, $x, z\in \R^d$
	\begin{equation*}
		\frac{1}{ \sigma_{\jd}^2 r^2 }\int|\langle x,y-z\rangle_{\Rd}|^2 \jd_r(y-z)dy=|x|^2.
	\end{equation*}
\end{lemma}
\begin{proof}
	Write $e = \sum_{i = 1}^{d} c_i e_i$. Then $\sum_{i = 1}^{d} c_i^2 = 1$, and hence, by orthogonality,
	\begin{equation*}
		\int \langle e,x \rangle _{\Rd}^2\jd(x) dx = \sum_{i = 1}^{d} c_i^2 \int \langle e_i,x \rangle _{\Rd}^2\jd(x) dx  = \sigma_{\jd}^{2}.
	\end{equation*}
	The second identity is immediate when $x=0$. For $x\ne0$, using the identity above with $e=x/|x|$ gives
\begin{align*}
	\frac{1}{ \sigma_{\jd}^2 r^2 }\int\langle x,y-z\rangle _{\Rd}^2 \jd_r(y-z)dy
	&=|x|^2 \sigma_{\jd}^{-2}\int\langle \frac{x}{|x|},\frac{y-z}{r}\rangle _{\Rd}^2 \jd_r(y-z)dy\\
	&=|x|^2 \sigma_{\jd}^{-2}\int\langle \frac{x}{|x|},y\rangle _{\Rd}^2 \jd(y)dy 
	 =|x|^2,
\end{align*}	
which proves the second statement.
\end{proof}

Recall that $\mu_r(dx)=(\v\phi_r * \mu)(dx) = \int \phi_{r}(y-x)\mu(dy)dx$ and $Q_r(x,dy)=\jd_r(y-x)dy$. 
\begin{lemma} \label{weakConvergeRelation}
The following holds.
\begin{enumerate}[label=(\arabic*)]
\item For $p\in[1,\8)$ and $f \in L^p(\mu_{r})$, we have $\phi_{r}*f\in L^p(\mu)$ and $\|\phi_{r}*f\|_{L^p(\mu)} \leq \|f\|_{L^p(\mu_r)}$.
\item For any function $f\in L^1(\mu_r)$, we have $\mu_r(f):=\int f d \mu_r=\int (\phi_r*f)(x)\mu(dx)$.
\item For any $f\in C_c(\R^d)$,   $\lim_{r\downarrow0}\mu_r(f)=\mu(f)$.
\item If $f_r \in L^2(\mu_{r})$ converges weakly to $f \in L^2(\mu)$ as $r\downarrow0$ in the sense of \Cref{weak convergence}, then $\phi_{r}*f_r \in L^2(\mu)$ converges weakly to $f$ in $L^2(\mu)$ as $r\downarrow0$.
\end{enumerate}
\end{lemma}
\begin{proof}
Notice that by Jensen's inequality
$$\int |\phi_r * f|^p d\mu\<\int \phi_r*(|f|^p)d\mu=\int |f|^p d\mu_r.$$
Hence (1) holds. By Fubini's Theorem, for non-negative $f$ on $\R^d$, 
\begin{align*}
	\int f(x) \mu_r(dx)
	& = \int f(x) \int \phi_r(y-x)\mu(dy)dx 
	  = \iint   f(x) \phi_r(y-x) dx \mu(dy)\\
	& = \int (\phi_r*f)(y)\mu(dy).
\end{align*}
This proves (2). Note that for $f\in C_c(\R^d)$, there is $R_0>1$ so that ${\rm supp} [ \phi_r*f] \subset B(0, R_0)$  for every $0<r\leq 1$. 
Property (3) then follows from (2) and the bounded convergence theorem.   
For (4), (1) implies $\phi_{r}*f_r \in L^2(\mu)$. To show the convergence, let $g\in C^2_c(\R^d)$ which is uniformly continuous. 
Recall that $\phi$ has compact support by \ref{jd0}. 
Let $M> 0 $ so that $\operatorname{supp}\phi\subset B(0,M)$ and $\operatorname{supp}g\inn B(0,M)$. 
Define for $r>0$,  $\omega_g (r)= \sup_{x, y\in \R^d: |x-y|<r} |g(x) -g(y)|$.
By Fubini's theorem, for every $0<r<1$, 
\begin{align} \label{eq:conv}
	\begin{split} 
		&\left|\int  f_r g\,d\mu_r-\int (\phi_r*f_r) g \,d\mu\right| \\
		=
		&\left|\int (f_rg)(x)\,\mu_r(dx)
		 -\iint \phi_r(z-x)f_r(x)(g(x)+g(z)-g(x))\,dx\,\mu(dz)\right| \\
		\leq & \iint \phi_r(z-x)|g(x)-g(z)| |f_r(x)|\,dx\,\mu(dz)\\
		= & \iint \1_{B(0, 2M)}(x)  |g(x)-g(z)| \phi_r(z-x)   \,  |f_r(x)| \mu(dz) \, dx \\
		\le &\omega_g(Mr )\int_{ B(0,2M)} |f_r(x)|\,\mu_r(dx) \\
		\leq & \omega_g( Mr ) \mu_r( B(0,2M))^{1/2}\|f_r\|_{L^2(\mu_r)}, 
	\end{split}
\end{align}
where the second equality is due to the fact that $\phi_r (z-x) =0$ if $|x-z|\geq Mr$, 
and   if $|x|\geq 2M$ and $|x-z| <Mr\<M$, then $|z|>M$.
By \Cref{weak convergence norm bounded} $\|f_r\|_{L^2(\mu_r)}$ is bounded for all sufficiently small $r>0$.
Moreover,
$$ 
\mu_r(B(0,2M))=\int\int  \1_{B(0,2M)} (x)   \phi_r(y-x)\,dx\,\mu(dy)\le\mu(B(0,3M))<\infty.$$
Therefore by \eqref{eq:conv}, $\lim_{r\downarrow 0} |\int  f_r g\,d\mu_r - \int (\phi_r*f_r) g \,d\mu| = 0$. 
As  $\lim_{r\downarrow 0} \int  f_r g\,d\mu_r =\int f g\,d\mu$ by assumption, 
we get $\lim_{r\downarrow0}\int \phi_{r}*f_r g\,d\mu= \int f g\,d\mu$ for $g\in C_c^{2}(\Rd)$. Note that by \Cref{weak convergence norm bounded} and (1), we have $\limsup_{r\downarrow0} \|\phi_{r}*f_r\|_{L^2(\mu)} \< \limsup_{r\downarrow0} \|f_r\|_{L^2(\mu_{r})}< \infty$.
The desired weak convergence follows from the density of $C_c^{2}(\Rd)$ in $L^2(\mu)$.
\end{proof}

Next we derive some basic properties of bilinear forms $\mathcal E$ and $\mathcal E^{r}$.

\begin{lemma} \label{form mollifier monotone}
	For any compactly supported probability density function $\phi$ (with respect to the Lebesgue measure on $\R^d$), we have $\mathcal E(\phi*u)\<\mathcal E(u)$ for $u\in BL(\Rd)$ and $\mathcal E^{r}(\phi*u)\<\mathcal E^{r}(u)$ for $u\in L^2_{\mathrm{loc}}(\Rd)$.

\end{lemma}
\begin{proof}
Let $u\in BL(\Rd)$. Since $\phi$ has compact support, the convolutions with $\phi$ are well defined and $\nabla(\phi*u)=\phi*\nabla u$ (in the sense of weak derivative). By Jensen's inequality and shifting invariance of the Lebesgue measure we have 
\begin{align*}
	\mathcal E(\phi*u) 
	&= \frac{1}{2}\int |\nabla(\phi*u)(x)|^2dx \\
	&=  \frac{1}{2}\int |(\phi*\nabla u)(x)|^2dx\\
	&\<  \frac{1}{2}\int \int \phi(z)|\nabla u(x-z)|^2 dz dx\\
	&=  \frac{1}{2}\int \int |\nabla u(x-z)|^2 dx \phi(z)dz\\
	&= \frac{1}{2} \int |\nabla u(x)|^2dx 
	 = \mathcal E(u).
\end{align*}
Similarly for $u\in L^2_{\mathrm{loc}}(\Rd)$, we have
\begin{align*}
	\mathcal E^{r}(\phi*u)
		&=\frac{1}{2 \sigma_{\jd}^2  r^2}\int\int(\phi*u(x)-\phi*u(y))^2\jd_r(y-x)dydx\\
	     &=\frac{1}{2 \sigma_{\jd}^2  r^2}\int\int \left(\int (u(x-z)-u(y-z))\phi(z)dz \right)^2\jd_r(y-x)dydx\\
		&\<\frac{1}{2 \sigma_{\jd}^2  r^2}\int\int\int (u(x-z)-u(y-z))^2\phi(z)dz\jd_r(y-x)dydx\\
		&=\frac{1}{2 \sigma_{\jd}^2  r^2}\int\int\int (u(x-z)-u(y-z))^2\jd_r(y-x)dydx\phi(z)dz\\
		&=\frac{1}{2 \sigma_{\jd}^2  r^2}\int\int\int (u(x)-u(y))^2\jd_r(y-x)dydx\phi(z)dz\\
		&=\mathcal E^{r}(u)\int \phi(z)dz = \mathcal E^{r}(u).
	\end{align*}
where the fourth equality uses the translation invariance of the Lebesgue measure. 
\end{proof}

\begin{lemma}\label{inW12 form monotone}
		Let $u\in L^2_{\text{loc}}(\R^d)$.
	\begin{enumerate}  
	\item[(i)]  If $\liminf_{r\downarrow0}\mathcal E^{r}(u)<\8$,   
	then $u\in BL(\R^d)$.
	
	\item[(ii)] If  $u\in BL(\R^d)$, then
	$\mathcal{E}^{r}(u)\<\mathcal{E}(u)$ for any $r>0$, $\lim_{r\downarrow 0}\mathcal E^{r}(u)$ exists and $\mathcal E(u) = \lim_{r\downarrow 0}\mathcal E^{r}(u)$. 
	 \end{enumerate}
\end{lemma}

\begin{proof}
(i) Suppose   $u\in L^2_{\text{loc}}(\R^d)$ and $\liminf_{r\downarrow0}\mathcal E^{r}(u)<\8$. 
For any $\e>0$ let $u_\e=\varphi_\e*u$,  
where $\varphi$ is any nonnegative smooth function on $\R^d$ with compact support
having $\int \varphi (x) dx=1$ and $\varphi_\e (x):= \e^{-d} \varphi (\e^{-1} x)$.
For  every  $R>0$,
by Taylor approximation we know 
\begin{align*}
I_{R,r} := &\frac{1}{2 \sigma_{\jd}^2  r^2}\int_{|x|<R}\int_{|y-x|<R} |u_\e(y)-u_\e(x)|^2 \jd_{r}(y-x) dydx\\
=&  \frac{1}{2 \sigma_{\jd}^2  r^2}\int_{|x|<R}\int_{|y-x|<R} \langle\nabla u_\e(x),y-x\rangle_{\Rd}^2 \jd_r(y-x) dydx + o_r(1) \\
=& \sum_{i,j = 1}^{d} \frac{1}{2} \int_{|x|<R}  \6_{i} u_\e(x) M_{ij}^{r,R} \6_{j} u_\e (x) dx  + o_r(1),
\end{align*}
where $M_{ij}^{r,R} = \frac{1}{ \sigma_{\jd}^2 r^2 }\int_{|y|<R} y_{i}y_{j} \jd_{r}(y)dy$
and the Taylor error term is of $o_r(1)$ as $r\downarrow0$  (which is dependent on $\e >0 $)
because $|x|<R$ and $|y|< 2R$.
We claim $\lim_{r\downarrow 0} M_{ij}^{r,R} = \d_{ij}$. 
 This is because 
$$
	\left| \frac{1}{ \sigma_{\jd}^2  r^2}\int_{|y|\>R} y_{i}y_{j} \jd_{r}(y)dy \right|
	\<  \frac{1}{2 \sigma_{\jd}^2  r^2}\int_{|y|\>R} (y_{i}^{2}+y_{j}^{2}) \jd_{r}(y)dy 
	\<  \frac1{\sigma_{\jd}^2 }\int_{|y|\>R/r} |y|^{2} \jd(y)dy, 
$$
which converges to $0$ as $r\downarrow 0$. This together with \Cref{formCalulate} establishes the claim.
Hence
\begin{align*}
	\lim_{r\downarrow0} I_{R,r} = \sum_{i,j = 1}^{d} \frac{1}{2} \int_{|x|<R}  \6_{i} u_\e(x) \d_{ij} \6_{j} u_\e (x) dx
	= \frac{1}{2}\int_{|x|<R}|\nabla u_\e(x)|^2 dx.
\end{align*}
On the other hand, 
\begin{align*}
	\lim_{r\downarrow0} I_{R,r} 
	&= \lim_{r\downarrow0}\frac{1}{2 \sigma_{\jd}^2  r^2}\int_{|x|<R}\int_{|y-x|<R} |u_\e(y)-u_\e(x)|^2 \jd_r(y-x) dydx \\
	&\< \liminf_{r\downarrow0}\mathcal E^{r}(u_\e) 
	\< \liminf_{r\downarrow0} \mathcal E^{r}(u)
\end{align*}
where in the last inequality we used \Cref{form mollifier monotone}.
Hence we have  for every $\e >0$, 
$$
\mathcal E (u_\e)=\lim_{R\uparrow\8}\frac{1}{2}\int_{|x|<R}|\nabla u_\e(x)|^2 dx\<\liminf_{r\downarrow0}\mathcal E^{r}(u).
$$
  Namely, $|\nabla u_\e|$ is bounded in $L^2(\R^d)$. Then we can find a sequence of positive numbers $\e_n\downarrow0$ such that for each coordinate $i\in\{1,2,\ldots,d\}$,
  partial derivative $\6_i u_{\e_n}$ converges weakly in $L^2(\R^d)$ to some function $v_i\in L^2(\R^d)$ as $n\to \infty$. Thus
   for any $f\in C^\8_c(\R^d)$ we have
$$(f,v_i)_{L^2(\R^d)}=\lim_{n\to\8}(f,\6_i u_{\e_n})_{L^2(\R^d)}=\lim_{n\to\8}(-\6_i f, u_{\e_n})_{L^2(\R^d)}=(-\6_i f, u)_{L^2(\R^d)}$$
which shows $\6_i u = v_i\in L^2(\R^d)$ and so $u\in BL(\R^d)$.

\medskip

(ii) Let $u\in   L^2_{\text{loc}}(\R^d) \cap BL(\R^d)$. For $\e>0$, let $u_\e=\varphi_\e*u$, where $\varphi_\e$ is as in part (i). By the calculation in part (i) and \Cref{form mollifier monotone},
\begin{equation*}
\mathcal E(u_\e)\<\liminf_{r\downarrow0}\mathcal E^r(u_\e)\<\liminf_{r\downarrow0}\mathcal E^r(u).
\end{equation*}
Since $\nabla u_\e$ converges to $\nabla u$ in $L^2(\Rd)$ as $\e\downarrow0$, we have $\mathcal E(u)\<\liminf_{r\downarrow0}\mathcal E^r(u)$.
Suppose momentarily that
$u$  has compact support.  Then $u\in L^{2}(\Rd)$ and by Fourier transform  we have for any $r>0$
\begin{align} \label{form compact monoton}
	\begin{split}
		\mathcal{E}^{r}(u) 
		&= \frac{1}{ \sigma_{\jd}^2 r^2 } \int\int |\hat u(z)|^2 (1-\cos(\langle y,z\rangle_{\Rd}))\jd_r(y)dydz\\
		&\< \frac{1}{ \sigma_{\jd}^2 r^2 }\int\int |\hat u(z)|^2  \frac{1}{2}\langle y,z\rangle_{\Rd} ^{2}\jd_r(y)dydz\\
		&= \frac{1}{2}\int |\hat u(z)|^2|z|^{2}dz = \mathcal{E}(u).
	\end{split}
\end{align}
Here $\hat u(z):=(2\pi)^{-d/2}\int e^{-i\langle x,z\rangle_{\Rd}}u(x)\,dx$.
 For $u\in L^2_{\text{loc}}(\R^d) \cap BL(\R^d)$, we choose a sequence of functions $\{u_n; n\geq 1 \} \subset C_{c}^{\8}(\Rd)$ such that 
 $\lim_{n\to \infty} \mathcal{E}(u_n-u)= 0$. By a property of BL functions (cf. \cite{deny1954espaces}) there exists a sequence of constants $\{C_n; n\geq 1\}$ such that $u_n + C_n$ 
 is $L_{\text{loc}}^{2}$-convergent to $u$ as $n\to\infty$. Take a strictly increasing subsequence of indices $\{n_k; k\geq 1\}$ 
 such that $u_{n_k} + C_{n_k}$ converges a.e. to $u$ as $k\to \infty$.
   By Fatou's lemma and \eqref{form compact monoton}, we have for any $r>0$
\begin{align*}
	\mathcal E^{r}(u) 
	&\< \liminf_{ k\to \infty }\mathcal E^{r}(u_{n_k}) \< \liminf_{ k\to \infty }\mathcal E(u_{n_k}) = \mathcal{E}(u).
\end{align*}
 Hence $\limsup_{r\downarrow0} \mathcal{E}^{r}(u) \< \mathcal{E}(u) \< \liminf_{r\downarrow0} \mathcal{E}^{r}(u)$.
 This  shows $\lim_{r\downarrow0} \mathcal{E}^{r}(u) = \mathcal{E}(u)$.
 \end{proof}

We also need the following two lemmas to prove the Mosco convergence.

\begin{lemma}\label{energy bounded}
	Recall that $\phi$ is the probability density function as in \ref{jd0} and  $\phi_r(x) :=r^{-d} \phi(x/r) $ for $r>0$. 
	There is a constant $C>0$, independent of $r$ and $u$, such that for every $r>0$ and $u\in L^2_{\mathrm{loc}}(\Rd)$ with $\mathcal E^r(u)<\8$,
\begin{equation}\label{e:3.3}
\mathcal E(\phi_r*u)\le C\mathcal E^r(u).
\end{equation}
	In particular, for any family $\{u_r; r>0\} \subset  L^2_{\mathrm{loc}}(\Rd)$ with $\sup_{r>0}\mathcal E^r(u_r)<\8$,
	we have  $\sup_{r>0} \mathcal E(\phi_r*u_r)<\8$.
	Moreover, for each $r>0$ and $u\in\mathcal F^r$, we have $\phi_r*u\in\mathcal F^\mu\subset\mathcal F_e$.
\end{lemma}

\begin{proof}
For any $u\in L_{\text{loc}}^1(\R^d)$, by the property of weak derivatives we have $\nabla(\phi_{r}*u)=(\nabla \phi_{r})*u$. Since $\phi_{r}$ has compact support, we have $\int_{\Rd} \nabla \phi_{r}(z)dz = 0$. Hence
	\begin{align*}
		\nabla(\phi_r*u)(x)
		&= (\nabla \phi_{r})*u
		=\int  (\nabla   \phi_r ) (z)  u(x-z)  dz  \\
		&=   \int  (\nabla   \phi_r ) ( z) ( u(x-z) -u(x)) dz\\
		&=\int\frac{1}{r^{d+1}}(\nabla\phi)(z/r)(u(x-z)-u(x))dz\\
		&=\int\frac{1}{r}(\nabla\phi)(z)(u(x-rz)-u(x))dz .
	\end{align*}
Writing $C_0$ for the constant in \ref{jd0}, we obtain
\begin{align*}
	|\nabla(\phi_r*u)(x)|
	&\<\frac{1}{r}\int|\nabla\phi(z)||u(x-rz)-u(x)|dz 
	  \<\frac{C_0}{r} \int |u(x-rz)-u(x)| \jd(z) dz\\
	&=\frac{C_0}{r} Q_r|u-u(x)|(x) .
\end{align*}
Then we have by Jensen's inequality
\begin{align*}
	\mathcal E(\phi_r*u)
&\< \frac{C_0^2}{2r^2}\int (Q_r|u-u(x)|)^2(x)dx 
    \< \frac{C_0^2}{2r^2}\int Q_r(u-u(x))^2(x)dx\\
&=C_0^2\sigma_{\jd}^2\mathcal E^{r}(u).
\end{align*}
This proves the inequality \eqref{e:3.3}  with $C=C_0^2\sigma_{\jd}^2$.
Next fix $r>0$ and let $u\in\mathcal F^r$. Choose a sequence $\{u_j;j\geq1\}\subset C_c^\infty(\Rd)$ with $\lim_{j\to\infty}\mathcal E_1^r(u_j-u)=0$.
Set $v=\phi_r*u$ and $v_j=\phi_r*u_j\in C_c^\infty(\Rd)\subset\mathcal F^\mu$. By the preceding estimate and \Cref{weakConvergeRelation},
\[
\limsup_{j\to\infty}\big(\mathcal E(v_j-v)+\|v_j-v\|_{L^2(\mu)}^2\big)
\le \lim_{j\to\infty}\big(C\mathcal E^r(u_j-u)+\|u_j-u\|_{L^2(\mu_r)}^2\big)=0.
\]
Thus $\{v_j\}$ is $\mathcal E_1^\mu$-Cauchy and converges to $v$ in $L^2(\mu)$ as $j\to\infty$. By the closedness of $(\mathcal E,\mathcal F^\mu)$, we conclude that $v\in\mathcal F^\mu$.
This finishes the proof.
\end{proof}

\begin{lemma}\label{form monotone 2}
	Let $v_{s}\in BL(\Rd)$ for $s>0$ and $\sup_{s>0}\mathcal{E}(v_{s})<\8$. Then for any $r>0$ we have $\limsup_{s\downarrow0} \mathcal{E}^{r}(v_{s}) \< \limsup_{s\downarrow0} \mathcal{E}^{s}(v_{s})$ and $\liminf_{s\downarrow0} \mathcal{E}^{r}(v_{s}) \< \liminf_{s\downarrow0} \mathcal{E}^{s}(v_{s})$.
\end{lemma}
\begin{proof}
	First we assume $v\in W^{1,2}(\Rd):=\{u\in L^{2}(\Rd):\nabla u \in L^{2}(\Rd)\}$. 
	Fourier transform implies that 
	\begin{equation*}
		\mathcal{E}^{r}(v) = 
		\frac{1}{ \sigma_{\jd}^2 }\int\int |\hat v(z)|^2 \frac{1-\cos(\langle ry,z\rangle_{\Rd})}{r^{2}}\jd(y)dydz
	\end{equation*}
	Let $R>0$. On the one hand, there exists $s_{0}$ depending only on $r,R$ such that for any $s<s_{0}$ we have
	\begin{align*}
		&\frac{1}{ \sigma_{\jd}^2 }\int\int_{|\langle ry, z \rangle|<R} |\hat v(z)|^2 \frac{1-\cos(\langle ry,z\rangle_{\Rd})}{r^{2}}\jd(y)dydz\\
		\<& \frac{1}{ \sigma_{\jd}^2 }\int\int_{|\langle ry, z \rangle|<R} |\hat v(z)|^2 \frac{1-\cos(\langle sy,z\rangle_{\Rd})}{s^{2}}\jd(y)dydz\\
		\<& \mathcal{E}^{s}(v).
	\end{align*}
	On the other hand, for any $R>0$ we have
	\begin{align*}
		&\frac{1}{ \sigma_{\jd}^2 }\int\int_{|\langle ry, z \rangle|\>R} |\hat v(z)|^2 \frac{1-\cos(\langle ry,z\rangle_{\Rd})}{r^{2}}\jd(y)dydz\\
		\<& \frac{1}{ \sigma_{\jd}^2 }\int\int_{|\langle ry, z \rangle|\>R} |\hat v(z)|^2 \frac{2}{r^{2}}\jd(y)dydz\\
		\<& \frac{1}{ \sigma_{\jd}^2 }\frac{2}{R^{2}}\int\int_{|\langle ry, z \rangle|\>R} |\hat v(z)|^2 |\langle y, z \rangle|^{2} \jd(y)dydz\\
		\<& \frac{2}{\sigma_{\jd}^2 R^{2}}\int\int |\hat v(z)|^2 |y|^{2} |z|^{2} \jd(y)dydz\\
		=& \frac{ 2d}{R^{2}}\int|\hat v(z)|^2|z|^{2} dz = \frac{4d}{R^{2}} \mathcal{E}(v).
	\end{align*}
Then we have
\begin{align}\label{form monotone 2:1}
	\mathcal{E}^{r}(v) 
	&\< \mathcal{E}^{s}(v) + \frac{4d}{R^{2}} \mathcal{E}(v).
\end{align}
When $v\in BL(\Rd)$, we can take a sequence of $v_{n}\in W^{1,2}(\Rd)$ such that $\lim_{n\to\infty}\mathcal{E}(v_{n}-v)=0$. By \Cref{inW12 form monotone} we also have $\limsup_{n\to\infty}\mathcal{E}^{l}(v_{n}-v)\leq\lim_{n\to\infty}\mathcal{E}(v_{n}-v)=0$ for any fixed $l>0$. By a standard approximation argument, \eqref{form monotone 2:1} holds for all $v\in BL(\Rd)$.
Replacing $v$ in \eqref{form monotone 2:1} by $v_{s}$, then taking the limit superior and limit inferior as $s\downarrow0$ and letting $R\uparrow\8$, we get the desired result. 
\end{proof}

For an open set $D\subset\Rd$, set
\[
\mathcal H_r^D:=L^2(D;\mu_r|_D),\qquad
\mathcal H^D:=L^2(D;\mu|_D),\qquad \mathcal C^D:=C_c^\infty(D),
\]
and let $\Phi_r^D:\mathcal C^D\to\mathcal H_r^D$ be the natural inclusion.
By \Cref{weakConvergeRelation}, these spaces converge in the sense of
\Cref{hilbert space convergence}. Throughout the following arguments, strong
and weak convergence of $u_r\in\mathcal H_r^D$ to $u\in\mathcal H^D$ mean
\Cref{strong convergence} and \Cref{weak convergence}, respectively, with this core and
these inclusions, along every sequence $\{r_n;n\geq1\}$ of positive numbers decreasing to $0$.
The same convention with $D=\Rd$ specifies convergence in the whole space.

Recall that $(\mathcal F^{r,D},\mathcal E^{r,D})$ and $(\mathcal F^{\mu,D},\mathcal E^{\mu,D})$ are defined in \eqref{dirichlet form r E} and \eqref{dirichlet form mu E} respectively in \Cref{section mosco}. The processes $X^{r,D}$ and $X^{D}$ are the killed processes associated to the Dirichlet forms.

\begin{lemma}\label{boundary restriction}
Let $D\subset \R^d$ be a bounded open set and satisfy $\mu(\6 D)=0$.
\begin{enumerate}
\item[\rm (i)] For every $g\in C(\overline D)$,
$\lim_{r\downarrow0} \int_D g\,d\mu_r = \int_D g\,d\mu.$
 In particular, $ \lim_{r\downarrow0}  \mu_r(D)= \mu(D)$.

\item[\rm (ii)] If $g\in C(\overline D)$ and $g_r$ are bounded functions on $D$ such that $g_r$ converges to $g$ uniformly on $D$ as $r\downarrow0$, then $g_r\in\mathcal H_r^D$ converges strongly to
$g\in\mathcal H^D$ as $r\downarrow0$.

\item[\rm (iii)] If $h_r\in\mathcal H_r^{\Rd}$ converges strongly to
$h\in\mathcal H^{\Rd}$ as $r\downarrow0$, then $h_r|_D\in\mathcal H_r^D$
converges strongly to $h|_D\in\mathcal H^D$.
Suppose  that $u_r\in\mathcal H_r^D$ converges weakly to $u\in\mathcal H^D$ as
$r\downarrow0$. Define their zero extensions to $\R^d$ by
\[
\bar u_r(x):=\begin{cases}u_r(x),&x\in D,\\0,&x\notin D,\end{cases}
\qquad
\bar u(x):=\begin{cases}u(x),&x\in D,\\0,&x\notin D.
\end{cases}
\]
Then $\bar u_r\in\mathcal H_r^{\Rd}$ converges weakly to
$\bar u\in\mathcal H^{\Rd}$ as $r\downarrow0$.
\end{enumerate}
\end{lemma}

\begin{proof} 
   (i) Let $D_n:= \{x\in \R^d: {\rm dist}(x, \partial D) <2^{-n}\}$ and
$\{\chi_n; n\geq 1\}$ be a decreasing sequence of continuous functions on $\R^d$ so that 
$0\leq \chi_n\leq 1$, $\chi_n =1 $ on $D_n$ and $\chi_n=0$ on $D_{n-1}^c$. 
Observe that $\chi_n$ decreases to $\1_{\partial D}$. As $\mu (\partial D)=0$, for every $\e >0$, there is some 
$n_0\geq 1$ so that $\int_{\R^d} \chi_{n_0} d \mu <\e$.
Since, by \Cref{weakConvergeRelation},  $\mu_r$ converges vaguely to $\mu$ on $\R^d$ as $r\downarrow0$,
$\lim_{r\downarrow0} \int_{\R^d} \chi_{n_0} d \mu_r =  \int_{\R^d} \chi_{n_0} d \mu <\e$. 
Thus there is some $\delta \in (0, 1)$ so that 
\begin{equation}\label{e:3.4} 
\int_{\R^d} \chi_{n_0} d \mu_r  < 2\e \quad \hbox{for every } r\in (0, \delta).
\end{equation}

For each  $g\in C(\overline D)$, $g (1-\chi_{n_0} ) \1_D \in C_c (\R^d)$ so by the vague convergence of 
  $\mu_r$ to $\mu$  on $\R^d$ as $r\downarrow0$, we have  together with \eqref{e:3.4} that 
 \begin{eqnarray*}
&& \limsup_{r\downarrow0}  \Big| \int_D g  d \mu_r  - \int_D g   d \mu \Big| \\
&\leq&   \limsup_{r\downarrow0} 
\Big| \int_D g (1-\chi_{n_0}) d \mu_r  - \int_D g (1-\chi_{n_0})  d \mu \Big|
+  \limsup_{r\downarrow0} \| g\|_\infty \Big(\int \chi_{n_0} d\mu_r + \int \chi_{n_0} d \mu \Big)  \\
&\leq& 0 + 3 \e \| g\|_\infty. 
\end{eqnarray*}
This establishes $\lim_{r\downarrow0} \int_D g  d \mu_r  = \int_D g   d \mu$.

(ii) First fix $g\in C(\overline D)$. Choose a sequence $\{v_m;m\geq1\}\subset C_c^\infty(D)$
so that $\lim_{m\to\infty}\|v_m-g\|_{\mathcal H^D}=0$.
Since $|v_m-g|^2\in C(\overline D)$, (i) gives
\[
\lim_{r\downarrow0}\|v_m-g\|_{\mathcal H_r^D}
=\|v_m-g\|_{\mathcal H^D}.
\]
Thus  $g\in\mathcal H_r^D$ converges strongly to
$g\in\mathcal H^D$ as $r\downarrow0 $  in the sense of \Cref{strong convergence}, using the
original core $C_c^\infty(D)$. Moreover,
\[
\limsup_{r\downarrow 0} \|g_r-g\|_{\mathcal H_r^D}
\leq \lim_{r\downarrow 0} \|g_r-g\|_{\infty,D} \, \mu_r(D)^{1/2}
= 0 , 
\]
which proves (ii).

(iii) Fix an arbitrary sequence $\{r_n;n\geq1\}$ of positive numbers decreasing to $0$. Let $h_{r_n}\in\mathcal H_{r_n}^{\Rd}$ converge strongly to
$h\in\mathcal H^{\Rd}$ in the sense of \Cref{strong convergence}, and let
$\{\psi_m\}\subset C_c^\infty(\Rd)$ be its approximating sequence.
The indices $n$ and $m$ are independent: the limit in $n$ is taken first, as in \Cref{strong convergence}.
Then   
\[
\|h_{r_n}|_D-\psi_m|_D\|_{\mathcal H_{r_n}^D}
\leq\|h_{r_n}-\psi_m\|_{\mathcal H_{r_n}^{\Rd}}.
\]
The restrictions $\psi_m|_D$ need not belong to $C_c^\infty(D)$, so we
construct approximants in this core explicitly. For each $m\geq1$, choose
$v_m\in C_c^\infty(D)$ such that
\[
\|v_m-\psi_m|_D\|_{\mathcal H^D}<\frac1m.
\]
By (ii) applied to $v_m-\psi_m|_D\in C(\overline D)$,
\[
\lim_{n\to\infty}\|v_m-\psi_m|_D\|_{\mathcal H_{r_n}^D}
=\|v_m-\psi_m|_D\|_{\mathcal H^D}<\frac1m.
\]
Consequently,
\begin{align*}
 \limsup_{m\to \infty} \|v_m-h|_D\|_{\mathcal H^D}
&\leq  \lim_{m\to \infty}  \Big( \frac1m+\|\psi_m-h\|_{\mathcal H^{\Rd}} \Big) 
= 0,\\
\limsup_{n\to\infty}\|v_m-h_{r_n}|_D\|_{\mathcal H_{r_n}^D}
&\leq\frac1m+
\limsup_{n\to\infty}\|\psi_m-h_{r_n}\|_{\mathcal H_{r_n}^{\Rd}}. 
\end{align*}
  The right-hand side of the last display tends to zero as $m\to\infty$ by \Cref{strong convergence}
  and that $\{\psi_m\}\subset C_c^\infty(\Rd)$ is an approximating sequence for 
  the strong convergence of $h_{r_n}\in\mathcal H_{r_n}^{\Rd}$  to
$h\in\mathcal H^{\Rd}$. 
   Thus $h_{r_n} |_D \in \mathcal H_{r_n}^D$ converges strongly to $h|_D\in \mathcal H^D$
in the sense of  \Cref{strong convergence} with  $\{v_m\} \subset C_c^\infty(D)$
as   the approximating sequence.

Now suppose $u_r\in\mathcal H_r^D$ converges weakly to $u\in\mathcal H^D$ as $r\downarrow 0$.
For every strongly convergent sequence $h_{r_n}\in\mathcal H_{r_n}^{\Rd}$ with limit
$h\in\mathcal H^{\Rd}$, we have
\[
\lim_{n\to\infty} (\bar u_{r_n},h_{r_n})_{\mathcal H_{r_n}^{\Rd}}
= \lim_{n\to\infty} ( u_{r_n},h_{r_n}|_D)_{\mathcal H_{r_n}^D}
= (u,h|_D)_{\mathcal H^D}
= (\bar u,h)_{\mathcal H^{\Rd}}.
\]
This gives  the weak convergence
of  $\bar u_{r_n}\in\mathcal H_{r_n}^{\R^d}$ to $\bar u \in \mathcal H^{\Rd}$
 in the sense of \Cref{weak convergence}.
Since $r_n\downarrow0$ was arbitrary, both assertions hold as $r\downarrow0$.
\end{proof}

We are now ready to prove the Mosco convergence of the Dirichlet forms.

\begin{prop} \label{mosco}
Let $D$ be a bounded open set with $C^{1,1}$ boundary and $\mu(\6 D)=0$.
The Dirichlet form $(\mathcal F^{r,D},\mathcal E^{r,D})$ on $L^2(D;\mu_r|_D)$ (resp. $({\mathcal F}^{r},{\mathcal E}^{r})$ on $L^2(\R^d;\mu_r)$) is Mosco convergent to $(\mathcal F^{\mu,D},\mathcal E^{\mu,D})$ on $L^2(D;\mu|_D)$ (resp. $({\mathcal F^{\mu}},{\mathcal E})$ on $L^2(\R^d;\mu)$) as $r\downarrow0$. Hence the corresponding semigroups converge in the sense of \Cref{operator convergence}.
\end{prop}

\begin{proof}
	We focus on establishing Mosco convergence.

Take any sequence $\{r_n;n\geq1\}$ of positive numbers decreasing to $0$. We only need to show the convergence holds along any such sequence. With slight abuse of notation, replace the subscript/superscript $r_n$ in $\mathcal F^{r_n},$ 
$\mathcal E^{r_n},$ $ Q_{r_n},$ $\mu_{r_n}$, etc by $n$ to denote the corresponding objects. 
Let $u_n\in L^2(\mu_n)$ converge weakly to $u\in L^2(\mu)$ as $n\to\infty$ in the sense of \Cref{weak convergence}. Assume that $u_n\in\mathcal F^n$ and $\sup_n \mathcal E^{n}(u_n)<\8$. Set $v_n=\phi_n*u_n$. By \Cref{energy bounded}, $v_n\in\mathcal F^\mu$. 
By \Cref{weakConvergeRelation},   $v_n$ converges weakly to $u$ in $L^2(\mu)$ as $n\to \infty$. 
By \Cref{form mollifier monotone}, \Cref{energy bounded} and choosing a subsequence (still denoted as $\{u_{n}\}$ and $\{ v_n\}$), we may assume that  $\limsup_{n\to\infty}\mathcal E^{n}(v_n)\<\liminf_{n\to\infty}\mathcal E^{n}(u_n)<\8$ and $\sup_n\mathcal E(v_n)<\8$. By the 
weak convergence of $v_n$ to $u$ in $L^2(\mu)$,  we have $\sup_n\|v_n\|_{L^2(\mu)}<\8$. Hence $\{v_n; n\geq 1\}$ is a bounded sequence in $\mathcal F^{\mu}$ with respect to $\mathcal E_{1}^{\mu}$. Then by the Banach-Saks theorem, we can find a subsequence of $\{v_n; n\geq 1\}$
 (still denoted as $\{v_n; n\geq 1\}$) such that $w_n:=(\sum_{i=1}^nv_i)/n$ converges strongly to 
 some
  $\~u\in \mathcal F^{\mu}$ in $\mathcal E_{1}^\mu$ as $n\to\infty$. As $w_n$ converges to $u$ weakly in $L^2(\mu)$, 
  $\~u = u$ $\mu$-a.e.
   by the uniqueness of weak limits in $L^{2}(\mu)$. Thus $u\in\mathcal F^\mu$.
For each $k\geq1$, \Cref{inW12 form monotone}(ii) gives
$\limsup_{n\to\infty}\mathcal E^k(w_n-u)\leq\lim_{n\to\infty}\mathcal E(w_n-u)=0$.
Thus, by the triangle inequality,
\begin{equation*}
	\mathcal E^{k}( u)
	=\lim_{n\to \8}\mathcal E^{k}(w_n)
	\<\liminf_{n\to\8} \left( \frac{1}{n}\sum_{i=1}^{n} \sqrt{\mathcal  E^{k}(v_i)} \right)^{2}
	\<\limsup_{n\to\8}\mathcal E^{k}(v_n).
\end{equation*}
Then by \Cref{form monotone 2} and \Cref{form mollifier monotone}, for each $k\geq 1$, 
\begin{equation*}
	\limsup_{n\to\8}\mathcal E^{k}(v_n)
	\<\limsup_{n\to\8}\mathcal E^{n}(v_n)
	\<\liminf_{n\to\8}\mathcal E^{n}(u_n).
\end{equation*}
Hence by \Cref{inW12 form monotone} we know $\mathcal E( u)\<\liminf_{k\to\8}\mathcal E^{k}( u)\<\liminf_{n\to\8}\mathcal E^{n}(u_n)$. This establishes the first property for the Mosco convergence of $({\mathcal F}^{n},{\mathcal E}^{n})$ on $L^2(\R^d;\mu_n)$ to
$({\mathcal F^{\mu}},{\mathcal E})$ on $L^2(\R^d;\mu)$.

For the second property of the Mosco convergence, 
by \cite[Theorem 5.1.6]{CF12} we know $C^\8_c(\R^d)$ is $\mathcal E_{1}^{\mu}$-dense in $\mathcal F^\mu$. 
For $u\in \mathcal F^{\mu}$ we can choose a sequence $\{u_n;n\geq1\}\subset C^\8_c(\R^d)$ that $\mathcal E_{1}^{\mu}$-converges to $u$ as $n\to\infty$. 
By \Cref{make convergence sequence} we can find a sequence $v_n=u_{k_n}$, with $k_n\uparrow\infty$ nondecreasing, that converges to $u$ as $n\to\infty$ in the sense of \Cref{strong convergence}. Note that $v_n$ also $\mathcal E$-converges to $u$ as $n\to\infty$.
By \Cref{inW12 form monotone} we have $\limsup_{n\to \8} \mathcal E^{n}(v_n)\<\limsup_{n\to \8} \mathcal E(v_n)=\mathcal E(u)$.

For $(\mathcal F^{n,D}, \mathcal E^{n,D})$,  the proof is similar.
For the first property of the Mosco convergence,
 let $u_n\in L^2(D;\mu_n|_D)$ converge weakly to $u\in L^2(D;\mu|_D)$ as $n\to\infty$ in the sense of \Cref{weak convergence}. Assume that $u_n\in\mathcal F^{n,D}$ and $\sup_n\mathcal E^{n,D}(u_n)<\8$, and extend these functions by zero on  $\R^d \setminus D$.
Under such an extension, $u_n \in \mathcal F^{n,D} \subset \mathcal F^{n}$ and 
$\mathcal  E^{n}(u_n) =\mathcal  E^{n,D}(u_n)$. 
 By \Cref{boundary restriction}(iii),   $u_n\in L^2(\R^d;\mu_n )$  converges weakly to $u\in L^2(\R^d;\mu )$. So by \Cref{weakConvergeRelation},  $v_n:=\phi_n*u_n$ converges weakly to $u$ in $L^2(\R^d; \mu)$.
 Applying the preceding argument,
 we obtain $ u\in\mathcal F^\mu$  and the Cesaro averages $w_n$ of a subsequence of $\{v_n\}$ that converge to $ u$ in $\mathcal E_1^\mu$ as $n\to\infty$. Representing $u$ by  its quasi-continuous $\mu$-version $\~ u$, a subsequence $\{w_{n_k};k\geq1\}$ of these averages converges to $\~u$ $\sE$-q.e. as $k\to\infty$ by \cite[Theorem 1.3.2]{CF12}. Since $\phi_n$ is supported in a ball of radius tending to zero and the zero extensions of $u_n$ vanish on $D^c$, we have $\lim_{n\to\infty}w_n(x)=0$ for every $x\notin\overline D$.
Thus $\~u=0$ $\sE$-q.e. on $\R^d \setminus \overline D$. Since $D$ is $C^{1,1}$, 
  each boundary point of $D$ is a regular point of $\R^d \setminus \overline D$ 
 with respect to Brownian motion. It follows from the right continuity of $t \mapsto \~u(W_t)$ for $\sE$-q.e. starting point that
 $\~ u (x) =0$ for $\sE$-q.e. $x\in  \partial D$ and so for $\sE$-q.e. $x\in D^c$.  
By \eqref{dirichlet form mu E}, $u= \~u\in\mathcal F^{\mu,D}$.      
   Therefore
\begin{equation} \label{e:3.5}
{\mathcal E^{\mu,D}}( u)={\mathcal E}( u)
\< \limsup_{n\to \8}  \mathcal E^{n}(v_n)
\<\liminf_{n\to \8} \mathcal E^{n}(u_n)
=\liminf_{n\to \8} {\mathcal E}^{n,D}(u_n).
\end{equation} 

For the second property of the Mosco convergence, fix $u\in\mathcal F^{\mu,D}$. Since
 $C^\8_c(D)$ is $\mathcal E_{1}^{\mu, D}$-dense in $ \mathcal F^{\mu, D}$,
 by \Cref{make convergence sequence} we can find a sequence $u_n\in C^\8_c(D) \inn L^2(D;\mu_n)$ that converges in the sense of \Cref{strong convergence} to $u\in L^2(D;\mu)$ and in
 $\mathcal E_1^{\mu,D}$ as $n\to\infty$.
 Hence
$$\limsup_{n\to \8} \mathcal E^{n,D} (u_n)=\limsup_{n\to \8} \mathcal E^{n}(u_n)\<\limsup_{n\to \8} \mathcal E(u_n)=\mathcal E(u)= \mathcal E^{\mu,D}(u).
$$
This completes the proof.
\end{proof}

\subsection{Convergence in the Skorokhod topology}

In this section, we use the semigroup convergence obtained above to prove the finite-dimensional convergence of the stopped processes, 
 and then deduce by a  tightness result of Aldous \cite{aldous1989stopping} 
 and a localization argument    the weak convergence of the time-changed random walks in the Skorokhod topology.  

\medskip

For an $\mathbb{R}^d_\partial$-valued path $\Gamma$ that is right-continuous and has left limits (rcll)
 indexed by $I=[0,\8)$ or $I=\{0,1,2,\ldots\}$, define
 $$
 \tau_D(\Gamma):=\inf\{s\in I:\Gamma_s\notin D\}
 $$
for $D\subset\R^d$, and write $\tau_R:=\tau_{B(0,R)}(\Gamma)$.
We first choose radii for which restriction to the corresponding balls is compatible with the vague convergence of the measures.

\begin{lemma} \label{R sequence}
	There exists an unbounded increasing sequence   $\{R_k; k\geq 1\}$ of positive numbers such that
	$\mu(\6 B(0,R_k))=0$ for every $k$.
\end{lemma}

\begin{proof}
	The spheres $\6 B(0,R)$ are pairwise disjoint, and $\mu$ is finite on bounded sets. Hence at most countably many radii $R>0$ satisfy $\mu(\6 B(0,R))>0$, which proves the assertion.
\end{proof}

Take an unbounded increasing sequence $\{R_k; k\geq1\}$ of positive numbers as in \Cref{R sequence}.
Let $R \in \{R_{k}\}_{k}$ be large enough such that $f_{0} = 0$ outside $B(0,R)$, where $f_{0}$ is from \Cref{mainthm1}, and set $D:=B(0,R)$.
Throughout the rest of this section, we fix this $R$ and suppress the dependence on $R$ in the notation. In particular, we denote $\tau(\Gamma) := \tau_{R}(\Gamma)$ for a path $\Gamma$.

\medskip

For the function $f_0$ in \Cref{mainthm1}, recall that $Z_r=\int f_0\,d\mu_r$, $Z=\int f_0\,d\mu$, $\nu_r=Z_r^{-1}f_0\cdot\mu_r$, and $\nu=Z^{-1}f_0\cdot \mu$. By \Cref{weakConvergeRelation}, $\lim_{r\downarrow 0} Z_r = Z>0$. 

 \medskip
 
Let $Y^r$ (resp. $Y$) be the stopped process of $X^r$ (resp. $X$) upon leaving $D$; that is,
$$
Y^r_{\cdot}:=X^r_{\cdot\wedge\tau(X^{r})} \quad \hbox{and} \quad  Y_{\cdot}:=X_{\cdot\wedge\tau(X)}.
$$ 

\medskip

We first prepare a lemma for \Cref{stopped}.

\begin{lemma} \label{uint}
	Random vectors $\{X^r_{\tau(X^{r})}\}_{r\in(0,1]}$ are  uniformly integrable under $\P_{x}$ for any $x\in B(0,R)$. 
	Consequently, for each $t\>0$, $\{Y_{t}^{r}\}_{r\in(0,1]}$ under $\P_{\nu_r}$   are  uniformly integrable.
\end{lemma}

\begin{proof}
	Recall that $\{W_{k}^{r}; k\geq 0\}$ is the discrete time random walk defined in \eqref{e:1.5} with scaling $r>0$. Let $N_r=\inf\{k: |W_k^r|\> R\}$. Then $|Y_t^r|\<|X^r_{\tau(X^{r})}| = |W^r_{N_r}|
	\< R+r|\xi_{N_r}|$. We only need to show $r|\xi_{N_r}|$ is uniformly integrable. For $a>0$ we have
\begin{align*}
\EX_x\left[r|\xi_{N_r}|\1_{\{r|\xi_{N_r}|>a\}}\right]
&\<
\EX_x\sum_{k=1}^{N_r} r|\xi_k|\1_{\{r|\xi_k|>a\}}  \\
&= \sum_{k\> 1} \EX_x[ r|\xi_k|\1_{\{r|\xi_k|>a\}}; N_r\> k] \\
&=
\sum_{k\ge1}\P_x(N_r\> k)
\EX\left[r|\xi_1|\1_{\{r|\xi_1|>a\}}\right] \\
&=
\EX_x [ N_r ]\,  \EX\left[r|\xi_1|\1_{\{r|\xi_1|>a\}}\right].
\end{align*}
Here we used the fact that $\{N_r\ge k\}$ depends only on $\xi_1,\dots,\xi_{k-1}$, so it is independent of $\xi_k$. 
By the central limit theorem, there exist $0<p<1$ and $A\geq 1$ such that $W^{r}$ has probability at least $p$ to move distance at least $2R$ within $\lceil A R^2/r^2\rceil$ steps, where $\lceil a\rceil$ denotes the least integer greater than or equal to $a\in\R$. 
 Set $m:=\lceil A R^2/r^2\rceil$. The Markov property gives $\mathbb P_x(N_r>jm)\le (1-p)^j$ and therefore $\EX_{x}N_r\< C R^{2}/r^{2}$ for some $C>0$. Hence
\begin{align*}
	\EX_x N_r \EX\left[r|\xi_1|\1_{\{r|\xi_1|>a\}}\right]
	\< CR^{2} \EX\left[|\xi_1|/r\1_{\{|\xi_1|>a/r\}}\right] 
	\< CR^{2}a^{-1} \EX\left[|\xi_1|^{2}\1_{\{|\xi_1|>a/r\}}\right]
\end{align*}
which tends to 0 as $a\uparrow\8$. Hence $\{r|\xi_{N_r}|; r\in (0, 1]\}$ is uniformly integrable and so is $\{Y_{t}^{r}; r\in (0, 1 ] \}$.
\end{proof}

We next establish convergence of the exit distributions, which will be used to prove convergence of the stopped processes. Recall that, under $\P_x$,  
$
W_n^r=x+r\sum_{i=1}^n\xi_i,
$
where the $\xi_i$ are i.i.d. with density $\jd$. Since time changes do not alter the trajectories of the processes, it suffices to compare the exit positions of $W^r$ and Brownian motion.
Here $W$ denotes standard Brownian motion, started from $x$ under $\P_x$.

\begin{prop}\label{hitting converge}
	For $f\in C_b(\Rd)$ and $x\in\overline D$, set
	\[
	H_rf(x):=\EX_x\big[f(X^r_{\tau(X^r)})\big],
	\qquad
	Hf(x):=\EX_x\big[f(W_{\tau(W)})\big].
	\]
	Then $Hf$ is continuous on $\overline D$, and whenever $\lim_{r\downarrow0}x_r=x\in\overline D$,
	$
	\lim_{r\downarrow0} H_rf(x_r)=  Hf(x).
	$
	Consequently, $H_rf$ converges to $Hf$ uniformly on $\overline D$ as $r\downarrow0$.
\end{prop}

\begin{proof}
(i) Since time changes preserve exit positions,
\[
H_rf(x)=\EX_x\big[f(W^r_{\tau(W^r)})\big],\qquad x\in\overline D,
\]
and $Hf(x)=\EX_x[f(X_{\tau(X)})]$ for $\mathcal E$-q.e. $x\in D$.
  
 Recall that $D=B(0,R)$ and
$\tau(\Gamma)=\inf\{t\geq0:\Gamma_t\notin D\}$.
We first prove continuity of  $\tau(\Gamma)$  and $\Gamma_{\tau(\Gamma)}$   
for almost every Brownian path.  
For a continuous path $\Gamma$ with $\tau:=\tau(\Gamma)<\infty$, suppose that
for every $\varepsilon>0$ there exists $s\in[\tau,\tau+\varepsilon)$ with
$\Gamma_s\notin\overline D$. If a sequence $\{\Gamma_n;n\geq1\}\subset\DD([0,\infty);\Rd)$ converges locally uniformly to $\Gamma$ as $n\to\infty$,
then, for each $a<\tau$, the compact set $\Gamma([0,a])$ has positive
distance from $D^c$. Hence $\tau(\Gamma_n)>a$ for all sufficiently large
$n$, and $\liminf_{n\to\infty}\tau(\Gamma_n)\geq\tau$ (also trivially
when $\tau=0$). For $s$ as above, the positive distance of $\Gamma_s$ from
$\overline D$ implies $\Gamma_n(s)\notin D$ for large $n$, so
$\limsup_{n\to\infty}\tau(\Gamma_n)\leq s<\tau+\varepsilon$.
Letting $\varepsilon\downarrow0$ proves convergence of the exit times.
For any $T>\tau$, eventually $\tau(\Gamma_n)<T$, and
\[
\limsup_{n\to\infty}|\Gamma_n(\tau(\Gamma_n))-\Gamma(\tau)|
\leq\lim_{n\to\infty}\Big(\sup_{0\leq t\leq T}|\Gamma_n(t)-\Gamma(t)|
+|\Gamma(\tau(\Gamma_n))-\Gamma(\tau)|\Big)=0.
\]

Since $D$ is a ball, every $z\in\partial D$ is regular for
$\Rd\setminus\overline D$; the strong Markov property at $\tau(W)$ thus
implies that Brownian paths almost surely enter this open exterior
arbitrarily soon after exit.
Thus both the exit-time and exit-position maps are continuous under the
locally uniform convergence topology for almost every Brownian path.

(ii) Let $W^0$ be Brownian motion starting from the origin. To prove continuity of $Hf$,   take
a sequence $\{x_n;n\geq1\}\subset\overline D$ converging to $x\in\overline D$ as $n\to\infty$. The paths $x_n+W^0$ converge locally uniformly
to $x+W^0$, so their exit positions converge almost surely by the preceding
argument. Thus by the  bounded convergence theorem
we have $\lim_{n\to \infty}  Hf(x_n) = Hf(x)$.

(iii) Define
\[
S_t^r:=W^r_{\lfloor t/(\sigma_{\jd}^2r^2)\rfloor},\qquad t\geq0,
\]
where $\lfloor a\rfloor$ denotes the greatest integer less than or equal to
$a\in\R$. By the multidimensional Donsker invariance principle
(see, e.g., \cite[Theorem 4.3.5]{Whi02}), $S^r$ under $\P_{x_r}$ converges
in distribution, as $r\downarrow0$, to $W$ under $\P_x$ in
$\DD([0,\infty);\Rd)$ under the Skorokhod $J_1$ topology.
Convergence in this topology to a continuous path implies locally uniform
convergence.  
As $H_rf(x)=\EX_x\big[f( W^r_{\tau( W^r)})\big] =\EX_x\big[f( S^r_{\tau( S^r)})\big]$, 
we have by (i) and  the continuous mapping theorem  that 
$\lim_{r\downarrow 0} H_rf(x_r)=  Hf(x)$.

To obtain the uniform convergence, suppose otherwise. Then there would exist $\e>0$, a sequence $\{r_n;n\geq1\}$ of positive numbers decreasing to $0$, and a sequence $\{x_n;n\geq1\}\subset\overline D$ such that $|H_{r_n}f(x_n)-Hf(x_n)|\geq\e$. By compactness, there is  a subsequence $\{x_{n_k}\}$ that converges to some $x\in\overline D$. The convergence just proved gives $\lim_{k\to \infty} H_{r_{n_k}}f(x_{n_k})= Hf(x)$, while the continuity of $Hf$ gives $\lim_{k\to \infty} H f(x_{n_k}) = Hf(x)$, which is a contradiction. This proves that  $H_rf$ converges to $Hf$ uniformly on $\overline D$ as $r\downarrow0$.
\end{proof}

\begin{lemma}\label{localized observables}
	For every $t>0$,
	$ P_t^{r,D}\1_D \in \mathcal H_r^D $ converges to  $P_t^D\1_D \in \mathcal H^D$ 
 	strongly   as $r\downarrow0$. 
	Moreover, for $0<t_1<\cdots<t_k$ and $f_1,\ldots,f_k\in C_b(\Rd)$, define
	\[
	F_r(x):=\EX_x \Big[\prod_{i=1}^k f_i(Y^r_{t_i})\Big],
	\qquad
	F(x):=\EX_x\Big[ \prod_{i=1}^k f_i(Y_{t_i}) \Big].
	\]
	Then $F_r|_D\in \mathcal H_r^D $  converges to $ F|_D  \in \mathcal H^D $ strongly as $r\downarrow0$.
	 
\end{lemma}

\begin{proof}
	By \Cref{boundary restriction}(ii), the constant function $\1_D\in\mathcal H_r^D$ converges strongly to $\1_D\in\mathcal H^D$ as $r\downarrow0$. Hence the first assertion follows from the semigroup convergence in \Cref{mosco}.

	Let $P_t^{r,D}$ and $P_t^D$ be the semigroups of the killed processes, and let $\widehat P_t^r$ and $\widehat P_t$ be the semigroups of $Y^r$ and $Y$, respectively. For $f\in C_b(\Rd)$ and $x\in D$, the Markov property gives
	\begin{align*}
	H_rf(x)
	&=P_t^{r,D}H_rf(x)
	  +\EX_x\big[f(X^r_{\tau(X^r)});\,\tau(X^r)\leq t\big],\\
	\widehat P_t^rf(x)
	&=P_t^{r,D}f(x)
	  +\EX_x\big[f(X^r_{\tau(X^r)});\,\tau(X^r)\leq t\big].
	\end{align*}
	Consequently,
	\begin{equation}\label{stopped semigroup}
	\widehat P_t^rf=H_rf+P_t^{r,D}(f-H_rf)\quad\hbox{on }D.
	\end{equation}
	Similarly,
	\[
	\widehat P_tf=Hf+P_t^D(f-Hf)\quad\mathcal E\text{-q.e. on }D.
	\]
	By \Cref{hitting converge}, $H_rf$ converges to $Hf$ uniformly on $\overline D$ as $r\downarrow0$.  
	Hence \Cref{boundary restriction}(ii), \Cref{mosco}, and \eqref{stopped semigroup} imply the strong convergence of $\widehat P_t^rf|_D\in\mathcal H_r^D$ to $\widehat P_tf|_D\in\mathcal H^D$ as $r\downarrow0$.

	To obtain the assertion for cylinder functions, set $\psi_k^r=f_k$ and define recursively
	\[
	\psi_j^r=f_j\widehat P_{t_{j+1}-t_j}^r\psi_{j+1}^r,
	\qquad j=k-1,\ldots,1,
	\]
	and define $\psi_j$ analogously using $\widehat P_t$. Outside $D$, the stopped process is constant, so
	$\psi_j^r=\prod_{i=j}^k f_i=\psi_j$ on $D^c$.
	We claim by backward induction that $\psi_j^r|_D\in\mathcal H_r^D$ converges strongly to $\psi_j|_D\in\mathcal H^D$ as $r\downarrow0$. This holds for $j=k$ by \Cref{boundary restriction}(ii). Suppose it holds for $j+1$, and put $s_j=t_{j+1}-t_j$ and $\varphi_{j+1}:=\prod_{i=j+1}^k f_i$. By \eqref{stopped semigroup},
	\[
	\widehat P_{s_j}^r\psi_{j+1}^r
	=H_r\varphi_{j+1}
	 +P_{s_j}^{r,D}(\psi_{j+1}^r-H_r\varphi_{j+1})
	\quad\hbox{on }D.
	\]
	The induction hypothesis, \Cref{hitting converge}, and \Cref{mosco} imply that $(\widehat P_{s_j}^r\psi_{j+1}^r)|_D\in\mathcal H_r^D$ converges strongly to $(\widehat P_{s_j}\psi_{j+1})|_D\in\mathcal H^D$ as $r\downarrow0$ in the sense of \Cref{strong convergence}.
	Multiplication by $f_j$ preserves this convergence: approximate a core function times $f_j\in C(\overline D)$ using (ii) of \Cref{boundary restriction}, and bound the remaining error by $\|f_j\|_\infty$ times the original approximation error. This proves the claim.

	Applying the same argument once more at time $t_1$ shows that $F_r|_D=\widehat P_{t_1}^r\psi_1^r|_D\in\mathcal H_r^D$ converges strongly to $\widehat P_{t_1}\psi_1|_D=F|_D\in\mathcal H^D$ as $r\downarrow0$ in the sense of \Cref{strong convergence}.
\end{proof}

\begin{prop} \label{stopped}
	The stopping time $\tau(X^r)$ under $\P_{\nu_r}$ converges in law to $\tau(X)$ under $\P_{\nu}$ as $r \downarrow 0$.
	The stopped process $Y^r$ under $\P_{\nu_r}$ converges in law to $Y$ under $\P_{\nu}$ as $r\downarrow0$ in
	$\DD([0,\8 );\R^d)$ under the uniform convergence topology over the time interval $[0,\infty)$.
\end{prop}

\begin{proof}
By \Cref{boundary restriction}(ii), the constant family
$f_0|_D\in\mathcal H_r^D$ converges strongly to $f_0|_D\in\mathcal H^D$
as $r\downarrow0$ in the sense of \Cref{strong convergence}. Hence \Cref{localized observables} gives
	\[
	\lim_{r\downarrow0}\P_{\nu_r}(\tau(X^r)>t)
	=\lim_{r\downarrow0}Z_r^{-1}\langle f_0,P_t^{r,D}\1_D\rangle_{\mu_r}
	=
	Z^{-1}\langle f_0,P_t^D\1_D\rangle_\mu
	=\P_\nu(\tau(X)>t).
	\]
	Thus the survival functions converge, which proves that $\tau(X^r)$ converges in law to $\tau(X)$ as $r\downarrow0$. 
	By the second assertion of \Cref{localized observables},  for any $0<t_1<\cdots<t_k$ and $f_1,\ldots,f_k\in C_b(\Rd)$,
	\[
	\lim_{r\downarrow0}\EX_{\nu_r} \! \Big[\prod_{i=1}^k f_i(Y^r_{t_i}) \Big]
	=\lim_{r\downarrow0}Z_r^{-1}\langle f_0,F_r\rangle_{\mu_r}
	=
	Z^{-1}\langle f_0,F\rangle_\mu
	=\EX_\nu\! \Big[ \prod_{i=1}^k f_i(Y_{t_i}) \Big].
	\]
	This shows that the finite-dimensional distributions of the stopped process $Y^r$ converge weakly to those of $Y$ as $r\downarrow0$. 

	By \Cref{uint} and \cite[Proposition 1.2]{aldous1989stopping}, the finite-dimensional convergence implies that $Y^r$ converges in law to $Y$ as $r\downarrow0$ in the Skorokhod topology on $\DD([0,\8);\Rd)$. Since $Y$ has continuous paths, the convergence holds in the topology of locally uniform convergence.

	It remains to extend the convergence to the one in the   uniform topology on $[0,\infty)$. For a path $\omega$, let $\pi_T\omega(t):=\omega(t\wedge T)$. Local uniform convergence implies the law of $\pi_TY^r$ converges weakly to that of $\pi_TY$ as $r\downarrow0$ in the uniform topology on $[0,\infty)$ for each fixed $T$. Moreover, the exit times converge in law to the almost surely finite random variable $\tau(X)$, while $Y^r$ and $Y$ are constant after their respective exit times. Hence, given $\varepsilon>0$, we can choose a continuity point $T$ of the law of $\tau(X)$ so large that
	\[
	\limsup_{r\downarrow0}\P_{\nu_r}(Y^r\ne\pi_TY^r)
	\leq\limsup_{r\downarrow0}\P_{\nu_r}(\tau(X^r)>T)<\varepsilon
	\]
	and $\P_\nu(Y\ne\pi_TY)<\varepsilon$. Letting $\varepsilon\downarrow0$ proves the desired convergence in the uniform topology on $[0,\infty)$.
\end{proof}

To extend the convergence of stopped processes $Y^r$ to $Y$ to the convergence of whole processes $X^{r}$ to $X$, we need the following lemma.

\begin{lemma}\label{uniform transience}
	Suppose that $d\geq 3$. For any $M>0$ and $R\>2M$, there exists a constant $C>0$ such that
	\begin{equation}\label{uniform transience estimate}
		\limsup_{r\downarrow0}\sup_{|y|\geq R}
		\P_y\big(W^r\text{ hits }B(0,M)\big)
		\leq C\left(\frac MR\right)^\gamma
		\quad \hbox{where } \gamma=\frac{2(d-2)}d.
	\end{equation}
\end{lemma}

\begin{proof}
	Under $\P_y$, write
	$W_k^r=y+rS_k$, $S_k:=\sum_{i=1}^k\xi_i$,
 	and let  $H:=\inf\{k\geq0:W_k^r\in B(0,M)\}$. We first note that, for some
	$a_0,C>0$,
	\begin{equation}\label{concentration Sn}
		\sup_{z\in\Rd}\P(S_k\in B(z,a))\leq Ca^dk^{-d/2},
		\qquad a\geq a_0,\quad k\geq1.
	\end{equation}
	Indeed, let $\varphi(x)=\EX[e^{\mathrm{i}x \cdot\xi_1}]$ for $x\in \R^d$. The multidimensional
	Ess\'een inequality \cite[Lemma~1]{GoetzeZaitsev2022} gives
	\[
		\sup_{z\in \R^d} \P(S_k\in B(z,a))
		\leq C_da^d\int_{|t|\leq c_d/a}|\varphi(t)|^k\,dt.
	\]
	Since $\EX\xi_1=0$ and $\operatorname{Cov}(\xi_1)=\sigma_{\jd}^2I$, by Taylor expansion,
	there is some $a_0>0$ so that  
	$|\varphi(x)|\leq e^{-c|x|^2}$ for $|x| \leq c_d/a_0$. 
	Hence, for any $a\geq a_0$, the above display gives
	$$	\sup_{z\in \R^d} \P(S_k\in B(z,a))
		\leq C_da^d\int_{\R^d}  e^{-ck |x|^2} \,dx \leq Ca^dk^{-d/2},
		$$
	which proves \eqref{concentration Sn}.

	Fix an integer $N\geq1$. If $H\leq N$ and $|y|\geq R$, then
	$$
	\max_{k\leq N} |rS_k| 
	\geq |y| - \inf_{ k\leq N } | W_k^r| 
	\geq R-M.
	$$  
	Since $R\geq2M$, this together with Doob's $L^2$-maximum inequality gives
	\begin{equation}\label{early hit estimate}
		\sup_{|y|\geq R}\P_y(H\leq N)
		\leq \frac{ \EX \left[ \max_{k\leq N} |rS_k|^2 \right]} { (R-M)^2} 
		\leq \frac{4r^2\EX |S_N|^2}{(R-M)^2}
		\leq C\frac{Nr^2}{R^2}.
	\end{equation}

	Now set $v:=\EX|\xi_1|^2$ and $\ell:=\left\lfloor\frac{M^2}{4vr^2}\right\rfloor$.
	For every $z\in B(0,M)$ and $0\leq j\leq\ell$, Chebyshev's inequality yields
	\[
		\P_z(W_j^r\in B(0,2M))
		\geq\P(|rS_j|<M)
		\geq 1 -\frac{ \EX [  |r S_j|^2] }{M^2} 
		\geq\frac34.
	\]
	Thus, for all sufficiently small $r$,
	\[
		\inf_{z\in B(0,M)}\EX_z \Big[ \sum_{j=0}^{\ell} \1_{B(0,2M)}(W_j^r) \Big]
		\geq 3 \ell /4 
		\geq cM^2/r^2.
	\]
	To use this estimate, set
	\[
		G_r(z):=\EX_z\Big[\sum_{j=0}^{\ell} 		\1_{B(0,2M)}(W_j^r) \Big].
	\]
	On $\{N<H<\8\}$ we have $W_H^r\in B(0,M)$, and the strong Markov
	property at $H$ gives
	\begin{align*}
		\EX_y \Big[\1_{\{N<H<\8\}}
		\sum_{j=0}^{\ell}\1_{B(0,2M)}(W_{H+j}^r) \Big]
		&=\EX_y\left[\1_{\{N<H<\8\}}G_r(W_H^r)\right]  	 \\
		& 		\geq c\frac{M^2}{r^2}\P_y(N<H<\8).
	\end{align*}
	Consequently, using
	\eqref{concentration Sn} with $a=2M/r$, we obtain
	\begin{align}
		\P_y(N<H<\8) 
		 		&\leq C\frac{r^2}{M^2}\sum_{k>N}\P_y(W_k^r\in B(0,2M))   \notag\\
			&\leq C\frac{r^2}{M^2}\left(\frac Mr\right)^d
		\sum_{k>N}k^{-d/2}
		\leq C\left(\frac Mr\right)^{d-2}N^{-(d-2)/2}. \label{late hit estimate}
	\end{align}
	Combining \eqref{early hit estimate} and \eqref{late hit estimate}, we have
	\begin{equation} \label{uniform transience estimate 2}
		\sup_{|y|\geq R}\P_y(H<\8)\leq C\left(\frac{Nr^2}{R^2}
		+\left(\frac Mr\right)^{d-2}N^{-(d-2)/2}\right).
	\end{equation}

	Put $q:=M/R$, $\theta:=q^\gamma$, and
	$N:=\lfloor\theta R^2/r^2\rfloor$. For all sufficiently small $r$, the two
	terms in the last display are bounded, respectively, by $C\theta$ and
	$Cq^{d-2}\theta^{-(d-2)/2}=Cq^\gamma$, because $\gamma d/2=d-2$. Taking $\limsup_{r\downarrow0}$ 
	on  both sides of \eqref{uniform transience estimate 2}
	proves 	\eqref{uniform transience estimate}.
\end{proof}

\begin{prop}\label{localization lemma}
	Let $\pi_r$ and $\pi$ be probability measures on $\R^d$. Suppose that
	there is an increasing sequence of positive numbers $\{R_n; n\geq 1\}$ tending to infinity 
	and satisfies $\mu(\6 B(0,R_n))=0$ for every $n\geq 1$, such that, for each fixed $n$, the stopped process
	\[
	Y^{r,{R_n}}:=X^r_{\cdot\wedge\tau_{R_n} (X^r)}
	\quad\hbox{under }\P_{\pi_r}
	\]
	converges in law in $\DD([0,\8);\Rd)$ as $r\downarrow0$ to
	\[
	Y^{R_n}:=X_{\cdot\wedge\tau_{R_n} (X)}
	\quad\hbox{under }\P_\pi.
	\]
	\begin{enumerate}
	\item[\rm (i)] If $X^r$ and $X$ are conservative and 
	$\tau_{R_n}(X^r)$ under $\P_{\pi_r}$ 	  converges in law to $\tau_{R_n}(X)$ 
	under $\P_\pi$   for each $n\geq 1$ as $r \downarrow 0$, 
	 then $X^r$ converges in law to $X$ in $\DD([0,\8);\Rd)$ as $r\downarrow0$. 
	 
	 \item[\rm (ii)]  If $d\geq3$, the process  $X^r$ converges in law to $X$ in $\DD([0,\8);\Rd_\partial)$ as $r\downarrow0$.  
	 \end{enumerate} 
\end{prop}

\begin{proof}
	(i) Suppose that the processes are conservative and  $\tau_{R_n}(X^r)$ under $\P_{\pi_r}$
	  converges in law to $\tau_{R_n}(X)$ under $\P_\pi$  for each $n\geq 1$ as $r \downarrow 0$.
	Fix $T>0$ and let $F$ be bounded and Lipschitz on $\DD([0,T];\Rd)$. Since $X^r$ and $Y^{r,{R_n}}$ agree on $[0,T]$ when $\tau_{R_n}(X^r)>T$,
	\begin{align*}
	\left|\EX_{\pi_r}F(X^r)-\EX_\pi F(X)\right|
	&\leq \left|\EX_{\pi_r}F(Y^{r,{R_n}})-\EX_\pi F(Y^{R_n})\right| 	\\
	&\quad +2\|F\|_\8\big(\P_{\pi_r}(\tau_{R_n}(X^r)\leq T)
	+\P_\pi(\tau_{R_n}(X)\leq T)\big).
	\end{align*}
	For fixed $n\geq1$, the first term tends to zero as $r\downarrow 0$
	by the stopped-process convergence, while the exit-time convergence gives
	\[
	\limsup_{r\downarrow0}\P_{\pi_r}(\tau_{R_n}(X^r)\leq T)
	\leq\P_\pi(\tau_{R_n}(X)\leq T).
	\]
	By conservativeness, the last probability tends to zero as $n\to\8$. This proves the claimed convergence on $\DD([0,T];\Rd)$, and hence on $\DD([0,\8);\Rd)$.

	(ii) Suppose   $d\geq3$. Fix $T,M>0$ and take $R_n>2M$. Denote $R_n$ by $R$  and set
	\begin{align*}
	\mathcal R^r_{M,R}&:=\{X_t^r\in B(0,M)\text{ for some }t\in[\tau_R(X^r),T]\},\\
	\mathcal R_{M,R}&:=\{X_t\in B(0,M)\text{ for some }t\in[\tau_R(X),T]\},
	\end{align*}
	where $[a,b]=\emptyset$ when $a>b$. Put $a_L:=(1+L^2)^{-1/2}$. On $(\mathcal R^r_{M,R})^c$,
	\begin{equation}\label{full stopped distance}
	\sup_{0\leq t\leq T}\mathsf d(X_t^r,Y_t^{r,R})\leq a_M+a_R,
	\end{equation}
	and the same estimate holds for $X$ and $Y^R$ on $(\mathcal R_{M,R})^c$. By the strong Markov property, \Cref{uniform transience}, and the Brownian hitting formula,
	\begin{align}
	\limsup_{r\downarrow0}\P_{\pi_r}(\mathcal R^r_{M,R})
	&\leq C\left(\frac MR\right)^\gamma,\label{return estimate random walk}\\
	\P_\pi(\mathcal R_{M,R})
	&\leq C\left(\frac MR\right)^{d-2}.
	\label{return estimate Brownian}
	\end{align}
	For the second estimate, see \cite[Corollary 3.19]{MortersPeres2010}.

	Let $F$ be bounded by $1$ and $1$-Lipschitz on $\DD([0,T];\Rd_\partial)$ with respect to the bounded $J_1$ metric induced by $\mathsf d$. The preceding estimates give
	\begin{align*}
	\limsup_{r\downarrow0}
	\left|\EX_{\pi_r}F(X^r)-\EX_\pi F(X)\right|
	&\leq \limsup_{r\downarrow0}
	\left|\EX_{\pi_r}F(Y^{r,R})-\EX_\pi F(Y^R)\right|\\
	&\quad+C(a_M+a_R)
	+C\left(\frac MR\right)^\gamma
	+C\left(\frac MR\right)^{d-2}.
	\end{align*}
	The first term is zero by the stopped-process convergence. Letting $n\to \infty$ so $R=R_n \to\8$  and then $M\to\8$ proves weak convergence on $\DD([0,T];\Rd_\partial)$. Since $T>0$ is arbitrary, the proof is complete.
\end{proof}

Now we can present the proof for \Cref{mainthm1}.

\begin{proof}[Proof of \Cref{mainthm1}]
	For every sufficiently large $R$ from  the  unbounded  increasing sequence   $\{R_k; k\geq 1\}$  of positive numbers
	in  \Cref{R sequence}, \Cref{stopped} gives the weak convergence of the stopped processes and of their exit times. If the processes are conservative, the first assertion of \Cref{localization lemma} applies to get the desired result. Otherwise,  $d\geq3$ by Remark \ref{rm2}(iv), and the second assertion of  \Cref{localization lemma} applies. In either case, $X^r$ under $\P_{\nu_r}$ converges in law to $X$ under $\P_\nu$ in $\DD([0,\8);\Rd_\partial)$ as $r\downarrow0$.
\end{proof}

\section[Approximation under individual starting points]{Approximation under individual starting\protect\\ points}\label{section thm2}

For pointwise starting convergence, we need regularity properties and estimates for the analytic counterparts (Green functions, harmonic functions, and transition semigroups) of the continuous-time random walk processes with or without time change.
  Recall that $W^{n}$ is the discrete time random walk and $\~W^n$  (resp. $X^{n}$) is the continuous time random walk without time change (resp. with time change) with scaling $r=r_{n}$ for some sequence $\{r_n;n\geq1\}$ of positive numbers decreasing to $0$.
Because the processes $\~W^n$ and $X^{n}$ started from a single point have positive probability of staying at that point for some time, their transition semigroups are not absolutely continuous with respect to Lebesgue measure. We need to work around this difficulty.
 
Let $W^{n,D}$ and $\~W^{n,D}$ be the processes $W^n$ and $\~W^n$ killed upon leaving $D$, respectively, and for $k\>1$ let $q_k^{n,D}(x,y)$ be the transition density of $W^{n,D}_k$. Note that $q_k^{n,D}(x,y)$ is symmetric in $x,y$. When $x$ is the origin, we denote the density by $q_k^{n,D}(y)$. If further $D = \Rd$ we denote $q_{k}^{n,D}$ by $q_{k}^{n}$. When $r_{n}=1$ we omit the superscript $n$ of $q$.

Let $g^{n,D}(x,y)$ be the occupation time density of $\~W^{n}$ from the second jump until exit from the domain $D\inn \Rd$, i.e., for any nonnegative $f\in C_b(\Rd)$, $g^{n,D}(x, y)$ satisfies
\begin{align*}
	\int f(y) g^{n,D}(x, y) dy
	&= \EX_x \left[\int_{\~\eta_0^{n} + \~\eta_1^{n}}^{\tau_D(\~W^{n}) \vee (\~\eta_0^{n} + \~\eta_1^{n}) }f(\~ W^{n}_t)dt ; W_{1}^{n}\in D \right] \\
	&= \sum_{k = 2}^{\8} \int_{0}^{\8} e^{- \sigma_{\jd}^{-2} r_{n}^{-2} t } \frac{(\sigma_{\jd}^{-2} r_{n}^{-2} t)^{k-2}}{(k-2)!} dt ~\EX_x[f(W^{n}_k); k < \t_{D}(W^{n})] \\
	&= \sum_{k = 2}^{\8} \sigma_{\jd}^{2}r_{n}^{2} \int f(y) q_{k}^{n,D}(x, y) dy, 
\end{align*}
where $\~\eta_{0}^{n}$ and $\~\eta_{1}^{n}$ are first and second waiting times of $\~W^{n}$. 
Hence
\begin{equation*}
	g^{n,D}(x,y) = \sigma_{\jd}^2r_{n}^{2} \sum_{k = 2}^{\8} q_k^{n,D}(x,y).
\end{equation*}
Note that $g^{n,D}(x,y)$ is symmetric in $x,y$. When $x$ is the origin, we denote the density by $g^{n,D}(y)$. In addition, it enjoys shifting property, i.e., for any $x\in \Rd$ we have $g^{n,D}(x,\cdot) = g^{n,D-x}(\cdot-x)$.
When the underlying scaling $r_{n}=1$ we denote $g^{n,D}$ by $g^{D}$. 
We denote $g^{n,D}$ by $g^{n}$ when  $D = \Rd$.

\begin{remark}\rm
	We separate the first two jumps instead of just one jump because the proof of \Cref{off diagonal weak bound} requires $k\>2$.
\end{remark}
 
 \medskip

\Cref{green est} and \Cref{holder Harmonic functions} are the main ingredients in establishing the H\"older regularity of the semigroup of the process $X^{n,D}$.

\subsection{Green function estimates} \label{section green function}

In this section we prove the following crude Green function bounds. Note that we only need the assumption \ref{jd8} to get the estimates.

\begin{prop} \label{green est}
    For $R>0$ and  $D = B(0,R)$, we have
	\begin{equation*}
		g^{n,D}(z)\<g_{R}(z)
	\end{equation*}
for all $z\in \R^d$ and $n \in \N^*$,
where 

$$ g_{R}(z) := 
\begin{cases}
    C_{1}R\1_{B(0,R)}(z) &\text{if $d = 1$} , \\
    C_{2}((\log(R/|z|))_+ +1)\1_{B(0,R)}(z) &\text{if $d=2$}, \\
    C_d|z|^{2-d}\1_{B(0,R)}(z) &\text{if $d\>3$} .
\end{cases}
$$
\end{prop}

Let $\~N_t^n$ denote the number of jumps made by $\~W^n$ up to time $t$. Then $\~N_t^n$ has a Poisson distribution with mean $t/(\sigma_{\jd}^2r_n^2)$. Let $p_{t}^{n}$ denote the density of the measure $\P_{0}(\~W^n_{t}\in \cdot , \~\eta_{0}^{n} \< t)$.
Recall that for $k\>1$, $q_{k}^{n}$ is the transition density of $W_{k}^{n}$. Then we can see that
$$
	p_{t}^{n} (x)
	= \sum_{k = 1}^{\8} \P(\~N_{t}^{n} = k) q_{k}^{n}(x)
	= \sum_{k = 1}^{\8} e^{- \sigma_{\jd}^{-2} r_{n}^{-2} t } \frac{(\sigma_{\jd}^{-2} r_{n}^{-2} t)^{k}}{k!} q_{k}^{n}(x).
$$
Let $\bar p(z) = \frac{1}{(2\pi \sigma_{\jd}^{2})^{d/2}}\exp\{-|z|^{2}/(2\sigma_{\jd}^{2})\}$ be the normal density with zero mean and covariance matrix $\sigma_{\jd}^{2}I$. Set $\bar p_k(z) = k^{-d/2} \bar p(z / \sqrt{k})$.

\begin{prop} \label{LCLT}
    If $\|\jd\|_{\8}<\8$, then
    \begin{equation*}
        \sup_{z\in\Rd} | q_k(z) - \bar p_k(z) | = o(k^{-d/2}) \quad \hbox{as } k\to \8. 
    \end{equation*}
\end{prop}

\begin{proof}
    This is essentially due to the local central limit theorem. The result in \cite[Theorem 4.3.1]{ibragimov1975independent} can be generalized to $d$ dimensions. By condition \ref{jd8} $\jd$ is bounded and the local central limit theorem applies. We have 
    \begin{equation*}
        \sup_{z\in\Rd} | k^{d/2} q_k(\sqrt{k} z) - \bar p(z) | = o_k(1)  \quad \hbox{as } k\to \8 .
    \end{equation*}
	Plugging in $\bar p(z) = k^{d/2} \bar p_k(\sqrt{k}z)$ we get the desired result.
\end{proof}

\begin{corollary}
	\label{qk bounds}
	There exists $C>0$ such that for $k\> 1$ we have 
	\begin{equation*}
		\|q_{k}\|_{\8} \< C k^{-d/2} , 
	\end{equation*}
	and for any $L>0$,  there exist $c_L > 0$ and $K_{L} > 0$ such that 
	\begin{equation*}
		q_{k}(z) \> c_{L} k^{-d/2}
	\end{equation*}
	for any $k \> K_{L}$ and all $|z| \< L\sqrt{k}$.
\end{corollary}
\begin{proof}
	From \Cref{LCLT} we see that 
	\begin{equation*}
		\|q_{k}\|_{\8} \< \|\bar p_{k}\|_{\8} + o(k^{-d/2}) \< C' k^{-d/2} 
	\end{equation*}
	for all $k\geq K$, for some integer $K\geq1$. Since $\|q_k\|_\8\leq C_{\jd}$ for every $k\geq1$, taking $C=C'\vee(C_{\jd}K^{d/2})$ gives the bound for every $k\geq1$.
	On the other hand, we can find $K_{L}>0$ such that 
	\begin{equation*}
		\inf_{|z| \< L\sqrt{k}} q_{k}(z) \> \inf_{|z| \< L\sqrt{k}} \bar p_{k}(z) - o(k^{-d/2}) \> c_{L} k^{-d/2}
	\end{equation*}
	for any $k \> K_{L}$. We complete the proof.
\end{proof}

\begin{prop}\label{pt bounds}
	There exists $C>0$ such that
	\begin{equation*}
		\| p_t^{n}\|_{\8} \< Ct^{-d/2}
	\end{equation*}
	for all $t>0$, $n\in\N^*$; and for any $L>0$ there exist $c_{L}',K_{L}'>0$ such that
	\begin{equation*}
		p_{t}^{n}(x) \> c_{L}' t^{-d/2}
	\end{equation*}
	for all $r_{n}>0$, $t>0$ and $x\in \Rd$ satisfying $t/(\sigma_{\jd}^{2}r_{n}^{2})> K_{L}'$ and $|x| \< L \sqrt{t/(2\sigma_{\jd}^2)}$.
\end{prop}
\begin{proof}
	We know that $p_{t}^{n}$ is a Poisson mixture of $q_{k}^{n}$, so it is straightforward to derive bounds of $p_{t}^{n}$ from $q_{k}^{n}$. By \Cref{qk bounds} and scaling we have $|q_{k}^{n}|\< C k^{-d/2} r_{n}^{-d}$ for all $k,n\>1$.
    Choosing $\e = 1/2$, standard exponential estimates show that there exist $c_{1}, c_{2} > 0$ such that
    $\P( | N_{s} - s | \> \e s ) \< c_{1} e^{-c_{2}s}$ for any $s>0$ where $N$ is a rate-one Poisson process. In particular we can choose $c_{1}<1$ if we let $s$ be sufficiently large, say $s> C_{0}$. Setting $s=\frac{t}{\sigma_{\jd}^{2} r_{n}^{2}}$, then 
    \begin{align*}
		p_{t}^{n}(x) 
        = \sum_{k = 1}^{\8} \P(\~N_{t}^{n} = k) q_{k}^{n}(x)
        \< c_{1} e^{-c_{2} t / (\sigma_{\jd}^{2} r_{n}^{2})}  C r_{n}^{-d} + \sum \P(\~N_{t}^{n} = k) q_{k}^{n}(x) , 
	\end{align*}
    where here and for the remainder of this proof we write just $\sum$ to denote the sum over all integers $k$ with $|k-\frac{t}{\sigma_{\jd}^{2} r_{n}^{2}}| < \e \frac{t}{\sigma_{\jd}^{2} r_{n}^{2}}$.
    Hence for these $k$ we have $(1-\e)t < \sigma_{\jd}^{2} r_{n}^{2}k < ( 1 + \e ) t $ . On the one hand,
    \begin{align*}
        \sum \P(\~N_{t}^{n} = k) q_{k}^{n}(x)
        &\< \sum \P(\~N_{t}^{n} = k) C r_{n}^{-d} k^{-d/2} 
         \< \sum \P(\~N_{t}^{n} = k) C (t / \sigma_{\jd}^{2})^{-d/2}\\
        &\< C (t / \sigma_{\jd}^{2})^{-d/2}.
    \end{align*}
	On the other hand, 
	\begin{equation*}
		c_{1} e^{-c_{2} t / (\sigma_{\jd}^{2} r_{n}^{2})} C r_{n}^{-d}
		= f\left(\frac{t}{\sigma_{\jd}^{2}r_{n}^{2}}\right) C (t/\sigma_{\jd}^{2})^{-d/2} , 
	\end{equation*}
	where  $ f(z) = c_{1} e^{-c_{2} z } z^{d/2} $ is bounded in $z>0$.
    Hence together we have $p_{t}^{n} \< C (t/\sigma_{\jd}^{2})^{-d/2}$.

\medskip

	For the lower bound, we may choose $c_{1}<1$ when $s=\frac{t}{\sigma_{\jd}^{2} r_{n}^{2}} > C_{0}$. For any $L>0$, let $c_{L}, K_{L}$ be as in \Cref{qk bounds} and let $ \frac{(1-\e)t}{\sigma_{\jd}^{2} r_{n}^{2}} \> K_{L}$. Then $k \> K_{L}$ for $k$ in the summation set and
	\begin{align*}
		p_{t}^{n}(x) 
        &\> \sum \P(\~N_{t}^{n} = k) q_{k}^{n}(x) \\
        &\> \sum \P(\~N_{t}^{n} = k) r_{n}^{-d} c_{L} k^{-d/2}\\
		&\> (1 - c_{1}e^{-c_{2}t/(\sigma_{\jd}^{2} r_{n}^{2})}) c_{L} \left( \frac{(1+\e)t}{\sigma_{\jd}^{2}} \right)^{-d/2} \\
		&\> (1-c_{1})c_{L} (\sigma_{\jd}^{2}/2)^{d/2} t^{-d/2}
	\end{align*}
	for all $|x| \< L \sqrt{\frac{(1-\e)t}{\sigma_{\jd}^{2} r_{n}^{2}}} r_{n} = L ( 2\sigma_{\jd}^{2})^{-1/2}\sqrt{t}$. Hence we can set $c_{L}' = (1-c_{1})c_{L} (\sigma_{\jd}^{2}/2)^{d/2}$. Combining the restriction $\frac{t}{\sigma_{\jd}^{2} r_{n}^{2}} > C_{0}$ and $ \frac{(1-\e)t}{\sigma_{\jd}^{2} r_{n}^{2}} \> K_{L}$ we get $\frac{t}{\sigma_{\jd}^{2} r_{n}^{2}} \> C_{0} \vee (2K_{L})$. Hence we can set $K_{L}' = C_{0} \vee (2K_{L})$.
\end{proof}

\begin{prop} \label{off diagonal weak bound}
    If $\int |z|^{m} \jd(z)dz < \8$ for some $m\>2$, then we have 
    \begin{equation*}
        q_k(x) \< C k^{-d/2} \left(\sqrt{k}/|x|\right)^{m} \quad \hbox{for } k\geq 2. 
    \end{equation*}
  \end{prop}
  
\begin{proof}
	The proof is inspired by \cite[Proposition 2.4.6]{lawler2010random}. Here $W^1$ denotes the discrete time random walk $W^r$ with $r=1$.
    For $k\> 2$, let $\bar k = k / 2$ if $k$ is even and $\bar k = (k+1) / 2$ otherwise. In either case as $k\>2$, we have $\bar k \wedge (k - \bar k)\>1$ and
    \begin{align*}
         q_k(x) 
        &=  q_{k - \bar k} * q_{\bar k} (x)\\
        &= \int_{|z| > |x| / 2}  q_{k - \bar k} (x-z) q_{\bar k} (z) dz + \int_{|z| \< |x| / 2}  q_{k - \bar k} (x-z) q_{\bar k} (z) dz\\
		&\< C k^{-d/2}  \P_{0}(|W_{\bar k}^{1}|\>|x|/2) + C k^{-d/2} \P_{0}(|W_{k - \bar k}^{1}|\>|x|/2).
    \end{align*}
	In the last inequality we use \Cref{qk bounds}.
    We claim $\P_{0}(|W_{\bar k}^{1}|\>|x|/2)\< C(\sqrt{k}/|x|)^{m}$. To see this, we can apply the BDG inequality (see, e.g. \cite[Theorem 23.12]{kallenberg1997foundations}) to get $\EX_{0}[|W_{\bar k}^{1}|^{m}]\< C k^{m/2}$, and then by Chebyshev's inequality 
    \begin{equation*}
        \P_{0}(|W_{\bar k}^{1}|>|x|/2)\< C \frac{ \EX_{0}[|W_{\bar k}^{1}|^{m}] }{ |x|^{m}} \< C (\sqrt{k}/|x|)^{m}.
    \end{equation*}
    We have a similar bound for $\P_{0}(|W_{k - \bar k}^{1}|\>|x|/2)$. Hence we finish the proof.
\end{proof}

\begin{prop} \label{killed weak bound}
    Let $R>0$ and $D\inn B(0,R)$. For $d = 1, 2$, there exists $c_R=cR^{-2}>0$ for some $c>0$ such that for all $k \> 2$, all $r_n>0$ with $r_n/R$ sufficiently small, and any $x\in D/r_{n}$,
    \begin{equation*}
        q^{D/r_n}_{k}(x) \< C k^{-d/2} e^{- c_R r_{n}^{2}k} .
    \end{equation*}
\end{prop}
\begin{proof}
    Let $\bar k = k / 2$ if $k$ is even and $\bar k = (k+1) / 2$ otherwise. 
    Then by the Chapman-Kolmogorov equation and \Cref{qk bounds}
    \begin{equation*}
        q^{D/r_n}_{k}(x) 
		\< C k^{-d/2}\int q_{\bar k}^{D/r_n} (y)dy.
    \end{equation*}
	On the other hand, 
    \begin{equation*}
        \int q_{\bar k}^{D/r_n} (y)dy 
        \< \P_{0}\left(\max_{i\<\bar k}|W_{i}^{1}|\< R/r_n\right) , 
    \end{equation*} 
	where $W^{1}$ denotes $W^{r}$ for $r=1$ (i.e. the discrete time random walk without scaling).
    Let $j_n$ be the largest nonnegative integer such that $j_n \< \bar k / \lfloor {R^{2}}/r_n^{2}\rfloor$. 
    Then by the strong Markov property, 
    \begin{align*}
        \P_{0}(\max_{i\<\bar k}|W_{i}^{1}| < R/r_n) 
		&\< \sup_{|x|\< R/r_{n}}\P_{x}\left(\max_{i\<\lfloor R^{2}/r_{n}^{2} \rfloor}|W_{i}^{1}| < 2R/r_n\right)^{j_n} \\
		&\< \P_{0}\left(\max_{i\< \lfloor R^{2}/r_{n}^{2} \rfloor}|\frac{r_{n}}{R} W_{i}^{1}| < 3 \right)^{j_n} \\
		&\< C e^{-{c_R r_{n}^{2}k}}
    \end{align*}
	where $c_{R}=cR^{-2}$ for some $c>0$.
	The last inequality holds when $r_n/R$ is sufficiently small. Indeed, by Donsker's theorem,
$\sup_{i\< \lfloor s^{-2} \rfloor}|s W_{i}^{1}|$ converges in distribution, as $s\downarrow 0$, to
$\sup_{0\<t\<1}|\sigma_\jd W_{t}|$, where $W$ is the standard Brownian motion.
\end{proof}

Recall that $q_k^{n,D}$ is the transition density of $W^{n}_k$ killed upon leaving domain $D\inn \Rd$ and $g^{n,D}(y) = \sigma_{\jd}^{2}r_{n}^{2} \sum_{k = 2}^{\8} q_k^{n,D}(y)$.
When the underlying scaling $r_{n}=1$, we omit the first superscript. When $D = \Rd$ we omit the second superscript.

\begin{proof}
	[Proof of \Cref{green est}]
    By scaling we see that 
   $        q^{n,D}_{k}(z) = r_{n}^{-d} q^{D/r_n}_{k}(z/r_n) , $
     which implies
     \begin{equation}\label{green function scaling}
         g^{n,D}(z) = \sigma_{\jd}^{2} \sum_{k=2}^{\8} r_{n}^{2-d} q^{D/r_n}_{k}(z/r_n).
    \end{equation}
    First we show the case when $d \> 3$.  
        Using \Cref{green function scaling} we write 
    \begin{equation*}
         g^{n,D}(z) = \sigma_{\jd}^{2} |z|^{2-d} \sum_{k=2}^{\8} (|z|/r_{n})^{d-2} q^{D/r_n}_{k}(z/r_n).
    \end{equation*}
    As $q^{D/r_n}_{k} \< q_{k}$, we only need to show that
    \begin{equation} \label{eq:green sum}
        \sup_{y\in\Rd}\sum_{k=2}^{\8} |y|^{d-2} q_{k}(y) < \8.
    \end{equation}
    
    By Assumption \ref{jd8}, we can choose $m\geq2$ with $m>d-2$ such that $\int_{\Rd}|z|^m\jd(z)\,dz<\infty$.   
    For $k \< |y|^{2}$ we apply \Cref{off diagonal weak bound} to get   
    \begin{eqnarray*}
        \sum_{2\leq k\< |y|^{2}} |y|^{d-2}  q_{k}(y)
      &  \< &  |y|^{d-2} \sum_{2\leq k\< |y|^{2}} C k^{-d/2} (\sqrt{k}/|y|)^{m}
        = C|y|^{d-2 -m} \sum_{2\leq k\< |y|^{2}} k^{(m-d)/2} \\
        &\leq & C .
    \end{eqnarray*} 
     For $k \> |y|^{2}$ we apply \Cref{qk bounds} to see that 
    \begin{equation*}
        \sum_{k\>2\vee |y|^{2}}   |y|^{d-2} q_{k}(y) \< C   |y|^{d-2} \sum_{k\>2\vee |y|^{2}} k^{-d/2} \leq C.
    \end{equation*}
Hence \eqref{eq:green sum} is proved.

    For $d = 2$, by \Cref{off diagonal weak bound} 
    \begin{equation*}
        \sum_{2\leq k\< |z / r_{n}|^{2}}  q_{k}^{D/r_{n}}(z / r_{n})
        \<\sum_{2\leq k\< |z / r_{n}|^{2}} C k^{-1} (\sqrt{k}/|z / r_{n}|)^{2}\\
      \leq  C.
    \end{equation*}
    On the other hand, by \Cref{killed weak bound}
    \begin{align*}
        \sum_{k\>2\vee |z/r_{n}|^{2}} q_{k}^{D/r_{n}}(z / r_{n})
        &\< C\sum_{k\>2\vee |z/r_{n}|^{2}} \frac{1}{kr_{n}^{2}} e^{-c_R r_{n}^{2}k} r_{n}^{2} \\
		&= C\sum_{k\>2\vee |z/r_{n}|^{2}} \frac{1}{k} e^{-c (r_{n}/R)^{2}k}\\
		&\< C\sum_{2\vee |z/r_{n}|^{2}\< k \< (R/r_{n})^{2}} \frac{1}{k} + C\sum_{k\>2\vee (R/r_{n})^{2}} \frac{1}{k} e^{-c (r_{n}/R)^{2}k}\\
		&\< C\left(\log\frac{R^2/r_n^2}{|z|^2/r_n^2}+1\right) + C \\
        &\< C_2 ((\log(\frac{R}{|z|}))_+ +1).
    \end{align*}
	
    For $d = 1$, by \Cref{killed weak bound}
    \begin{align*}
        \sum_{k=2}^{\8} r_{n} q^{D/r_n}_{k}(z/r_n)
        &\< \sum_{k=2}^{\8} r_{n} C k^{-1/2} e^{-c_R r_{n}^{2}k} \\
        &= \sum_{k=2}^{\8}  C (r_{n}^{2}k)^{-1/2} e^{-c_R r_{n}^{2}k} r_{n}^{2}\\
        &\< C\int_{0}^{\8} y^{-1/2}e^{-c_R y}dy \\
        &= CR\int_{0}^{\8} (y/R^{2})^{-1/2}e^{-cR^{-2}y}R^{-2}dy
		= C_1 R.
    \end{align*}
    For $d=1,2$, the preceding argument applies when $r_n/R$ is sufficiently small. For the remaining scales, put $a=R/r_n\leq a_0$, where $a_0$ is fixed, and write $B_a=B(0,a)$ and $T_a=\tau_{B_a}(W^1)$. Since
    \[
    q_2^{B_a}(w,z)=\int_{B_a}\jd(y-w)\jd(z-y)\,dy\leq C_{\jd}^2|B_a|,
    \]
    applying the Markov property at step $k-2$, for $k\geq2$ and $z\in B_a$, gives
    \[
    q_k^{B_a}(z)
    =\EX_0[q_2^{B_a}(W^1_{k-2},z);\,T_a>k-2]
    \leq C_{\jd}^2|B_a|\P_0(T_a>k-2).
    \]
    As in the proof of \Cref{uint}, $\EX_0T_{a_0}<\infty$. Since $T_a\leq T_{a_0}$ for $a\leq a_0$, summing the preceding bound yields
    \[
    \sum_{k=2}^\infty q_k^{B_a}(z)\leq C_{\jd}^2|B_a|\EX_0T_a\leq Ca^d.
    \]
    By \eqref{green function scaling}, for $z\in D$,
    \[
    g^{n,D}(z)\leq C r_n^{2-d}a^d
    =\begin{cases}CR,&d=1,\\ Ca^2\leq Ca_0^2,&d=2.\end{cases}
    \]
    Replacing $C_1$ by $C_1\vee C$ and $C_2$ by $C_2\vee(Ca_0^2)$ gives the stated bounds, since $1+(\log(R/|z|))_+\geq1$ and $g^{n,D}$ vanishes outside $D$.
    This completes the proof of the proposition.
\end{proof}

\subsection{H\"older regularity of harmonic functions} \label{section harmonic function holder}

We say a function $h$ is harmonic in a bounded open set $D\inn \Rd$ with respect to a process $Z$ if for any relatively compact open subset $D_{1}$ of $D$, $\EX_x[|h(Z_{\tau_{D_{1}}})|; \tau_{D_{1}}<\8]<\8$ and $h(x) = \EX_x[h(Z_{\tau_{D_{1}}}); \tau_{D_{1}}<\8]$ for all $x\in D_{1}$.

We prove the following H\"older regularity of harmonic functions. 

\begin{prop}\label{holder Harmonic functions}
	There are $\r'>0, c_{0}>0$ and $C>0$ such that for any $n\>1$, if $h$ is bounded and harmonic with respect to $X^{n}$ in a ball $B(x_0,2r)$, then 
	\begin{equation*}
		|h(x)-h(y)|\< C \left(\frac{|x-y|{\vee c_{0}r_{n}}}{r}\right)^{\r'} \|h\|_\8 \qquad \text{for }x,y\in B(x_0,r).
	\end{equation*}
\end{prop}

We use the standard approach to prove the H\"older regularity of harmonic functions. But since $X^{n}$ has positive probability to stay at the starting point and the transition semigroup is not absolutely continuous with respect to the Lebesgue measure, we need to overcome this difficulty. 

Recall that $\~W^{n}$ is the continuous time random walk (without time change). Note that since $X^{n}$ is a continuous time change of $\~W^{n}$, 
the harmonicity with respect to $X^{n}$ is  the same as the harmonicity with respect to $\~W^{n}$.

\begin{lemma}
	\label{exittime}
	For any $\e>0$, there exists $\d=\d(\e)>0$ such that
	$$\P_{ x_{0}}(\tau_{B({ x_{0}},R)}(\~W^n)\<\d R^2) \< \e$$
	for all $x_{0}\in \Rd$, $n\in\N^*$ and $R>0$.
\end{lemma}
\begin{proof}
	Without loss of generality, we can assume $x_{0}$ is the origin.
	Let $\varphi$ be a rotation invariant $C^2(\R^d)$ function such that $0\<\varphi\<1$, $\varphi(0) = 0$ and $\varphi(x) = 1$ for $|x|\>1$. Define $\varphi_R(x) = \varphi(x/R)$. Then $\varphi_R(\~W^n_t)-\int_{0}^{t} \mathcal L^n \varphi_R(\~W^n_s)ds$ is a martingale, where $\mathcal{L}^n = (\sigma_{\jd}r_{n})^{-2}(Q_n - I)$ is the infinitesimal generator of $\~W^n$ ($I$ is the identity map). By Taylor approximation error bound we have for some $L > 0$ (depending only on the choice of $\varphi$)
	\begin{equation*}
		|\varphi(y+x) - \varphi(x) - \nabla\varphi(x)\. y| \< \frac{L}{2}|y|^2.
	\end{equation*}
	Using the fact that $\jd$ has zero mean, we have
	\begin{align*}
		|\mathcal{L}^n \varphi_R(x)|
		&= \left|\frac{1}{\sigma_{\jd}^{2} r_n^2}\int (\varphi(\frac{r_ny+x}{R})-\varphi(\frac{x}{R}))\jd(y)dy \right|\\
		&= \left|\frac{1}{\sigma_{\jd}^{2} r_n^2}\int (\varphi(\frac{r_ny+x}{R})-\varphi(\frac{x}{R}) - \nabla\varphi(\frac{x}{R})\. \frac{r_n y}{R})\jd(y)dy \right|\\
		&\< \frac{1}{\sigma_{\jd}^{2} r_n^2}\int \frac{Lr_n^2 |y|^2}{2R^2}\jd(y)dy 
		 = \frac{L d}{2R^2}
	\end{align*}
	for any $x\in \Rd$.
	Hence $\varphi_R(\~W^n_t)- \frac{L d}{2R^2}t$ is a supermartingale. Then 
	\begin{equation*}
		\P_0(\tau_{B(0,R)}(\~W^n)\<\d R^2) \< \EX_0[\varphi_R(\~W^n_{\tau_{B(0,R)}\wedge \d R^2})] \< \frac{L
		d}{2R^2} \EX_0[\tau_{B(0,R)}(\~W^n) \wedge \d R^2] \< \frac{Ld \d}{2}.
	\end{equation*}
	Choosing $\d = 2\e / (L d)$ finishes the proof.
\end{proof}

\begin{lemma} \label{exittimeMT}
	For any $c_{1}>0$, there exists $C_{1} > 0$ such that for any $x_{0},x\in\Rd$, $n\in \N^*$, and $r > c_{1}r_{n}$, we have $\EX_x[\tau_{B(x_{0},r)}(\~W^n)] \< C_{1}r^2$.
\end{lemma}
\begin{proof}
	By \Cref{pt bounds}, taking $t=c_{2}r^2$ for $c_{2}$ to be determined later, we have
	\begin{align*}
		\P_x(\tau_{B(x_{0},r)}(\~W^n) > t ) 
		&\< \P_x(\~W^n_t \in B(x_{0},r), \~\eta_{0}^{n} \< t) + \P_x(\~\eta_{0}^{n} > t)\\
		&\< Ct^{-d/2}|B(x,r)| + e^{ -\frac{t/2}{\sigma_{\jd}^{2} r_{n}^{2}}}\\
		&\< C c_{2}^{-d/2} + e^{ -c c_{2}c_{1}^{2} / \sigma_{\jd}^{2}}.
	\end{align*}
	By choosing $c_{2}$ large enough, we can make the above less than $1/2$.
	By the Markov property, for $m\in \N$,
\begin{align*}
\mathbb{P}_x\left(\tau_{B(x_{0},r)}(\~W^n)>(m+1) t\right) & \leq \mathbb{E}_x\left[\mathbb{P}_{\~W^n_{m t}}\left(\tau_{B(x_{0},r)}(\~W^n)>t\right) ; \tau_{B(x_{0},r)}(\~W^n)>m t\right] \\
& \leq \frac{1}{2} \mathbb{P}_x\left(\tau_{B(x_{0},r)}(\~W^n)>m t\right) .
\end{align*}
By induction, we have
\begin{equation*}
\mathbb{P}_x\left(\tau_{B(x_{0},r)}(\~W^n)>m t\right) \leq 2^{-m} \text{ for all } m \in \mathbb{N}.
\end{equation*}
Summing over $m\in\N$ proves the lemma.
\end{proof}

The next proposition establishes near-diagonal bounds for the transition probability of the killed process.

\begin{prop}
	\label{killedlowerbound}
	There exist constants $c_{0}, c_{1}, c_{2}>0$ and $\Theta > c_{1}$ such that for any $x_0\in \R^d$, $n\in \N^*$ and $t>c_{0} \sigma_{\jd}^{2} r_{n}^{2}$, $D=B(x_0,r)$ with $r\>\Theta t^{1/2}$, $x\in B(x_0,c_{1} t^{1/2})\inn D$, and any nonnegative integrable function $f\in L^{1}(\Rd)$ with support in $B(x_0,c_{1}t^{1/2})$, we have
	\begin{equation*}
		\EX_{x}[f(\~W_{t}^{n,D})]\> c_{2} t^{-d/2} \|f\|_{1}.
	\end{equation*}
	
\end{prop}
\begin{proof}
	For simplicity of notation we set $Y=\~W^n$ and $T = \tau_D(\~W^n)$. By \Cref{pt bounds} there exists $c, c_{0}, c_{1}>0$ such that $p_{t}^{n}(z) \> c t^{-d/2}$ for any $n\in \N^*$, $t>c_{0}\sigma_{\jd}^{2} r_{n}^{2}$ and $|z| \< 2 c_{1} t^{1/2}$. Let $x\in B(x_0,c_{1} t^{1/2})$. For a nonnegative function $f\in L^1(\R^d)$ with support in $B(x_0,c_{1} t^{1/2})$, we have
	\begin{equation} \label{*1}
		\EX_x[f(\~W_t^{n,D})]
		=\EX_x[f(Y_t)]-\EX_x[f(Y_t);T\<t].
	\end{equation}
	For $\EX_x[f(Y_t)]$ we have
	\begin{equation} \label{*2}
		\EX_x[f(Y_t)] 
		\> \EX_x[ f(Y_{t}); \~\eta_{0}^{n}\< t] 
		= \int f(y) p_{t}^{n}(y-x) dy 
		\> c t^{-d/2}\|f\|_{1} .
	\end{equation}
	For $\EX_x[f(Y_t);T\<t]$ we first estimate the exit probability $\P_{x}( T \< t / 2)$. For any $\e \in (0,1)$, by \Cref{exittime} there exists $\d = \d(\e)$ such that $\P_{x}( \t_{B(x,s)}(\~W^{n}) \< s^{2}\d)\<\e$ for any $s>0$.
	If $t/2 \< (r - |x-x_{0}|)^{2}\d$, then we have
	\begin{align*}
			\P_{x}( T \< t / 2) 
			&\< \P_{x} \left( \t_{B(x,r-|x-x_{0}|)}(\~W^{n}) \< t / 2 \right) \\
			&\< \P_{x} \left( \t_{B(x,r-|x-x_{0}|)}(\~W^{n}) \< (r - |x-x_{0}|)^{2}\d \right)\\
			&\< \e.
	\end{align*}
	The above holds when $r \> |x-x_{0}| + (2\d)^{-1/2}t^{1/2}$. So we set $\Theta  = c_{1} + (2\d)^{-1/2}$. Returning to $\EX_x[f(Y_t);T\<t]$ we have 
	\begin{align}\label{*3}
		\begin{split}
		\EX_x[f(Y_t);T\<t] 
		&= \EX_x[f(Y_t);T\<t/2] + \EX_x[f(Y_t);t/2<T\<t]\\
		&=:S_1(x) + S_2(x)
		\end{split}
	\end{align}
	For $S_1(x)$, using the strong Markov property and \Cref{pt bounds} we have
	\begin{align}\label{*4}
		\begin{split}
		S_1(x) 
		&= \EX_x[f(Y_t);T\<t/2] \\
		&= \EX_x[\EX_{Y_T}f(Y_{t-T});T\<t/2]\\
		&\<\P_x(T\<t/2) \sup_{z\in D^c, s\in[t/2,t]} \EX_z f(Y_s)\\
		&\< \e \sup_{z\in D^c, s\in[t/2,t]} \left( \EX_z [f(Y_s); \~\eta_{0}^{n} \< s] +  \EX_z [f(Y_s); \~\eta_{0}^{n} > s] \right)\\
		&\< \e Ct^{-d/2}\|f\|_1.
		\end{split}
	\end{align}
	Note that the second term in the second-to-last line is 0 because $Y_{s}=z$ when $\~\eta_{0}^{n}>s$ and $z$ is not in the support of $f$.

	For $S_2(x)$, let $T_{t}'=\sup\{l\<t:\~W^n_l\in D^{c}\}$ be the last hitting time of $D^{c}$ before time $t$. 
	Then we have $T_{t}'\>T$ and 
	\begin{equation*}
		\EX_{y}[f(Y_{t}); t/2 < T \< t] \< \EX_{y}[f(Y_t);t/2\<T_{t}'\<t]
	\end{equation*}
	for all $y\in\Rd$.
	Let $g\in L^1(\R^d)$ with $g\>0$ and support in $B(x_{0},c_{1}t^{1/2})$ and $t/2=t_0<t_1<...<t_k=t$ with each $t_i = t/2 + (i/k)(t-t/2)$. Then
	\begin{align*}
		\int g(y)\EX_y[f(Y_t);t/2\<T_{t}'\<t]dy\< \liminf_{k\to\8} \sum_{j=0}^{k-1} \EX_{m_0}[g(Y_0)\1_{D^{c}}(Y_{t_{j}})\prod_{i=j+1}^k\1_{D}(Y_{t_i})f(Y_{t})]
	\end{align*}
	where $m_0$ is the Lebesgue measure.  
	By time reversal (see e.g. \cite[Lemma 4.1.2]{FOT}) we have 
	\begin{equation*}
		\EX_{m_0}[g(Y_0)\1_{D^{c}}(Y_{t_{j}})\prod_{i=j+1}^k\1_{D}(Y_{t_i})f(Y_{t})] = \EX_{m_0}[g(Y_t)\1_{D^{c}}(Y_{t-t_{j}})\prod_{i=j+1}^k\1_{D}(Y_{t-t_i})f(Y_{0})].
	\end{equation*}
	Summing over $j = 0,1,\ldots, k-1$ and letting $k\to \8$, the right hand side above tends to $\EX_{m_0}[g(Y_t)f(Y_0);T\<t - t/2]$. Hence 
	\begin{align*}
		\int \EX_{y}[f(Y_{t}); t/2 < T \< t] g(y) dy 
		&\< \EX_{m_0}[g(Y_t)f(Y_0);T\<t - t/2] \\
		&= \EX_{m_0}[g(Y_t)f(Y_0);T\<t/2, \~\eta_{0}^{n}<t].
	\end{align*}
	The last line above is because $f$ is supported in $B(x_{0},c_{1}t^{1/2})$, and since $g$ is also supported in $B(x_{0},c_{1}t^{1/2})$, we can apply the same estimate for $S_{1}$ to get
	\begin{equation*}
		\EX_{m_0}[g(Y_t)f(Y_0);T\<t/2, \~\eta_{0}^{n}<t] \< \e C t^{-d/2} \|g\|_{1} \|f\|_{1}
	\end{equation*} 
	and hence $\int S_{2}(y)g(y)dy\<\e C t^{-d/2} \|g\|_{1} \|f\|_{1}$ for any $f, g\in L^{1}(\Rd)$ with support in $U:=B(x_{0},c_{1}t^{1/2})$. Therefore, for Lebesgue-a.e. $y\in U$,
	\begin{equation}\label{eq:S2}
		S_{2}(y) \< \e C t^{-d/2} \|f\|_{1}
	\end{equation}
	We claim $S_{2}^{f}: y\mapsto S_{2}(y)$ is continuous and so \eqref{eq:S2} holds for all $y\in B(x_{0},c_{1}t^{1/2})$. To see this, notice that for fixed $t>0$ and $y\in B(x_{0},c_{1}t^{1/2})$, under $\P_0$, almost surely none of the finitely many positions visited by $y+Y$ up to time $t$ lies on $\partial D$, and $Y$ does not jump at $t/2$ or $t$.
	 For any sequence $\{y_i;i\geq1\}\subset B(x_0,c_1t^{1/2})$ converging to $y$ as $i\to\infty$, if $f$ is continuous, then by the dominated convergence theorem
	$\lim_{i\to\infty}S_{2}^{f}(y_i)=\lim_{i\to\infty}\EX_{0}[f(y_i+Y_t);t/2<\t_{D}(y_i+Y)\<t]=S_{2}^{f}(y)$.
	For $f\in L^{1}(U)$, choose a sequence $\{f_k;k\geq1\}\subset C_c(U)$ so that $\lim_{k\to\8}\|f_k-f\|_{1}=0$. Then uniformly in $x\in U$,
	\begin{align*}
		|S_2^{f_k}(x)-S_2^f(x)|
		&\< \EX_x[|f_k-f|(Y_t);t/2<T\le t] \\
		&\< \EX_x[|f_k-f|(Y_t);\~\eta_0^n\le t] \\
		&= \int |f_k(y)-f(y)|p_t^n(y-x)dy \\
		&\< \|p_t^n\|_\infty \|f_k-f\|_1  
		 \< C t^{-d/2}\|f_k-f\|_1.
	\end{align*}
	Thus $S_2^{f_k}$ converges to $S_2^f$ uniformly on $U$ as $k\to\infty$. Since each $S_2^{f_k}$ is continuous, $S_2^f$ is continuous.
	Finally we have 
	\begin{equation*}
		\EX_x[f(\~W_t^{n,D})]
		=\EX_x[f(Y_t)]- S_{1}(x) - S_{2}(x)
		\> (c-\e C)t^{-d/2}\|f\|_1.
	\end{equation*}
	Taking $\e$ small enough so that $\e C <c$, the desired result follows from \eqref{*1} to \eqref{eq:S2}.
\end{proof}

\begin{corollary}\label{hitting prob estimate}
	There exist $c_{0}', c>0$ and $\Xi_{0} \in (0,1)$ such that for any $\kappa \in (0,1]$ and $\Xi\in(0,\Xi_{0}]$, if $r\>c_{0}' \Xi^{-1} r_n$, $x\in \R^d$, and $A\inn B(x,\Xi r)$ with $|A|/|B(x,\Xi r)| \> \kappa$ and $y \in B(x,\Xi r)$, then
	\begin{equation*}
		\P_y(\sigma_A(\~W^n)< \tau_{B(x,r)}(\~W^n)) \> c\kappa.
	\end{equation*}
where $\sigma_A(\~W^n):=\inf\{t\geq0:\~W_t^n\in A\}$.
\end{corollary}
\begin{proof}
	Let $c_{0}, c_{1}, c_{2}, \Theta$ be the constants in \Cref{killedlowerbound}. Set $\Xi_{0} = c_{1}/\Theta\in(0,1)$. Fix $\kappa\in(0,1]$ and $\Xi\in(0,\Xi_{0}]$.
	Applying \Cref{killedlowerbound} with $t = (\Xi r / c_{1})^{2}$ and $f = \1_{A}$, we get for $y\in B(x,\Xi r) = B(x, c_{1}t^{1/2})$ and $A \inn B(x,\Xi r)$,
	\begin{align*}
		\P_y(\sigma_A(\~W^n)< \tau_{B(x,r)}(\~W^n)) 
		&\> \P_y \big(\~W^n_t\in A; t < \tau_{B(x,r)}(\~W^n) \big)  \\
		&\> c_{2} t^{-d/2} |A|  
		 = c_{2} (\Xi r/c_{1})^{-d}|A| , 
	\end{align*}
	provided $t \> c_{0}\sigma_{\jd}^{2}r_{n}^{2}$ and $r\>\Theta t^{1/2}$. The two conditions are equivalent to 
	$r\>c_{1}c_{0}^{1/2}\sigma_{\jd}\Xi^{-1}r_{n}$ and $\Xi\< \Xi_{0}$.
	We can then choose $c_{0}'=c_{1}c_{0}^{1/2}\sigma_{\jd}$. We then apply $|A|/|B(x,\Xi r)| \> \kappa$ to get 
	\begin{equation*}
		c_{2} (\Xi r/c_{1})^{-d}|A| \> c_{2} c_{1}^{d} |B(0,1)|\kappa.
	\end{equation*}
	Setting $c =c_{2} c_{1}^{d} |B(0,1)|$ finishes the proof.
\end{proof}

\begin{proof}
	[Proof of \Cref{holder Harmonic functions}] 
	By \Cref{hitting prob estimate}, there exist $c_{0}', c_1 > 0 $ and $\Xi_{0}\in(0,1)$ such that for any $\rho\in(0,\Xi_{0}], x\in \Rd, r \> c_{0}'\rho^{-1} r_{n}$, $A \inn B(x, \rho r)$ with $|A|>|B(x,\rho r)|/3$ and $y\in B(x,\rho r)$ we have
	\begin{equation*}
		\P_y(\sigma_A(\~W^{n}) < \tau_{B(x,r)}(\~W^{n})) \> c_1.
	\end{equation*} 
	By the L\'evy system formula (see, e.g. \cite[(A.3.33)]{CF12}), for any bounded function $f$ on $\Rd\times\Rd$ vanishing on the diagonal and any region $D\inn \Rd$, we have
	\begin{equation*}
		\EX_x[\sum_{0<s\<\t_{D}}f(\~W^n_{s-}, \~W^n_{s})] = \EX_x[\int_{0}^{\t_{D}} \int f(\~W^n_{s},z) Q_n(\~W^n_s,dz)\frac{\sigma_{\jd}^{-2}}{r_{n}^{2}}ds].
	\end{equation*}
	For $z_0\in B(x,r)$, taking $D=B(x,r)$ and $f(y,z) = \1_{B(x,r)}(y)\1_{B(x,r')^c}(z)$ for $r' \> 2r$ and noticing that $\~W^n\in D$ before time $\t_D$ we get
	\begin{align*}
		\mathbb{P}_{z_0}\left(\~W^n_{\tau_{B(x,r)}} \notin B(x, r') \right)
		&= \EX_{z_0}[\int_{0}^{\t_{B(x,r)}} \int_{B(x,r')^c} \frac{1}{ \sigma_{\jd}^2r_{n}^{2}}  \jd_n(z-\~W^n_s) dz ds] \\
		&\< \EX_{z_0}[\t_{B(x,r)}(\~W^n)] \int_{B(0,r'-r)^c}	\frac{1}{ \sigma_{\jd}^2r_{n}^{2}}  \jd_n(z)dz \\
		&\< \EX_{z_0}[\t_{B(x,r)}(\~W^n)]	\frac{1}{(r'-r)^{2}}	\frac{(r'-r)^{2}}{ \sigma_{\jd}^2r_{n}^{2}}  \int_{B(0,\frac{r'-r}{r_n})^c} \jd(z)dz.
	\end{align*}
	By \Cref{exittimeMT} there exists $c_{0}''>0$ such that $\EX_{z_0}[\t_{B(x,r)}(\~W^n)]\< c r^2$ for any $r \> c_{0}'' r_{n}$.
	Because $\jd$ has finite second moment, we have 
	\begin{equation*}
		\frac{(r'-r)^{2}}{ \sigma_{\jd}^2r_{n}^{2}}  \int_{B(0,\frac{r'-r}{r_n})^c} \jd(z)dz
		\< \frac{1}{\sigma_{\jd}^2} \int_{B(0,\frac{r'-r}{r_n})^c} |z|^2 \jd(z)dz \< d.
	\end{equation*}
	Note also that $\frac{1}{(r'-r)^{2}}\<\frac{4}{(r')^2}$ as $r' \> 2r$. Hence there exists $c_2>0$ such that, for every $z_0\in B(x,r)$,
	\begin{equation*}
		\mathbb{P}_{z_0}\left(\~W^n_{\tau_{B(x,r)}} \notin B(x, r') \right) \< c_2 (r/r')^{2}.
	\end{equation*}
	Let $\zeta = (1 - \frac{c_1}{4})^{1/2}$, $\rho = \Xi_{0} \wedge (\frac{c_1 \zeta}{8 c_2})^{1/2} \wedge \frac{1}{2}$ and $c_{0} = c_{0}'\rho^{-1} \vee c_{0}''$.
	Let $h$ be $\mathcal E^{n}$-harmonic in $B(x_0,2r)$ for $x_0\in\Rd$. By scaling, without loss of generality, assume $0\<h\<1$. For any $x\in B(x_0,r)$, $k\in\mathbb Z$, set $B_k = B(x,\rho^k r)$, $a_k = \inf_{B_k} h$, $b_k = \sup_{B_k} h$ and $\t_k = \tau_{B_k}(\~W^n)$. We use induction to show $b_k - a_k \< \zeta^{k-1} $ for $k\> 0$ and $\rho^{k}r \> c_{0}r_{n}$. The bound holds for $i=0,1$, and $b_i-a_i\leq1\leq\zeta^{i-1}$ for $i\leq0$. Now let $k\geq1$ and suppose $b_{i} - a_{i} \< \zeta^{i-1}$ for all $i\< k$ with $\rho^i r\geq c_0r_n$. If $\rho^{k+1}r\geq c_0r_n$, we are going to show that $b_{k+1} - a_{k+1} \< \zeta^{k}$.
	Set $A' = \{ z \in B_{k+1}: h(z) \< (a_k+b_k) / 2\}$. We may assume $|A'| \> |B_{k+1}|/2$ otherwise we can replace $h$ by $1-h$. Choose compact set $A\inn A'$ such that $|A|/|B_{k+1}| > 1/3$. For $\e>0$ choose $z_1$, $z_2\in B_{k+1}$ such that $h(z_1)\> b_{k+1} - \e$ and  $h(z_2)\< a_{k+1} + \e$. Then
	\begin{align*}
		b_{k+1}-a_{k+1}-2 \e 
		&\leq h\left(z_1\right)-h\left(z_2\right) \\
		&= \mathbb{E}_{z_1}\left[h\left(\~W^n_{\sigma_A \wedge \tau_{k}}\right)-h\left(z_2\right)\right] \\
		&= \mathbb{E}_{z_1}\left[h\left(\~W^n_{\sigma_A}\right)-h\left(z_2\right) ; \sigma_A<\tau_{k}\right] \\
		&\quad + \mathbb{E}_{z_1}\left[h\left(\~W^n_{\tau_{k}}\right)-h\left(z_2\right) ; \sigma_A>\tau_{k},\right. \left.\~W^n_{\tau_{k}} \in B_{k-1}\right] \\& \quad + \sum_{i=1}^{\infty} \mathbb{E}_{z_1}\left[h\left(\~W^n_{\tau_{k}} \right)-h\left(z_2\right) ; \sigma_A>\tau_{k},\right. \left.\~W^n_{\tau_{k}} \in B_{k-1-i} \setminus B_{k-i}\right] \\
		&\leq \left(\frac{a_k+b_k}{2}-a_k\right) \mathbb{P}_{z_1}\left(\sigma_A<\tau_{k}\right) \quad+\left(b_{k-1}-a_{k-1}\right) \mathbb{P}_{z_1}\left(\sigma_A>\tau_{k}\right) \\
		& \qquad + \sum_{i=1}^{\infty}\left(b_{k-1-i}-a_{k-1-i}\right) \mathbb{P}_{z_1}\left(\~W^n_{\tau_{k}} \notin B_{k-i}\right)\\
		&\leq (b_{k-1} - a_{k-1}) \left(1-\frac{\mathbb{P}_{z_1}\left(\sigma_A<\tau_{k}\right)}{2}\right)+\sum_{i=1}^{\infty} \zeta^{k-2-i}c_{2}\rho^{2i} \\
		&\leq \zeta^{k-2} \left(1-\frac{c_1}{2}\right) + 2 c_2 \zeta^{k-3} \rho^2 \\
		&\leq \zeta^{k-2} \left(1-\frac{c_1}{2}\right) +\frac{c_1}{4} \zeta^{k-2} \\
		&=  \zeta^{k} .
	\end{align*}
	Since $\e>0$ is arbitrary, we have $b_k - a_k \< \zeta^{k-1} $ for all $k\>0$ with $\rho^k r\geq c_0r_n$.

	Now for $x,y\in B(x_0,r)$ with $0<|x-y|\vee c_0r_n<r$, choose $k\geq0$ to be the largest integer such that $|x-y| \vee c_{0} r_{n} < \rho^{k}r$. Then $\log( \frac{|x-y| \vee c_{0} r_{n}}{r})\>(k+1)\log \rho$ and
	\begin{equation*}
		|h(x) - h(y)| \< \zeta^{k-1} \< \zeta^{\log( \frac{|x-y| \vee c_{0} r_{n}}{r})/\log \rho - 2} \< c_3 \left(\frac{|x-y|\vee c_{0} r_{n}}{r} \right)^{\log \zeta / \log \rho}.
	\end{equation*}
	Here $c_3=\zeta^{-2}\geq1$. When $|x-y|\vee c_0r_n\geq r$, the same inequality follows from $|h(x)-h(y)|\leq1$. The proof is now completed.
\end{proof}

\subsection{H\"older regularity of time-changed semigroups}

In this section,  
we assume the measure $\mu$ has full quasi-support on $\Rd$ and satisfies the condition \ref{condA}. Namely, for any $R > 0$, there exists 
$ \3_{\mu} = \3_{\mu}(R) > d - 2$  such that for any $x\in B(0,R)$ and $r\in(0,1)$ we have $\mu(B(x,r))\<C_R r^{\3_{\mu}}$.
This condition ensures that the measure $\mu$ is a smooth measure (\Cref{smooth measure}).  In addition,
\begin{equation*} 
	\lim_{r\downarrow0}r^{\3_{\mu}-d+2}\log\left(\int_{B(0,2R)}\frac{dx}{\int \phi((y-x)/r)\mu(dy)}\right)=0.
\end{equation*}
This condition is used in the proof of \Cref{Holder}.
We also assume $(\jd,\phi)$ satisfies assumptions \ref{jd0} and \ref{jd8}.

We first prepare several lemmas.
The main results of this subsection are \Cref{Holder} and \Cref{killed Holder}.

\begin{lemma} \label{greenfunction}
Let $g_{r}$ be the function given in \Cref{green est}. For $R>2$, there exist 
$\beta:= \3(R)>0$ and $C_R>0$ such that for any $x\in B(0,R)$ and $r\in(0,2)$
$$(g_{r}*\mu)(x)\equiv\int_{B(x,r)}g_{r}(x-y)\mu(dy)\<C_Rr^\3.$$
\end{lemma}
\begin{proof}
	For $d = 1$, \Cref{green est} gives
	\begin{equation*}
		\int_{B(x,r)}g_r(x-y)\mu(dy)\< C r\mu(B(x,r)) \< C_R r.
	\end{equation*}
	For $d \> 2$ and $x\in B(0, R)$, let  $A_k:=\{y\in\R^d:2^{-k-1}r\leq|y-x|<2^{-k}r\}$. 
	The   upper bound \eqref{measure upper bound condition} implies that $\mu(\{x\})=0$. Then $\cup_{k=0}^\infty A_k = B(x,r)\setminus \{x\}$.  Since $g_r(z)$ is decreasing in $|z|$,
	$$
		\int_{B(x,r)}g_{r}(x-y)\mu(dy)
		\< \sum_{k=0}^\8g_{r}(2^{-k-1}r)\mu(A_k)
		\< \sum_{k=0}^\8 g_{r}(2^{-k-1}r)C_R (2^{-k}r)^{\3_{\mu}}.
	$$
When $d=2$,
$$\sum_{k=0}^\8 g_{r}(2^{-k-1}r)C_R(2^{-k}r)^{\3_{\mu}}\<\sum_{k=0}^\8 C_2(\log(2^{k+1})+1)C_R (2^{-k}r)^{\3_{\mu}}\<Cr^{\3'}$$
for some $\3'\in(0,\3_{\mu})$. When $d \> 3$,
\begin{equation*}
	\sum_{k=0}^\8 g_{r}(2^{-k-1}r)C_R (2^{-k} r)^{\3_{\mu}}
	\<\sum_{k=0}^\8 C |2^{-k-1}r|^{2-d}C_R (2^{-k}r)^{\3_{\mu}}\<C_R r^{\3''}
\end{equation*}
for some $\3''=\3_{\mu}-(d-2)>0$. This proves the lemma.   
\end{proof}

\begin{lemma} \label{X exit time}
	Given $R>2$, there exist $c_{1}(R) > 0$ and $ \beta:= \3(R) > 0$
	 such that for any $r \in (0, 1)$, any $x_{0}\in B(0,R-1)$ and $x\in B(x_{0},r)$, we have
	\begin{equation*}
		\EX_x[\tau_{B(x_0,r)}(X^n)]\< c_1 r^\3 + o_n(1)
	\end{equation*}
	for $n \in \N^*$, and
	$		\EX_x[\tau_{B(x_0,r)}(X)]\< c_1 r^\3.$
 \end{lemma}
 
\begin{proof}
	Let $B = B(x_{0},r)$. Recall that $\eta_{i}^{n}(x)$ denotes the exponentially distributed waiting time of $X^{n}$ at the i-th step at the place $x\in\Rd$. We omit the variable $x\in\Rd$ of $\eta_{i}^{n}(x)$ if we do not care about it. Note that by the strong Markov property,
\begin{align*}
	\EX_x[\tau_{B}(X^n)]
	&= \EX_x[\tau_{B}(X^n); \eta_{0}^{n} + \eta_{1}^{n} \< \tau_{B}(X^n)] + \EX_x[\tau_{B}(X^n); \eta_{0}^{n} = \tau_{B}(X^n)] \\
	&= \EX_x[\eta_0^n + \eta_1^n + \tau_{B}(X^n)\circ \theta_{\eta_0^n + \eta_1^n}; W_{1}^{n} \in B] + \EX_x[\eta_{0}^{n};W_{1}^{n}\notin B]\\
	&= \EX_x[\eta_0^n(W_{0}^{n})] + \EX_x[\eta_1^n(W_{1}^{n}); W_{1}^{n} \in B] + \EX_x[\tau_{B}(X^n)\circ \theta_{\eta_0^n + \eta_1^n};W_{1}^{n} \in B].
\end{align*}

For the first two terms,
\begin{align*}
	\EX_x[\eta_0^n(W_{0}^{n})] + \EX_x[\eta_1^n(W_{1}^{n}); W_{1}^{n} \in B] 
	&\< 2 \sup_{x\in B} \frac{1}{\lambda_{n}(x)} 
 	= 2 \sup_{x\in B} r_n^2 \sigma_{\jd}^{2}\int \phi_n(y-x)\mu(dy)\\
	&\< 2\|\phi\|_\8 \sup_{x\in B(0,R)} r_n^{2-d}\sigma_{\jd}^{2}\mu(B(x,Cr_n))\\
	&\< 2C_{R} \|\phi\|_\8  \sigma_{\jd}^{2} r_n^{2-d+\3_{\mu}} = o_n(1),
\end{align*}
where the last equation is due to $\3_{\mu} > d- 2$.
For the third term, by \Cref{greenfunction} and \Cref{green est}
\begin{eqnarray*}
&& \EX_x[\tau_{B}(X^n) \circ \theta_{\eta_0^n + \eta_1^n};W_{1}^{n} \in B]
	=\int g^{n,B}(x,y)\mu_n(dy)\\
	& \leq  & \int g^{n,B(x,2r)}(x,y)\mu_n(dy)
	 \<\int g_{2r}(x-y)\mu_n(dy)\\
	& =& (g_{2r}*(\v\phi_n*\mu))(x) 	=(\v\phi_n*(g_{2r}*\mu))(x)\\
	& \leq  & (\v\phi_n*(C_Rr^\3))(x)=C_Rr^\3.
\end{eqnarray*}
For the process $X$,  we have
$	\EX_x[\tau_{B}(X)] \<\int g_{2r}(x-y)\mu(dy)\<C_Rr^\3.$
 This completes the proof.
\end{proof}

Given a bounded domain $D \inn B(0,R) \inn \Rd$, define
\[
U_\2^{n,D}f(x):=\EX_x\int_0^\8 e^{-\2 t}f(X^{n,D}_t)dt,
\qquad
P_t^{n,D} f(x):=\EX_xf(X^{n,D}_t),
\]
for $\2>0$, $t\geq0$ and any $f\>0$ on $D$, with the convention that functions on $D$ are extended to take value 0 at the cemetery point. Write $\|f\|_{p,n}:=(\int_D |f|^p\,d\mu_n)^{1/p}$ for $1\leq p<\infty$.

\begin{lemma} \label{ultracontractivity}
	There exist $\kappa > 0, n_{0}(R)>0, C_{R,\k} > 0$ such that for all $n\>n_{0}$ and any nonnegative $h \in L^{2}(D;\mu_{n})$ with $\|h\|_{1,n} \< 1$
\begin{equation*}
	\|P_{t}^{n,D} h\|_{2,n}^{2} \< C_{R,\k} \left( (t - O_{n} \log \|h\|_{2,n})_+ \right) ^{-1/\kappa} ,
\end{equation*}
where $O_n\geq0$ is independent of $h$ and $t$, and $O_{n} = O(r_{n}^{\3_{\mu}-d+2})$.
\end{lemma}

\begin{proof}
	Let $G^{n,D}f (x) := \EX_{x}\int_{0}^{\8}f(X_{t}^{n,D})dt$ be the Green operator. Then
	\begin{equation*}
		G^{n,D}1 (x) = \EX_{x}[\tau_{D}(X^{n})].
	\end{equation*}
	Like in the proof of \Cref{X exit time}, we can show that
	for any open set $U \inn D \inn B(0,R)$ and $p>1$
	\begin{align*}
	\|G^{n,U }1\|_\8 
	&\< O(r_{n}^{\3_{\mu}-d+2}) + \sup_{x\in U}\int_{U} g^{n,U}(x,y)\mu_n(dy)\\
	&\< O(r_{n}^{\3_{\mu}-d+2}) + \mu_n(U)^{\frac{p-1}{p}}\sup_{x\in B(0,R)}\left(\int_{U} g^{n,U}(x,y)^p \mu_n(dy)\right)^{1/p}\\
	&\< O(r_{n}^{\3_{\mu}-d+2}) + \mu_n(U)^{\frac{p-1}{p}}\sup_{x\in B(0,R)}\left(\int g_{2R}(x-y)^p \mu_n(dy)\right)^{1/p}\\
	&= O(r_{n}^{\3_{\mu}-d+2}) + \mu_n(U)^{\frac{p-1}{p}} \sup_{x\in B(0,R)}\left( \v\phi_n*(g_{2R}^p*\mu)(x)\right)^{1/p}.
\end{align*}
By the same method as in the proof of \Cref{greenfunction}, we can show there is a constant $C_R>0$ such that 
$$
\sup_{x\in B(0,R)}\left( \v\phi_n*(g_{2R}^p*\mu)(x)\right)^{1/p}\<C_{R},
$$
provided $p>1$ is chosen small enough so that $\3_{\mu}-p(d-2)>0$ in the case $d\>3$. Hence
\begin{equation*}
	\|G^{n,U}1\|_\8\< O(r_{n}^{\3_{\mu}-d+2})   + C_R \mu_n(U)^{\frac{p-1}{p}}.
\end{equation*}
By \cite[Lemma 3.2]{grigor2012two} we know the bottom of the spectrum $\Lambda_{\text{min}}$ of the nonnegative operator $-\mathcal A^{n,U}$ associated with $\mathcal E^{n,U}$ on $L^2(U;\mu_n)$ satisfies
\begin{equation}\label{min eigenvalue}
	\Lambda_{\text{min}}(U)
	\>\|G^{n,U}1\|_\8^{-1}
	\> (C_R \mu_n(U)^{\kappa} + O_{n})^{-1}.
\end{equation}
where $O_{n} = O(r_{n}^{\3_{\mu}-d+2}) $ and $\k = {\frac{p-1}{p}}$. We then follow the proof of \cite[Lemmas 5.4 and 5.5]{grigor2014upper}. For any $u\in \mathcal F^{n,D}\cap C_{c}(D)$ with $u\>0$,
the set $U_{s}:=\{x\in D: u>s\}$ is open for every $s>0$. By the Markov property, for $t>s>0$, $(u-t)_+$ vanishes outside $U_t\subset U_s$, and hence
\begin{equation*}
	\mathcal E^{n,D}(u)
	\>\mathcal E^{n,D}((u-t)_+)
	\>\Lambda_{\text{min}}(U_s)\int_{U_s}(u-t)_+^2\,\mu_n(dx).
\end{equation*}
Let $A=\|u\|_{1,n}$ and $B = \|u\|_{2,n}^{2}$. Since $u\>0$, we can use the inequality $(u-t)_{+}^{2}\> u^{2}-2tu$ to get 
\begin{equation} \label{formlowerBound2}
	\mathcal E^{n,D}(u) \>
	\Lambda_{\text{min}}(U_{s})(B-2tA).
\end{equation} 
On the other hand, we have 
\begin{equation} \label{measurebound}
	\mu_{n}\left(U_s\right) \leqslant \frac{1}{s} \int_{U_s} u d \mu_{n} \leqslant \frac{A}{s}.
\end{equation}
Combining \eqref{min eigenvalue}, \eqref{formlowerBound2} and \eqref{measurebound}, we get 
\begin{equation}
	\mathcal E^{n,D}(u) \> \frac{1}{C_R (A/s)^{\kappa} + O_{n}} \left(B-2tA\right).
\end{equation}
Letting $t\downarrow s = \frac{B}{4A}$, we get
\begin{equation}\label{formlowerBound}
	\mathcal E^{n,D}(u) \> \frac{1}{2}\frac{ \|u\|_{2,n}^{2}}{C_R (\|u\|_{1,n}^{2}/\|u\|_{2,n}^{2})^{\kappa} + O_{n}} .
\end{equation}
 As  $\mathcal E^{n,D}(|u|)\<\mathcal E^{n,D}(u)$, \eqref{formlowerBound} holds for any signed $u\in \mathcal F^{n,D}\cap {C_c}(D)$.
  Since $\mathcal F^{n,D}\cap {C_c} (D)$ is $\mathcal E^{n,D}_{1}$ dense (hence also $L^{1}$-dense by Cauchy-Schwarz inequality)
   in $\mathcal F^{n,D}$, a standard approximation argument shows that \eqref{formlowerBound} holds for any $u\in\mathcal F^{n,D}$.
	Set $u_{t} = P_{t}^{n,D}h$ for $h\>0$ with $\|h\|_{1,n} \< 1$, and denote $J(t) = \|u_{t}\|_{2,n}^{2}$. Using the fact that $P_{t}^{n,D}$ is contractive in $L^{1}(D;\mu_{n})$ (because $P^{n,D}_{t}$ is symmetric in $L^{2}(D;\mu_{n})$ and $P_{t}^{n,D}1\<1$), we have
\begin{equation*}
	\frac{dJ}{dt} = -2 \mathcal E^{n,D}(u_{t}) \< -   \frac{J}{C_R J^{-\kappa} + O_{n}} ,
\end{equation*}
or
\begin{equation*}
	C_R \frac{dJ}{J^{\k + 1}} + O_{n} \frac{dJ}{J} \< -  dt.
\end{equation*}
Integrating from $0$ to $t$ yields
\begin{equation*}
	- \frac{C_R}{\kappa} (J^{-\kappa} - \|h\|_{2,n}^{-2\kappa}) + O_{n} \log \frac{J}{\|h\|_{2,n}^{2}} \< - t.
\end{equation*}
Rearranging gives
\begin{equation*}
	t  - O_{n} \log \|h\|_{2,n}^{2} \< \frac{C_R}{\kappa} (J^{-\kappa} - \|h\|_{2,n}^{-2\kappa}) + O_{n}\log \frac{1}{J} 
	\< C_{R,\kappa} J^{-\kappa} , 
\end{equation*}
where we use the fact that $\log z \< c_{\k} z^{\k}$ for $z > 0$ in the last inequality and it holds for all sufficiently large $n$ since $\lim_{n\to\infty}O_n=0$.
This completes the proof.
\end{proof}

\begin{prop}\label{Holder}
There is $c_{0}>0$ such that for any $R>2, t>0,\2>0$, $D = B(0,R-1)$, there are $\r(R)>0$, $n_0(R,t)>0$, $ C_1(R,\2)>0$, $C_2(R,t)>0$ such that for any $n>n_0$ and bounded function $f$ supported on $D$ and $x,y\in B(0,R-2)$ with $|x-y|<1/4$, we have
	\begin{align*}
		|U_\2^{n,D} f(x)-U_\2^{n,D} f(y)| &\< C_1((|x-y|\vee c_{0}r_{n})^\r + o_{n}(1))  \|f\|_{\8}  , \\
		|P_t^{n,D} f(x)-P_t^{n,D} f(y)| &\<  C_{2}((|x-y|\vee c_{0}r_{n})^{\r} + o_{n}(1))\|f\|_{\8}.
	\end{align*}
Here the $o_n(1)$ terms tend to zero as $n\to\infty$, uniformly in $x,y,f$ for fixed $R,t,\2$.
\end{prop}

\begin{proof}
	To show the first inequality, we use a similar argument as that for
	 \cite[Proposition 2.4]{chen2015quenched}. We need to check the following properties:
	 for $R>2$ there are $c_1(R),\allowbreak c_2(R),\allowbreak \3(R),\allowbreak \r' \in (0, \8)$ such that both properties below hold for every $x_0\in D$.
	\begin{itemize}
	\item For all $x\in B(x_0,r)$ and $r\in(0,1)$,
	\begin{equation*}
		\EX_x[\tau_{B(x_0,r)}(X^n)]\< c_1 r^\3 + o_n(1).
	\end{equation*}
	\item There is $\r'>0$ such that if $h$ is bounded and harmonic with respect to $X^n$ in a ball $B(x_0,2r)$, then 
		$$|h(x)-h(y)|\< c_2 \left(\frac{|x-y|\vee c_{0}r_{n}}{r}\right)^{\r'} \|h\|_\8 \qquad \text{for }x,y\in B(x_0,r).$$
	\end{itemize}
The first property is implied by \Cref{X exit time}.
Since a continuous time change does not change harmonicity, the functions harmonic with respect to $X^n$ are the same as those harmonic with respect to the continuous-time random walk $\~W^n$ without time change. Thus the second property follows from \Cref{holder Harmonic functions}.

Now we use  the above two properties 
 to prove the stated resolvent estimate. For $x_0\in B(0,R-2)$ and $r\in(0,1)$, set $\tau^n_r:=\tau_{B(x_0,r)}(X^n) \< \tau_{D}(X^{n})$. By the strong Markov property, for $x\in B(x_{0},r/2)$
\begin{align*}
U_\2^{n,D} f(x) &=\mathbb{E}_x \int_0^{\tau_r^n} e^{-\2 t} f(X_t^{n,D}) d t+\mathbb{E}_x[(e^{-\2 \tau_r^n}-1) U_\2^{n,D} f(X_{\tau_r^n}^n)]+\mathbb{E}_x[U_\2^{n,D} f(X_{\tau_r^n}^n)] \\
&=: I_1+I_2+I_3.
\end{align*}
We have 
\begin{equation*}
	|I_1|\<\|f\|_\8 \EX_x \tau^n_r \< (c_1 r^\3 + o_n(1)) \|f\|_\8 .
\end{equation*}
By $\|U_\2^{n,D}f\|_\8 \< \2^{-1} \|f\|_\8$ and $|e^{-z}-1| \< z$ for $z\>0$,  we get
$$
|I_2|\<\2\EX_x \tau^n_r\|U_\2^{n,D}f\|_\8\< (c_1 r^\3 + o_n(1)) \|f\|_\8.
$$
Note that $I_3$ as a function of $x$ is bounded in $\R^d$ and harmonic in $B(x_0,r)$.
 Hence for $x,y\in B(x_0,r/2)$, 
the  second property above is applicable.
Combining all these together and applying $\|U_\2^{n,D}f\|_\8\<\2^{-1}\|f\|_\8$ again,  we get 
\begin{align*}
	|U_\2^{n,D} f(x)-U_\2^{n,D} f(y)|
	\<c_{3}\left(r^\3 + o_n(1) + \2^{-1}\left(\frac{|x-y|\vee c_{0}r_{n}}{r}\right)^{\r'}\right)\|f\|_\8.
\end{align*}
For any distinct $x,y\in B(0,R-2)$ and $|x-y|<1/4$, let $x_0=x$ and $r=(|x-y|{\vee c_{0}r_{n}})^{1/2}$, then $y\in B(x_0,r/2)$ and hence
\begin{align*}
	|U_\2^{n,D} f(x)-U_\2^{n,D} f(y)|
	&\<c_{3}\left(r^{\3}+o_n(1)+\2^{-1}r^{\r'}\right)\|f\|_\8\\
	&\leq(C_1(|x-y|\vee c_{0}r_{n})^{\r} + o_n(1))\|f\|_\8
\end{align*}
for $C_1=c_{3}(1+\2^{-1})$ and $\r=(\3\wedge\r')/2$.

To prove the stated semigroup estimate, we follow the idea in \cite[Proposition 3.4]{bass2010symmetric}. But since we do not have the ultracontractivity of the semigroup, extra work is needed. 
We note that the generator of the process $X^{n,D}$ is $\mathcal A^{n,D} = \lambda_{n}(Q_{n,D} - I)$, where $I$ is the identity map and $Q_{n,D}f = Q_{n}(f \1_{D})\1_{D}$. From now, we may let $\2 = 1$ and will not emphasize the dependency of constants on $\2$. For any $s>0$ denote $F_{s} := (\2 I - \mathcal A^{n,D})P_{s}^{n,D}f$. By the well-known relation between the resolvent and the generator that $U_\2^{n,D}(\2 I - \mathcal A^{n,D}) = I$, and the fact that $\mathcal A^{n,D}$ and $P_{s}^{n,D}$ commute, we have
 \begin{align*}
	P_{t}^{n,D}f 
	= U_\2^{n,D}(\2 I - \mathcal A^{n,D})P_{t}^{n,D}f 
	= U_\2^{n,D}P_{t/2}^{n,D}(\2 I - \mathcal A^{n,D})P_{t/2}^{n,D}f
	= U_\2^{n,D}P_{t/2}^{n,D}F_{t/2}.
\end{align*}
We want to show $\|P_{t/2}^{n,D}F_{t/2}\|_{\8}$ is bounded in $n$, then we can apply the preceding estimate for the resolvent $U_\2^{n,D}$ to get the desired result.
For $x\in D$, we have
\begin{align*}
	P_{t/2}^{n,D}F_{t/2}(x)
	&= \EX_{x}[F_{t/2}(X_{t/2}^{n,D}); t/2 \> \eta_{0}^{n}] + \EX_{x}[F_{t/2}(X_{t/2}^{n,D}); t/2 < \eta_{0}^{n}] \\
	&= \EX_{x}[\EX_{W_{1}^{n}}[F_{t/2}(X_{t/2-\eta_{0}^{n}}^{n,D})]; W_{1}^{n}\in D, t/2 \> \eta_{0}^{n}] + F_{t/2}(x)\P_{x}(t/2 < \eta_{0}^{n}) \\
	&=: I_{1}' + I_{2}'.
\end{align*}
For $I_{2}'$ we have
\begin{align*}
	|I_{2}'| 
	&= |F_{t/2}(x)|\P_{x}(t/2 < \eta_{0}^{n}) 
	 = e^{-\lambda_{n}(x)t/2} |(\2 I - \lambda_{n}(Q_{n,D} - I))P_{t/2}^{n,D}f|(x)\\
	&\< e^{-\lambda_{n}(x)t/2} (\2 + 2 \lambda_{n}(x))\|f\|_{\8}
	= o_{n}(1) \|f\|_{\8}.
\end{align*}
For $I_{1}'$ we have
\begin{align*}
	I_{1}'
	&= \EX_{x} \big[\EX_{W_{1}^{n}}[F_{t/2}(X_{t/2-\eta_{0}^{n}}^{n,D})]; W_{1}^{n}\in D, t/2 \> \eta_{0}^{n} \big]
	\\
	 & = \int_{0}^{t/2} Q_{n,D}P_{t/2-s}^{n,D}F_{t/2}(x) e^{-\lambda_{n}(x)s} \lambda_{n}(x)ds.
\end{align*}
We claim $\sup_{s\in(0,t/2)}\|Q_{n,D}P_{t/2-s}^{n,D}F_{t/2}\|_{\8}$ is bounded in $n$, then $I_{1}'$ is also bounded in $n$.
Using the spectral theorem for self-adjoint operators, there exist projection operators $E_{l} = E_{l}^{n,D}$ on the space $L^{2}(D;\mu_{n})$ such that
\begin{equation*}
	F_{s} = \int_{0}^{\8} (\2 + l) e^{-ls} d E_{l}(f).
\end{equation*}
Given $t_{0}>0$, for every $s\>t_{0}$, the bound $(\2 + l) e^{-ls} \< C_{t_{0}}$ gives
\begin{equation*}
	\|F_{s}\|_{2,n}^{2} =  \int_{0}^{\8} (\2 + l)^{2} e^{-2l s} d \langle E_{l}(f), E_{l}(f) \rangle_{n} \< C_{t_{0}} \|f\|_{2,n}^{2} ,
\end{equation*}
where $\langle \., \. \rangle_{n}$ denotes the inner product in $L^{2}(D;\mu_{n})$.
We see $F_{s} \in L^{2}(D;\mu_{n})$.
For a function $h\in L^{2}(D;\mu_{n}) $, we have
\begin{align*}
	|\langle F_{s}, h\rangle_{n}| & = \left|\int_0^{\infty}(\2 +l) e^{-l s} d\langle E_l(f), h\rangle_{n} \right|\\
	& \leq\left(\int_0^{\infty}(\2 +l) e^{-l s} d\langle E_l(f), f\rangle_{n}\right)^{1 / 2}\left(\int_0^{\infty}(\2 +l) e^{-l s} d\langle E_l(h), h\rangle_{n}\right)^{1 / 2} \\
	& \leq C_{t_{0}}\left(\int_0^{\infty} d\langle E_l(f), f\rangle_{n}\right)^{1 / 2}\left(\int_0^{\infty} e^{-l s / 2} d\langle E_l(h), h\rangle_{n}\right)^{1 / 2} \\
	& =C_{t_{0}} \|f\|_{2,n}\|P_{s / 4}^{n,D} h\|_{2,n}.
\end{align*}
For $x\in D$, by applying \Cref{ultracontractivity} with $h_{n,x}(y) = \jd_{n}(y-x) \frac{d\leb}{d\mu_{n}}(y)\1_{D}(y)$, we get
\begin{align*}
	|(Q_{n,D}F_{s})(x)| & = |\langle F_{s}, h_{n,x}\rangle_{n}|
	 \< C_{t_{0}} \|f\|_{2,n}\|P_{s / 4}^{n,D} h_{n,x}\|_{2,n} \\
	&\< C_{t_{0}} \|f\|_{2,n} \left( (C_{R}s/4 - O_{n} \log \|h_{n,x}\|_{2,n})_+ \right) ^{-1/(2\kappa)} .
\end{align*}
Note that
\begin{align*}
	\|h_{n,x}\|_{2,n}^{2} 
	&= \int_{D} \jd_{n}(y-x)^{2} \frac{d\leb}{d\mu_{n}}(y)^{2}\mu_{n}(dy) 
	\, = \, \int_{D} \frac{\jd_{n}(y-x)^{2}}{\int \phi_{n}(z-y)\mu(dz)} dy\\
	&\< C_{\jd}^{2} r_{n}^{-d} \int_{D} \frac{1}{\int \phi((z-y)/r_{n})\mu(dz)} dy.
\end{align*}
By \ref{condA}, $O_{n}\log \|h_{n,x}\|_{2,n}$ converges to $0$ as $n\to\infty$, uniformly in $x\in D$.
This shows there exists $n_{0}= n_{0}(t_{0},R)$ such that  
 $\|Q_{n,D}F_{s}\|_{\8,D} \< C_{t_{0},R} \|f\|_{2,n}$ for any $n> n_{0}$. 
 Now take $t_{0} = t/2$, we have
\begin{equation*}
	\sup_{s\in(0,t/2)}\|Q_{n,D}P_{t/2-s}^{n,D}F_{t/2}\|_{\8} 
	= \sup_{s\in(0,t/2)}\|Q_{n,D}F_{t-s}\|_{\8} 
	\<  C_{t,R} \|f\|_{2,n}.
\end{equation*}
Hence
\begin{equation*}
	\|P_{t/2}^{n,D}F_{t/2}\|_{\8} \< C_{t,R} \|f\|_{2,n} + o_{n}(1) \|f\|_{\8} \< C_{t,R} \|f\|_{\8}
\end{equation*}
where in the second inequality we use $\|f\|_{2,n} \< \mu_{n}(D)^{1/2} \|f\|_{\8} \< C_{R} \|f\|_{\8}$.
Finally, applying the preceding estimate for the resolvent $U_\2^{n,D}$ gives
\begin{align*}
	|P_t^{n,D} f(x)-P_t^{n,D} f(y)|
	&= |U_\2^{n,D}P_{t/2}^{n,D}F_{t/2}(x) - U_\2^{n,D}P_{t/2}^{n,D}F_{t/2}(y)| \\
	&\< C_{1}((|x-y|\vee c_{0}r_{n})^{\r} + o_{n}(1))\|P_{t/2}^{n,D}F_{t/2}\|_{\8} \\
	&\< C_{2}((|x-y|\vee c_{0}r_{n})^{\r} + o_{n}(1)) \|f\|_{\8}.
\end{align*}
This completes the proof.
\end{proof}

We can show a similar H\"older regularity property for the resolvent $\{U_{\2}^{B};\2>0\}$ and semigroup $\{P^B_t;t\>0\}$ associated to the killed process $X^B$, where $B$ is any ball in $\R^d$. We state the result and sketch the proof here.   \Cref{Holder} and \Cref{killed Holder} will be used in
the proof of  \Cref{pointwise}.

\begin{prop}\label{killed Holder}
	Given any compact set $K\inn B$, $\2>0$ and $t>0$, let $\d_K$ be the distance between $B^c$ and $K$.
	There
	 exist $\r(B)>0$, $C_1$ (depending on $K,\2$) and $C_2$ (depending on $K,t$) such that for every bounded function $f\in L^2(B;\mu)$ and any $x,y\in K$ with $|x-y|<\d_K^2 \wedge (1/4) $, we have
	$$|U_{\2}^{B} f(x)- U_{\2}^{B} f(y)|\<C_1|x-y|^\r\|f\|_{\8}, $$
	$$|P^B_t f(x)-P^B_t f(y)|\<C_2|x-y|^\r\|f\|_{2}, $$
where $\|f\|_{\8} :=\sup_{x\in B}|f(x)|$ and $\|f\|_{2}^2 :=\int_B f^2(x)\mu(dx)$.
\end{prop} 
\begin{proof}
	For a stopping time $\tau$ that is less than $\tau_B$,
	by the strong Markov property we have for any bounded function $f\in L^2(B;\mu)$
	\begin{align*}
	U_{\2}^{B} f(x)
	&=\EX_x\int_0^{\tau_B} e^{-\2 t}f(X_t)dt\\	
	&=\EX_x\int_0^{\tau} e^{-\2 t}f(X_t)dt+\EX_x\left[\EX_{X_{\tau}}\int_{0}^{\tau_B} e^{-\2 (t+\tau)}f(X_t)dt\right]\\	
	&=\EX_x\int_0^{\tau} e^{-\2 t}f(X_t)dt+\EX_x\left[(e^{-\2\tau}-1)U^B_\2 f(X_{\tau})\right]+\EX_x\left[U^B_\2 f(X_{\tau})\right].
	\end{align*}	
For any distinct $x,y\in K$ with $|x-y|<\d_K^2 \wedge (1/4) $, set $\t=\t_{B(x,r)}$ where $r=|x-y|^{1/2}$ (this guarantees that $\t<\t_B$). Then using \Cref{X exit time} we can apply the same estimate as in \Cref{Holder} to get the locally H\"older property of $U^B_\2$.
To prove H\"older property of the semigroup $P_{t}^{B}$,
	let $G^{B}f (x) := \EX_{x}\int_{0}^{\8}f(X_{t}^{B})dt$ be the Green operator and let 
	$R>2$ so that $B\inn B(0,R)$. Then for any open set $U\inn B$, we have
	\begin{align*}
		\|G^{U }1\|_\8 
		&\< \sup_{x\in U}\int_{U} g_{2R}(x-y)\mu(dy)
	 		\< \mu(U)^{\frac{p-1}{p}}\sup_{x\in B(0,R)} \left(\int_{B(0,R)} g_{2R}(x-y)^p \mu(dy)\right)^{1/p}\\
		&\le \mu(U)^{\frac{p-1}{p}}\sup_{x\in B(0,R)}\left( g_{2R}^p*\mu(x)\right)^{1/p}.
	\end{align*}
	Using the same estimates as in the proof of \Cref{greenfunction},
	 we can show 
	 $$\sup_{x\in B(0,R)}\left( g_{2R}^p*\mu(x)\right)^{1/p}\<C_{R}  <\infty
	 $$ 
	 provided $p>1$ is chosen small enough so that $\3_{\mu}-p(d-2)>0$ in the case $d\>3$. Hence  $\|G^{U}1\|_\8\<C\mu(U)^{\frac{p-1}{p}}$. By \cite[Lemma 3.2]{grigor2012two} we know the bottom of the spectrum $\Lambda_{\text{min}}$ of the nonnegative operator $-\mathcal A^{U}$ associated with $\mathcal E^{\mu,U}$ on $L^2(U;\mu)$ satisfies
	\begin{equation*}
		\Lambda_{\text{min}}(U)\>\|G^{U}1\|_\8^{-1}\> C \mu(U)^{-\frac{p-1}{p}}.
	\end{equation*}
	Then \cite[Lemma 5.5]{grigor2014upper} gives $\|P_{t}^{B}\|_{L^{1}\to L^{\8}}< C_{B,t}$. Hence the classical method in \cite[Proposition 3.4]{bass2010symmetric} applies and gives the H\"older continuity of $P^B_t$.
\end{proof}

\subsection[Convergence in the Skorokhod topology under individual starting points]{Convergence in the Skorokhod topology\protect\\ under individual starting points}

We first use Green function estimate \Cref{green est} to show that $\mu$ is a smooth measure.

\begin{prop}\label{smooth measure}
	The measure $\mu$ under condition \ref{condA} is a smooth measure in the strict sense. In particular, $\mu$ does not charge polar sets.
\end{prop}
\begin{proof}
Let $G_1^W$ be the Brownian $1$-resolvent kernel and set $\nu_R:=\1_{B(0,R)}\mu$. 
It is known that  $G_1^W(z)\leq C(1+g_1(z))$, where $g_1$ is the function $g_R$ with $R=1$ in Proposition \ref{green est}.
  Thus the annulus estimate in the proof of \Cref{greenfunction}, together with \eqref{measure upper bound condition}, gives
\[
\sup_{x\in\Rd}\int G_1^W(x-y)\nu_R(dy)<\infty\qquad\text{for every }R>0.
\]
Since $\nu_R$ is finite, this bound implies that it has finite energy integral by \cite[Exercise 4.2.2]{FOT}, and hence belongs to the class $S_{00}$ of finite measures of finite energy integral with bounded $1$-potential (\cite[(2.3.10)]{CF12}). The balls $B(0,R)$ increase to $\Rd$, so $\mu$ is smooth in the strict sense by the definition preceding \cite[Theorem 4.1.11]{CF12}. In particular, $\mu$ charges no polar sets by \cite[Theorem 2.3.7]{CF12}.
\end{proof}

With the above proposition, by \cite[Theorem A.3.9]{CF12} one may construct a strong Markov process $(X, \P_x)_{x\in \R^d}$ called the time-changed Brownian motion by $\mu$.

We first record a consequence of strong convergence that will be used to pass from absolutely continuous initial distributions to individual starting points.

\begin{lemma}\label{pointwise upgrade}
	Let $\{x_n;n\geq1\}\subset D$ converge to $x_0\in D$ as $n\to\infty$, let $\{u_n;n\geq1\}$ be a sequence of bounded functions on $D$ with $\sup_n\|u_n\|_\8<\8$, and let $u$ be bounded and continuous on $D$. Suppose that $u_n\in L^2(D;\mu_n)$ converges strongly to $u\in L^2(D;\mu)$ as $n\to\infty$ in the sense of \Cref{strong convergence} and, for some compact neighborhood $K\inn D$ of $x_0$,
	\[
	\lim_{\delta\downarrow0}\limsup_{n\to\8}
	\sup_{\substack{x,y\in K\\ |x-y|\leq\delta}}
	|u_n(x)-u_n(y)|=0.
	\]
Then $\lim_{n\to \infty} u_n(x_n) =  u(x_0)$. 
\end{lemma}

\begin{proof}
	Fix $\e>0$ such that $\overline{B(x_0,\e)}\inn D$. Since $\mu$ has full quasi-support, $\mu(B(x_0,\e))>0$. We may therefore choose $g_\e\in C_c(D)$ supported in $B(x_0,\e)$ such that $g_\e\geq0$ and $\int g_\e\,d\mu=1$. Then $\lim_{n\to\infty}\int g_\e\,d\mu_n=1$, and strong convergence gives
\[
\lim_{n\to\infty}\langle g_\e,u_n\rangle_{\mu_n}=\langle g_\e,u\rangle_\mu.
\]
 	Writing $a_n:=\int g_\e\,d\mu_n$, we have the decomposition
	\begin{align*}
	u_n(x_n)-u(x_0)
	&=(1-a_n)u_n(x_n)
	 +\int g_\e(y)(u_n(x_n)-u_n(y))\,\mu_n(dy)\\
	&\quad+\langle g_\e,u_n\rangle_{\mu_n}-\langle g_\e,u\rangle_\mu
	 +\int g_\e(y)(u(y)-u(x_0))\,\mu(dy).
	\end{align*}
	Consequently,
	\begin{align*}
	|u_n(x_n)-u(x_0)|
	&\leq |1-a_n|\,|u_n(x_n)|
	 +\int g_\e(y)|u_n(x_n)-u_n(y)|\,\mu_n(dy)\\
	&\quad+|\langle g_\e,u_n\rangle_{\mu_n}-\langle g_\e,u\rangle_\mu|
	 +\int g_\e(y)|u(y)-u(x_0)|\,\mu(dy).
	\end{align*}
	Let $K$ be as in the hypothesis. For all sufficiently small $\e$ and, for each such $\e$, all sufficiently large $n$, both $x_n$ and the support of $g_\e$ are contained in $K$, and $|x_n-y|\leq2\e$ on the support of $g_\e$. Hence
	\[
	\limsup_{n\to\8}|u_n(x_n)-u(x_0)|
	\leq
	\limsup_{n\to\8}\sup_{\substack{x,y\in K\\|x-y|\leq2\e}}
	|u_n(x)-u_n(y)|
	+\sup_{y\in B(x_0,\e)}|u(y)-u(x_0)|.
	\]
	Both terms tend to zero as $\e\downarrow0$, by the assumed asymptotic equicontinuity and the continuity of $u$, respectively. This proves the result.
\end{proof}

\begin{prop}\label{pointwise}
	Let $D=B(0,R)$ with $R>2$ and $\mu(\6 D)=0$, and let $\{x_n;n\geq1\}\subset\Rd$ converge to $x_0\in B(0,R-2)$ as $n\to\infty$. Then the law of $\tau_R(X^n)$ under $\P_{x_n}$ converges weakly to that of $\tau_R(X)$ under $\P_{x_0}$ as $n\to \infty$. Moreover,
	\[
	Y^{n,R}:=X^n_{\cdot\wedge\tau_R(X^n)}
	\quad\hbox{under }\P_{x_n}
	\]
	converges in law to
	\[
	Y^R:=X_{\cdot\wedge\tau_R(X)}
	\quad\hbox{under }\P_{x_0}
	\]
	in $\DD([0,\8);\Rd)$ under the uniform topology on $[0,\8)$ as $n\to \infty$.
\end{prop}

\begin{proof}
	For every $t>0$, set $s_n:=P_t^{n,D}\1_D$ and $s:=P_t^D\1_D$. By \Cref{localized observables}, $s_n\in L^2(D;\mu_n)$ converges strongly to $s\in L^2(D;\mu)$ as $n\to\infty$. We apply \Cref{Holder} with radius parameter $R+1$, so that its domain $B(0,(R+1)-1)$ equals $D=B(0,R)$. Its estimates give the asymptotic equicontinuity near $x_0$ required in \Cref{pointwise upgrade}, while \Cref{killed Holder} gives the continuity of $s$. Hence
	\[
	\lim_{n\to\infty}\P_{x_n}(\tau_R(X^n)>t)=\lim_{n\to\infty}s_n(x_n)
	= s(x_0)=\P_{x_0}(\tau_R(X)>t).
	\]
	Thus $\tau_R(X^n)$ converges in law to $\tau_R(X)$ as $n\to\infty$.

	Fix $0<t_1<\cdots<t_k$ and $f_1,\ldots,f_k\in C_b(\Rd)$, and define
	\[
	F_n(x):=\EX_x\Big[\prod_{i=1}^k f_i(Y^{n,R}_{t_i})\Big],
	\qquad
	F(x):=\EX_x\Big[\prod_{i=1}^k f_i(Y^R_{t_i})\Big].
	\]
	By \Cref{localized observables}, $F_n|_D\in\mathcal H_{r_n}^D$ converges strongly to $F|_D\in\mathcal H^D$ as $n\to\infty$. With the notation from its proof, set
	$\varphi:=\prod_{i=1}^k f_i$, $h_n:=\psi_1^{r_n}-H_{r_n}\varphi$, and $h:=\psi_1-H\varphi$.
	The backward induction there shows that $\psi_1^{r_n}|_D\in\mathcal H_{r_n}^D$ converges strongly to $\psi_1|_D\in\mathcal H^D$ as $n\to\infty$, while \Cref{hitting converge} and \Cref{boundary restriction}(ii) show that $H_{r_n}\varphi|_D\in\mathcal H_{r_n}^D$ converges strongly to $H\varphi|_D\in\mathcal H^D$ as $n\to\infty$. Hence $h_n|_D\in\mathcal H_{r_n}^D$ converges strongly to $h|_D\in\mathcal H^D$ as $n\to\infty$. Since $\psi_1^{r_n}=\varphi=\psi_1$ on $D^c$, \eqref{stopped semigroup} gives
$F_n=H_{r_n}\varphi+P_{t_1}^{n,D}h_n,$ 	$F=H\varphi+P_{t_1}^Dh,$
 	and
	$	\sup_n\|h_n\|_\8\leq2\prod_{i=1}^k\|f_i\|_\8.$
 	Hence, for a compact neighborhood $K\inn B(0,R-2)$ of $x_0$, \Cref{Holder}, applied to $D=B(0,R)$ as above, gives
	that for $x,y\in K$ with $ |x-y|<1/4$, 
	$$
	|F_n(x)-F_n(y)| 	 \leq |H_{r_n}\varphi(x)-H_{r_n}\varphi(y)| +C_K\big((|x-y|\vee c_0r_n)^\r+o_n(1)\big). 
	 $$
	  	By \Cref{hitting converge}, $H_{r_n}\varphi$ converges to $ H\varphi$ uniformly on $\overline D$ as $n\to\infty$, while $H\varphi$ is continuous. Therefore
	\[
	\lim_{\delta\downarrow0}\limsup_{n\to\8}
	\sup_{\substack{x,y\in K\\|x-y|\leq\delta}}
	|F_n(x)-F_n(y)|=0.
	\]
	By \Cref{hitting converge}, $H\varphi$ is continuous on $\overline D$, while \Cref{killed Holder} shows that $P_{t_1}^Dh$ is continuous on $D$. Thus the representation $F=H\varphi+P_{t_1}^Dh$ shows that $F$ is continuous on $D$. Applying \Cref{pointwise upgrade}, we obtain
$\lim_{n\to \infty} 	F_n(x_n)=  F(x_0)$.
 	Therefore the finite-dimensional distributions of $Y^{n,R}$ under $\P_{x_n}$ converge weakly to those of $Y^R$ under $\P_{x_0}$ as $n\to\infty$.

	The estimate in \Cref{uint} is uniform for starting points in compact subsets of $D$. Hence the same application of Aldous's criterion and the same truncation argument as in \Cref{stopped} give the asserted convergence in the uniform topology on $[0,\8)$.
\end{proof}

Finally, we are ready to prove \Cref{mainthm2}.

\begin{proof}[Proof of \Cref{mainthm2}]
	Let $\{r_n; n\geq 1\}$ be any decreasing sequence of positive numbers  that decreases to $0$   and set $x_n:=x_{r_n}$. For every sufficiently large $R$ in the unbounded increasing sequence $\{R_k\}$ of positive numbers from \Cref{R sequence}, \Cref{pointwise} gives the weak convergence of the stopped processes and of their exit times under $\P_{x_n}$. If $d\leq2$, the processes are conservative by Remark \ref{rm2}(iv), so the first assertion of \Cref{localization lemma} applies. If $d\geq3$, its second assertion applies. Thus $X^n$ under $\P_{x_n}$ converges in law to $X$ under $\P_{x_0}$ in $\DD([0,\8);\Rd_\partial)$ as $n\to\infty$. Since
this holds for any decreasing sequence $\{r_n; n\geq 1\}$  that converges to $0$, the desired result follows.
\end{proof}

\bigskip

\noindent {\bf Acknowledgement.} 
 Part of this work was carried out while both of the authors were 
in residence at the Mathematical Sciences Research Institute in Berkeley, California, during the spring semester of 2022
for {\it The Analysis and Geometry of Random Spaces} program, 
which is supported in part by the National Science Foundation
under Grant No. DMS-1928930.

\par\medskip

While the paper is finalized in October 2026, the main results of this paper have been reported at several conferences and workshops since 2022, including {\it The Analysis and Geometry of Random Spaces} workshop that took place March 28 - April 1, 2022, at MSRI (now SLMath), the workshop on  {\it Geometric and Stochastic Analysis on Metric Spaces} held at   Kyoto University from March 13 to March 14, 2023, 
and at the probability seminar at Technische Universit\"at Berlin on December 19, 2024.

\vskip 0.3truein 

{\bf Zhen-Qing Chen}

Department of Mathematics, University of Washington, Seattle, WA 98195, USA

Email: zqchen@uw.edu

\bigskip

{\bf Yang Yu}

Department of Mathematics, University of Washington, Seattle, WA 98195, USA

Email: yuy10@uw.edu

\end{document}